\documentclass{article}
\usepackage{amsfonts, amsmath, amssymb, amscd, amsthm, graphicx,enumerate}
\usepackage{hyperref,xypic}
\usepackage[usenames]{color}
\usepackage{tikz}
\usepackage{mathrsfs}

\usepackage{thmtools} 

\usepackage{thm-restate}

\makeatletter
\def\blfootnote{\xdef\@thefnmark{}\@footnotetext}
\makeatother

\newtheorem{thm}{Theorem}[section]
\newtheorem{cor}[thm]{Corollary}
\newtheorem{lem}[thm]{Lemma}

\newtheorem{conv}[thm]{Convention}
\newtheorem{fact}[thm]{Fact}

\theoremstyle{definition}
\newtheorem{defn}[thm]{Definition}

\theoremstyle{remark}
\newtheorem{rem}[thm]{Remark}

\newtheorem{const}[thm]{Construction}

\newfont{\eufm}{eufm10}

\renewcommand{\phi}{\varphi}

\newcommand{\R}{\mathbb R}
\newcommand{\N}{\mathbb N}
\newcommand{\Z}{\mathbb Z}
\renewcommand{\H}{\mathbb{H}}
\newcommand{\MCG}{\operatorname{MCG}}

\newcommand{\NS}{\mathcal{NS}}
\newcommand{\genus}{\operatorname{genus}}
\newcommand{\B}{\mathcal{B}}

\newcommand{\GL}{\mathcal{GL}}
\newcommand{\EL}{\mathcal{EL}}
\newcommand{\ELW}{\mathcal{ELW}}
\newcommand{\ML}{\mathcal{ML}}

\newcommand{\toCH}{\substack{CH \\ \to}}
\newcommand{\toH}{\substack{H \\ \to}}

\begin{document}

\title{Geometry of the graphs of nonseparating curves: covers and boundaries}
\author{Alexander J. Rasmussen}

\date{\today}
\maketitle

\begin{abstract}
We investigate the geometry of the graphs of nonseparating curves for surfaces of finite positive genus with potentially infinitely many punctures. This graph has infinite diameter and is known to be Gromov hyperbolic by work of the author. We study finite covers between such surfaces and show that lifts of nonseparating curves to the nonseparating curve graph of the cover span quasiconvex subgraphs which are infinite diameter and not coarsely equal to the nonseparating curve graph of the cover. In the finite type case, we also reprove a theorem of Hamenst\"{a}dt identifying the Gromov boundary with the space of ending laminations on full genus subsurfaces. We introduce several tools based around the analysis of bicorn curves and laminations which may be of independent interest for studying the geometry of nonseparating curve graphs of infinite type surfaces and their boundaries.
\end{abstract}


\section{Introduction}

In this paper we use the technique of \textit{bicorn curves} to investigate the geometry of the graphs of nonseparating curves. Let $S$ be an orientable surface of finite positive genus with either finitely many or infinitely many punctures. The \textit{nonseparating curve graph} $\NS(S)$ is an infinite diameter graph (see \cite{av} and \cite{dfv}) acted on by the mapping class group $\MCG(S)$ (the group of homeomorphisms of $S$ considered up to isotopy). Its vertices are the isotopy classes of nonseparating simple closed curves on $S$. Two vertices are joined by an edge when the isotopy classes have disjoint representatives (when $\genus(S)>1$) or when the isotopy classes have representatives that intersect at most once (when $\genus(S)=1$). We proved in \cite{nonsepbicorns} that $\NS(S)$ is $\delta$-hyperbolic, where the constant $\delta$ is independent of the particular surface $S$. This extended results of Hamenst\"{a}dt (\cite{nonsepmulti}) and Masur-Schleimer (\cite{disk}) on hyperbolicity (with hyperbolicity constant depending on the surface) in the case that $S$ has \textit{finitely many} punctures.

We extend the methods of \cite{nonsepbicorns} in this paper to prove the following theorems:

\begin{thm}
\label{coverthm}
Let $k\in \N$. There exists $F(k)>0$ with the following property. Let $\tilde{S}\to S$ be a degree $k$ cover where $S$ and $\tilde{S}$ are orientable and have finite positive genus and either finitely many or infinitely many punctures. Then the full subgraph of $\NS(\tilde{S})$ spanned by lifts of nonseparating curves on $S$ is $F(k)$-quasiconvex. Moreover, this subgraph is infinite diameter and not coarsely equal to $\NS(\tilde{S})$ itself.
\end{thm}

The second result was first proved by Hamenst\"{a}dt using different methods:

\begin{restatable}[Hamenst\"{a}dt, \cite{nonsepmulti} Corollary 3.10]{thm}{boundarythm}
\label{boundarythm}
Let $S$ be a \textit{finite type} orientable surface of positive genus. Then the Gromov boundary of $\NS(S)$ is $\MCG(S)$-equivariantly homeomorphic to the space of minimal laminations filling full genus subsurfaces of $S$, equipped with the coarse Hausdorff topology.
\end{restatable}

Here $V$ is \textit{full genus} if $\genus(V)=\genus(S)$. Our methods also produce an explicit family of coarse paths converging to any given point on the Gromov boundary. These coarse paths are known as ``infinite bicorn paths.'' Our methods are strong enough to prove:

\begin{cor}
Let $S$ be as in Theorem \ref{boundarythm} and let $\{c_n\}_{n=1}^\infty$ be a sequence of vertices of $\NS(S)$. Then $\{c_n\}$ converges to a point on the boundary $\partial \NS(S)$ if and only if $c_n$ coarse Hausdorff converges to a minimal lamination $L$ filling a full genus subsurface of $S$. In this case the point of $\partial \NS(S)$ represented by $\{c_n\}$ is identified with $L$ via the homeomorphism described in Theorem \ref{boundarythm}.
\end{cor}

In order to prove these results we develop new tools for studying the geometry of $\NS(S)$ using bicorns. The main tool is Corollary \ref{2kcornshausclose} which states that nonseparating $2k$-corns are uniformly Hausdorff close to geodesics. Here is a slightly less general version of the statement:

\begin{thm}
There exists a function $E:\N \to \R$ with the following property. Let $S$ be an orientable surface with $0<\genus(S)<\infty$. Let $a$ and $b$ be vertices of $\NS(S)$. Then the subgraph of $\NS(S)$ spanned by nonseparating $2l$-corns between $a$ and $b$ with $l\leq k$ is $E(k)$-Hausdorff close to any geodesic $[a,b]$. The function $E$ grows at most quadratically.
\end{thm}

\noindent In addition, we prove results about bicorns between nonseparating curves and laminations. We also analyze what happens if two nonseparating curves have a bicorn which is separating but does not bound a punctured disk. We hope that these tools will be of independent interest.

\subsection{Big mapping class groups}

The graphs $\NS(S)$ are particularly interesting when $S$ is infinite type, since one may attempt to use them to study the big mapping class group $\MCG(S)$. In addition to the \textit{loop graphs} $L(S)$ for surfaces $S$ containing an isolated puncture, the graphs $\NS(S)$ are among the few infinite diameter hyperbolic graphs associated to infinite type surfaces which are currently known. The graphs $\NS(S)$ are also useful since the surface $S$ is not required to have an isolated puncture for the definition.

However, a main difficulty for carrying out this program is that the Gromov boundary $\partial \NS(S)$ is not currently understood when $S$ is infinite type. In the case that $S$ is an infinite type surface with an isolated puncture, the Gromov boundary of the loop graph of $S$ has been described by Bavard-Walker in \cite{boundary} and \cite{simultaneous}. This description has proven very useful in understanding the corresponding big mapping class groups $\MCG(S)$: see \cite{boundary}, \cite{simultaneous}, and \cite{wwpd}.

Some of the arguments used in this paper to characterize the boundary of $\NS(S)$ when $S$ is \textit{finite type} would generalize easily to the case when $S$ is \textit{infinite type}. One would like to describe $\partial \NS(S)$ as a space of geodesic laminations when $S$ is infinite type. The main obstacle to generalizing the characterization in this paper of $\partial \NS(S)$ in the finite type case is that there is no ``structure theorem'' for geodesic laminations on infinite type surfaces. Whereas on a finite type surface every geodesic lamination is a finite union of minimal sublaminations and isolated leaves, in the infinite type case there is no such general description. We hope that after understanding better the structure of geodesic laminations on infinite type surfaces, the arguments in this paper might be of use in characterizing $\partial \NS(S)$ when $S$ is infinite type.

\subsection{Organization}

In Section \ref{background} we recall necessary background on hyperbolic metric spaces and surfaces.

In Sections \ref{2kcornsec} through \ref{lambicornsec} we develop tools that will be used in the proofs of our two main theorems. Namely in Section \ref{2kcornsec} we analyze $2k$-corns and show that nonseparating $2k$-corns are close to nonseparating bicorns. These in turn are Hausdorff close to geodesics. In Section \ref{nontrivsepsec} we analyze bicorns which are separating but do not bound punctured disks, showing that they are disjoint from nonseparating bicorns. In Section \ref{lambicornsec} we analyze bicorns between the leaves of a lamination and a nonseparating curve and their properties in homology.

In Section \ref{coversec} we prove Theorem \ref{coverthm}.

In Sections \ref{infpathsec} through \ref{surjsec} we prove Theorem \ref{boundarythm}.

The proofs of Theorems \ref{coverthm} (in Section \ref{coversec}) and \ref{boundarythm} (in Sections \ref{infpathsec} through \ref{surjsec}) may be read independently of each other. However, both proofs rely on the tools developed in Sections \ref{2kcornsec} through \ref{lambicornsec}.

\bigskip

\noindent \textbf{Acknowledgements.} I would like to thank Yair Minsky and Javier Aramayona for valuable conversations related to this work. I was partially supported by NSF Grant DMS-1610827.

\section{Background}
\label{background}

\subsection{Hyperbolic metric spaces}

Let $X$ be a geodesic metric space. We will always use the notation $d(\cdot, \cdot)$ to denote the distance function implicitly associated to $X$ . In this paper we will consider graphs as geodesic metric spaces by identifying their edges with the unit interval $[0,1]$. Given $x,y\in X$ we will frequently denote by $[x,y]$ a geodesic in $X$ from $x$ to $y$, even though it may not be uniquely determined by its endpoints. If $A\subset X$ then we denote by $B_r(A)$ the closed $r$-neighborhood of $A$ in $X$. If $x\in X$ then we will denote by $d(x,A)$ the distance from $x$ to $A$ (i.e. $d(x,A)=\inf_{a\in A} d(x,a)$).

\begin{defn}
We say that $X$ is \textit{$\delta$-hyperbolic} if given $x,y,z\in X$, we have $[x,z]\subset B_\delta([x,y]\cup [y,z])$.
\end{defn}

Given $a,x,y\in X$ we denote by $(x\cdot y)_a$ the \textit{Gromov product} defined by \[(x\cdot y)_a=\frac{1}{2}(d(a,x)+d(a,y) -d(x,y)).\]

\begin{lem}
Let $X$ be $\delta$-hyperbolic. Then we have that \[|(x\cdot y)_a-d(a,[x,y])|\leq 2\delta.\]
\end{lem}

\begin{lem}
Let $X$ be $\delta$-hyperbolic. Then, given $x,y,z,a \in X$ we have \[(x\cdot z)_a \geq \min \{ (x\cdot y)_a, (y\cdot z)_a\} -5\delta.\]
\end{lem}

These lemmas will frequently be used implicitly without mentioning them by name.

We now recall the notion of the \textit{Gromov boundary} of the hyperbolic metric space $X$. For this, fix a point $a\in X$. We say that the sequence $\{x_n\}_{n=1}^\infty\subset X$ \textit{converges at infinity} if $\lim_{m,n\to \infty} (x_n\cdot x_m)_a= \infty$. This notion is independent of the base point $a$. There is a natural equivalence relation on such sequences defined as follows. If $\{x_n\}_{n=1}^\infty$ and $\{y_n\}_{n=1}^\infty$ are sequences which converge at infinity then we consider them to be equivalent if $\lim_{m,n\to \infty} (x_n \cdot y_m)_a=\infty$. The Gromov boundary $\partial X$ consists of the equivalence classes of sequences which converge at infinity. If $\{x_n\}_{n=1}^\infty$ converges at infinity then we will denote by $[\{x_n\}]$ the equivalence class of $\{x_n\}$ in $\partial X$. The set $\partial X$ carries a natural topology defined as follows. Given $\xi,\eta\in \partial X$ we denote by $(\xi\cdot \eta)_a$ the quantity \[\inf \left\{ \liminf_{m,n\to \infty} (x_n\cdot y_m)_a : \{x_n\} \in \xi, \{y_n\}\in \eta\right\}.\] If $\{\xi_i\}_{i=1}^\infty \subset \partial X$ and $\xi \in \partial X$ then we say that $\xi_i\to \xi$ as $i\to \infty$ if $(\xi_i\cdot \xi)_a\to \infty$ as $i\to \infty$. A subset $C\subset \partial X$ is defined to be closed if given any sequence $\{\xi_i\}\subset C$ such that $\xi_i\to \xi$ in $\partial X$, we have $\xi \in C$.

\subsection{Surfaces, curves, and arcs}

In this paper, $S$ denotes an orientable surface of finite positive genus with either finitely many or infinitely many punctures. In all that follows, it will not matter if any of the isolated punctures of $S$ are replaced by boundary components. However, we will assume for simplicity that $\partial S=\emptyset$. We will also assume that $\chi(S)<0$ if $S$ is finite type. Thus the surface $S$ admits a complete hyperbolic metric of the \textit{first kind} (see, for instance \cite{simultaneous} Section 3 for a proof of this fact in the infinitely-punctured case). Recall that given a hyperbolic metric on $S$, the universal cover $\tilde{S}$ may be identified with $\H^2$ and we say that the metric is of the first kind if the limit set of $\pi_1(S)$ in the action on $\tilde{S}$ is equal to all of $\partial \tilde{S}=S^1$. We will fix a complete hyperbolic metric of the first kind on $S$.

Let $a$ and $b$ be isotopy classes of essential (see the definition below) simple closed curves on $S$. We define the \textit{geometric intersection number} $i(a,b)$ to be the minimum number of points of intersection between curves representing $a$ and curves representing $b$. There are unique simple closed geodesics in the isotopy classes $a$ and $b$. These geodesic representatives are in \textit{minimal position} meaning that they intersect the minimal number of times of curves in their respective isotopy classes. We will frequently conflate essential simple closed curves with their geodesic representatives and so we assume in particular that all essential simple closed curves have been put pairwise in minimal position.

If $\vec{a}$ and $\vec{b}$ are \textit{oriented} simple closed curves and $S$ has been endowed with an orientation, then we denote by $\hat{i}(\vec{a},\vec{b})$ the \textit{algebraic intersection number} of $\vec{a}$ and $\vec{b}$ (see for instance \cite{difftop} Chapter 3 for the definition).

Our object of study in this paper is the nonseparating curve graph $\NS(S)$.

\begin{defn}
The graph $\NS(S)$ has vertex set $\NS(S)^0$ equal to the set of isotopy classes of nonseparating simple closed curves on $S$. 
\begin{itemize}
\item If $\genus(S)=1$ then two isotopy classes $a$ and $b$ are joined by an edge whenever $i(a,b)\leq 1$.
\item If $\genus(S)>1$ then two isotopy classes $a$ and $b$ are joined by an edge whenever $i(a,b)=0$.
\end{itemize}
\end{defn}

\begin{defn}
A simple closed curve $c$ on $S$ is \textit{essential} if $c$ does not bound a disk or once-punctured disk in $S$. A subsurface $V\subset S$ is \textit{essential} if each boundary component of $V$ is an essential simple closed curve. A \textit{multicurve} on $S$ is a collection of pairwise disjoint essential simple closed curves on $S$, no two of which are isotopic.
\end{defn}

We are also interested in how curves and surfaces intersect in $S$:

\begin{defn}
Let $c$ be a simple closed curve and $V\subset S$ a subsurface. We say that $c$ intersects $V$ \textit{essentially} if $c$ cannot be isotoped to be disjoint from $V$.
\end{defn}

We will be particularly interested in the following class of subsurfaces of $S$:

\begin{defn}
Let $V\subset S$ be an essential subsurface. We say that $V$ is a \textit{witness} for $\NS(S)$ if every nonseparating simple closed curve in $S$ intersects $V$ essentially.
\end{defn}

It is not hard to see that $V$ is a witness for $\NS(S)$ if and only if $\genus(V)=\genus(S)$.

The idea of a witness was first articulated (in a more general context) by Masur-Schleimer in \cite{disk} where they used the term \textit{hole}. Schleimer has suggested that the term witnes should be used instead. Witnesses are interesting from the perspective of the geometry of $\NS(S)$:

\begin{lem}[Durham-Fanoni-Vlamis, \cite{dfv} Lemma 8.1]
If $V\subset S$ is a witness for $\NS(S)$ and has finite type then the inclusion $i:\NS(V)\to \NS(S)$ is an isometric embedding.
\end{lem}

Throughout this paper we will be interested in subarcs of various one-dimensional submanifolds of $S$. For the following, recall that a \textit{ray} is an embedding of the interval $[0,\infty)\subset \R$ into $S$.

\begin{conv}
Let $l$ be a simple ray, bi-infinite geodesic, or compact arc. If $p,q\in l$ then we denote by $l|[p,q]$ the unique embedded subarc of $l$ with endpoints $p$ and $q$.

If $c$ is a simple closed curve on $S$, then we denote by $\vec{c}$ the curve $c$ endowed with an orientation. If $p,q\in \vec{c}$ then $\vec{c}\setminus \{p,q\}$ consists of two embedded subarcs, each of which has an induced orientation from $\vec{c}$. We denote by $\vec{c}|[p,q]$ the unique arc in this pair which is oriented \textit{from} $p$ \textit{to} $q$.

If $\vec{l}$ is an oriented simple bi-infinite geodesic and $p\in \vec{l}$ then we denote by $\vec{l}|[p,\infty)$ the closed (containing $p$) ray of $\vec{l}$ which \emph{begins} at $p$, with respect to the orientation on $\vec{l}$. Similarly, $\vec{l}|(p,\infty)$ denotes the open ray $\vec{l}|[p,\infty) \setminus \{p\}$. We define $\vec{l}|(-\infty,p]$ to be the closed ray of $\vec{l}$ which \emph{ends} at $p$ and define $\vec{l}|(-\infty,p)=\vec{l}|(-\infty,p] \setminus \{p\}$.
\end{conv}

\subsubsection{Nonseparating curves and homology}

For a punctured finite genus surface $S$ we denote by $E(S)$ the space of ends of $S$ and by $\overline{S}$ the end compactification of $S$, $\overline{S}=S\cup E(S)$, which is a closed surface. We identify $E(S)$ with a (closed, nowhere dense, totally disconnected) subset of $\overline{S}$. We have an embedding $S \hookrightarrow \overline{S}$

We will use repeatedly the following fact:

\begin{fact}
\label{homfact}
A simple closed curve $c$ in $S$ is separating in $S$ if and only if the homology class of $c$ in $H_1(\overline{S};\Z)$ is zero when $c$ is endowed with either orientation.
\end{fact}

From now on, all homology groups will be understood to have coefficients in $\Z$. If $\vec{c}$ is an oriented simple closed curve then $[\vec{c}]$ will denote the homology class of $\vec{c}$ in $H_1(\overline{S})$ unless some other homology group is specified.

We also have:

\begin{fact}
A simple closed curve $c$ in $S$ is nonseparating in $S$ if and only if there exists a simple closed curve $d$ in $S$ with $i(c,d)=1$.
\end{fact}

We proved the following in \cite{nonsepbicorns}:

\begin{lem}[Rasmussen, \cite{nonsepbicorns} Lemma 2.1, Lemma 3.1]
\label{intersectionnumber}
Let $a,b\in \NS(S)^0$. Then $d_{\NS(S)}(a,b)\leq 2i(a,b)+1$.
\end{lem}

Here we make an \textit{important note}: \cite{nonsepbicorns} Lemmas 2.1 and 3.1 are stated for \textit{compact} surfaces $S$. However, the proofs go through when $S$ has (finitely many or infinitely many) punctures using Fact \ref{homfact}.

\subsection{Bicorns}
\label{bicornsection}

Here we define the main tools in this paper and recall some of their uses.
\begin{defn}
\label{bicorndefn}
Let $a$ and $b$ be simple closed curves on $S$. Let $\alpha$ and $\beta$ be subarcs of $a$ and $b$, respectively. If $\alpha$ and $\beta$ meet at their endpoints and nowhere in their interiors then the union $\alpha\cup \beta$ is a simple closed curve which we call a \textit{bicorn} between $a$ and $b$. We call the points of intersection of $\alpha$ and $\beta$ the \textit{corners} of the bicorn. We will also abuse the definition and consider $a$ and $b$ themselves to be bicorns.

If $A$ and $B$ are simple multicurves on $S$ then a bicorn between $A$ and $B$ is a bicorn between a curve $a\in A$ and a curve $b\in B$. 

More generally, if $L$ and $M$ are collections of pairwise disjoint simple geodesics on $S$ (closed or not, complete or not), then a bicorn between $L$ and $M$ is a curve $\alpha\cup \beta$ where $\alpha$ is a subarc of some geodesic in $L$, $\beta$ is a subarc of some geodesic in $M$ and $\alpha,\beta$ meet at their endpoints and nowhere in their interiors.
\end{defn}

If $a$ and $b$ are both nonseparating simple closed curves then we denote by $\B(a,b)$ the full subgraph of $\NS(S)$ spanned by the nonseparating bicorns between $a$ and $b$. Since we consider $a$ and $b$ themselves to be bicorns we have $a,b\in \B(a,b)$. Similarly if $A$ and $B$ are nonseparating muticurves then we denote by $\B(A,B)$ the full subgraph of $\NS(S)$ spanned by the nonseparating bicorns between $A$ and $B$ (consider the curves in $A$ and the curves in $B$ themselves as bicorns as well).

In \cite{nonsepbicorns} we used bicorn methods to prove the following Theorem:

\begin{thm}[\cite{nonsepbicorns}, Theorem 1.1 and Corollary 1.2]
\label{unihyp}
There exists $\delta>0$ with the following property. For any finite type or infinite type orientable surface $S$ with finite positive genus, $\NS(S)$ is $\delta$-hyperbolic.
\end{thm}

We also have the following corollary of our methods. To state the corollary recall the following definitions:

\begin{defn}
Let $R>0$ and $\Gamma$ be a graph. We say that a subgraph $\Delta\subset \Gamma$ is \textit{$R$-coarsely connected} if for every pair of vertices $a,b\in \Delta$, there exists a sequence $a_0=a, a_1,\ldots, a_n=b$ of vertices in $\Delta$ with $d(a_i,a_{i+1})\leq R$ for all $i$.
\end{defn}

\begin{defn}
Let $X$ be a metric space and $A,B$ compact subsets of $X$. The \textit{Hausdorff distance} between $A$ and $B$ is equal to the smallest $\epsilon>0$ such that $A$ is contained in the $\epsilon$-neighborhood of $B$ and $B$ is contained in the $\epsilon$-neighborhood of $A$. If the Hausdorff distance between $A$ and $B$ is at most $\delta$ then we say that $A$ and $B$ are $\delta$-Hausdorff close.
\end{defn}

\begin{cor}
\label{bicornshausclose}
There exists $C>0$ with the following property. Let $S$ be a finite type or infinite type orientable surface with finite positive genus and $a,b\in \NS(S)^0$. Then $\B(a,b)$ is 5-coarsely connected and $C$-Hausdorff close to any geodesic $[a,b]\subset \NS(S)$ from $a$ to $b$.
\end{cor}

\begin{proof}
In the proof of Theorem \ref{unihyp} from \cite{nonsepbicorns} we consider a graph $\NS^*(S)$ with the same vertex set as $\NS(S)$ but with two vertices being joined by an edge if the corresponding curves have intersection number $\leq 2$ (in \cite{nonsepbicorns} the graph $\NS^*(S)$ is simply denoted by $\NS(S)$). Denote by $\B^*(a,b)$ the subgraph of $\NS^*(S)$ spanned by the nonseparating bicorns between $a$ and $b$.

We prove in \cite{nonsepbicorns} Claim 3.3 that $\B^*(a,b)$ is connected. This proves that $\B(a,b)$ is 5-coarsely connected by Lemma \ref{intersectionnumber}. (We note here that although the proof of \cite{nonsepbicorns} Theorem 1.1 is given for compact surfaces, the entire proof goes through for non-compact surfaces using Fact \ref{homfact}. Hyperbolicity when $S$ is infinite type was originally stated as a \textit{corollary}, Corollary 1.2.)

To see that $\B(a,b)$ is $C$-Hausdorff close to $[a,b]$ we recall briefly the proof of Theorem \ref{unihyp}. The proof uses Bowditch's hyperbolicity criterion \cite{uniform} Proposition 3.1. We show that the subgraphs $\B^*(a,b)$ satisfy all the hypotheses of \cite{uniform} Proposition 3.1. Hence, in addition to the fact that $\NS^*(S)$ is hyperbolic, we have that $\B^*(a,b)$ is $C_0$-Hausdorff closed to any geodesic $[a,b]^*$ from $a$ to $b$ in $\NS^*(S)$, for any $a,b\in \NS^*(S)^0$. Moreover, the identity map $f:\NS^*(S)^0\to \NS(S)^0$ on vertices extends to a uniform quasi-isometry $f:\NS^*(S)\to \NS(S)$ by Lemma \ref{intersectionnumber}. Hence $\B(a,b)^0=f(\B^*(a,b)^0)$ is uniformly Hausdorff close to $f([a,b]^*)$. The set $f([a,b]^*)$ is a uniform quality quasi-geodesic from $a$ to $b$ in $\NS(S)$. Therefore it is uniformly Hausdorff close to any geodesic $[a,b]$ from $a$ to $b$ in $\NS(S)$ by the Morse Lemma.
\end{proof}

\subsection{Laminations}

Let $S$ be a surface endowed with a complete hyperbolic metric.

\begin{defn}
Let $L$ be a compact geodesic lamination on $S$. We say that $L$ is \textit{minimal} if every leaf of $L$ is dense in $L$. Moreover, we say that $L$ \textit{fills} a finite type subsurface $V\subset S$ if every essential simple closed curve in $V$ intersects $L$.
\end{defn}

If $S$ is finite type, then we denote by $\ELW(S)$ the space of minimal geodesic laminations on $S$ which fill finite type witnesses for $\NS(S)$ with the \textit{coarse Hausdorff topology}. The notation $\ELW(S)$ stands for \textit{ending laminations on witnesses}. Recall:

\begin{defn}
If $L$ and $M$ are compact geodesic laminations then denote by $d_{\operatorname{Haus}}(L,M)$ the Hausdorff distance between $L$ and $M$. The space $\GL(S)$ of compact geodesic laminations is a metric space with the distance $d_{\operatorname{Haus}}$. This metric space is compact. We denote by $L_n\toH M$ if $\{L_n\}_{n=1}^\infty \subset \GL(S)$, $M\in \GL(S)$ and $L_n$ converges to $M$ in the metric space $\GL(S)$. We say that $L_n$ \textit{Hausdorff converges} to $M$.

We say that $\{L_n\}\subset \GL(S)$ \textit{coarse Hausdorff converges} to a lamination $M$ and write $L_n \toCH M$ if for every subsequence $\{L_{n_i}\}_{i=1}^\infty$ such that $L_{n_i}$ Hausdorff converges to a lamination $N$, we have $M\subset N$. The \textit{coarse Hausdorff topology} on $\ELW(S)$ is defined by the condition that $A\subset \ELW(S)$ is closed if and only if whenever $\{L_n\} \subset A$ with $L_n\toCH L\in \ELW(S)$, we have $L\in A$.
\end{defn}

The topological space $\ELW(S)$ does not depend on the particular choice of hyperbolic metric on $S$. We will also denote by $\ML(S)$ the space of \textit{measured laminations} on $S$ with the weak${}^*$ topology.

\section{Nonseparating $2k$-corns are close to nonseparating bicorns}
\label{2kcornsec}

In this section we prove one of our main theorems, Theorem \ref{2kcornsclose} which says that nonseparating $2k$-corns are close to nonseparating bicorns. First we define what we mean by a $2k$-corn.

\begin{defn}
Let $A=\{a_1,\ldots,a_m\}$ and $B=\{b_1,\ldots,b_n\}$ be multicurves on $S$. Let $\alpha_1,\ldots,\alpha_k$ be a collection of pairwise disjoint closed subarcs of $A$ and $\beta_1,\ldots,\beta_k$ be a collection of pairwise disjoint closed subarcs of $B$. Denote by $p_i$ and $q_i$ the endpoints of the arc $\alpha_i$. Suppose that (indices being taken modulo $k$):
\begin{itemize}
\item $\alpha_i\cap \beta_{i-1}= \{p_i\}$,
\item $\alpha_i\cap \beta_i=\{q_i\}$,
\item $\alpha_i\cap \beta_j=\emptyset$ if $j\notin \{i,i-1\}$.
\end{itemize}
Then the closed curve \[c=\alpha_1\cup \beta_1\cup \alpha_2 \cup \beta_2 \cup \ldots \cup \alpha_k \cup \beta_k\] is simple. We call $c$ a \textit{$2k$-corn} between $A$ and $B$. We call the arcs $\alpha_i$ \textit{$A$-sides} of $c$ and the arcs $\beta_i$ \textit{$B$-sides} of $c$. We say that the $B$-sides $\beta_i$ and $\beta_j$ are \textit{adjacent} if $i=j+1$ or $i=j-1$ modulo $k$. We call the points $p_i$ and $q_i$ the \textit{corners} of $c$.
\label{2kcorndefn}
\end{defn}

\begin{defn}
Let $c=\alpha_1\cup \beta_1 \cup \ldots \cup \alpha_k \cup \beta_k$ be a $2k$-corn between $A$ and $B$. Let $\gamma$ be a closed subarc of $A$. We say that $\gamma$ \textit{joins} the side $\beta_i$ to the side $\beta_j$ if $\gamma$ intersects $c$ exactly at its endpoints and has an endpoint on $\beta_i$ and an endpoint on $\beta_j$. We say that $\gamma$ is \textit{parallel} to the side $\alpha_i$ if $\gamma$ joins $\beta_{i-1}$ to $\beta_i$.
\label{joining}
\end{defn}

\begin{defn}
Let $A,B$, and $c$ be as in Definition \ref{joining}. Let $\gamma_1$ and $\gamma_2$ be closed subarcs of $A$ joining two $B$-sides of $c$. We say that $\gamma_1$ is \textit{nested} inside $\gamma_2$ if one of the following two cases occurs:
\begin{enumerate}[(i)]
\item $\gamma_1$ and $\gamma_2$ both join the side $\beta_i$ of $c$ to itself, $\gamma_2$ has endpoints $r$ and $s$ on $\beta_i$, and $\beta_i|[r,s]$ contains the endpoints $\gamma_1\cap \beta_i$.
\item $\gamma_1$ and $\gamma_2$ are both parallel to the side $\alpha_i$ of $c$, $\gamma_2$ has endpoint $r$ on $\beta_{i-1}$ and endpoint $s$ on $\beta_i$, and $\beta_{i-1}|[r,p_i]\cup \beta_i|[q_i,s]$ contains the endpoints $\gamma_1\cap (\beta_{i-1}\cup \beta_i)$ (where $p_i$ and $q_i$ are as in Definition \ref{2kcorndefn}).
\end{enumerate}
\end{defn}

\begin{thm}
\label{2kcornsclose}
Let $k\geq 2$. Then there exists $D=D(k)>0$ with the following property. Let $S$ be a finite genus orientable surface and let $A$ and $B$ be two multicurves on $S$ such that each $a\in A$ and each $b\in B$ is nonseparating. Then any nonseparating $2k$-corn between $A$ and $B$ is $D$-close to a nonseparating bicorn between $A$ and $B$.
\end{thm}

Before giving the proof, we give a brief summary. The proof is related to that of Claim 3.5 from \cite{nonsepbicorns} and is by induction on $k$. Given a nonseparating $2k$-corn $c$ between $A$ and $B$ we must find a uniformly close nonseparating $2l$-corn with $l<k$. There are four steps in the inductive part of the proof, described below:
\begin{enumerate}[Step 1.]
\item First we construct a nonseparating bicorn $d$ between $A$ and $B$ with $A$-side contained in an $A$-side of $c$.
\item The bicorn $d$ may intersect $c$ many times. We analyze the points of intersection of $d$ with $c$ and the arcs of $d$ between consecutive points of intersection. We analyze many possible behaviors of these arcs and show that each of these behaviors is sufficient to easily produce a nonseparating $2l$-corn which is close to $c$, with $l<k$.
\item If none of the behaviors from Step 2 occur, the arcs of $d$ between consecutive points of intersection are strongly constrained. In this step we analyze what the intersection pattern between $d$ and $c$ is forced to look like.
\item Following the analysis in Step 3, we find that we may construct a nonseparating curve $c_0$ by replacing subarcs of $c$ by the ``outermost'' subarcs of $d$. The sequence $c,c_0,d$ is then a path of length two from $c$ to the nonseparating bicorn $d$. This will complete the proof.
\end{enumerate}

\begin{proof}[Proof of Theorem \ref{2kcornsclose}]
The proof is by induction on $k$. We will prove the induction step (therefore assuming $k>2$) first and return to the base case $k=2$ at the end of the proof.

Suppose the theorem has been proven for all $2\leq l<k$. Let $c$ be a nonseparating $2k$-corn between $A$ and $B$. We claim that there exists a nonseparating $2l$-corn $c'$ of $A$ with $B$ with $l<k$ and $d(c,c')\leq 2k+1$. By induction, there will exist a nonseparating bicorn $c''$ of $A$ with $B$ such that $d(c'',c')\leq D(l)$. This will complete the induction step, as we see that we may then set \[D(k)=\max\{ D(l):l<k\}+2k+1.\]

\noindent \textbf{Step 1 --- Constructing a nonseparating bicorn}

Choose $a\in A$. If $a$ intersects each $B$-side of $c$ at most once, then we have $i(a,c)\leq k$, $d(a,c)\leq 2k+1$ by Lemma \ref{intersectionnumber}, and the induction step is proven in this case. Otherwise, consider a $B$-side $\beta^*$ of $c$ with $|\beta^*\cap a|\geq 2$. Orient $a$ and number the points of intersection of $a$ with $\beta^*$ as $x_1, x_2, \ldots, x_m$ in the order that they appear along the resulting oriented curve $\vec{a}$. For each $i$, $c_i=\vec{a}|[x_i,x_{i+1}]\cup \beta^*|[x_i,x_{i+1}]$ is a bicorn between $a$ and $B$ (indices being taken modulo $m$) and we have (in $H_1(\overline{S};\Z)$) \[0\neq [\vec{a}] = \sum_{i=1}^m [c_i],\] where each $c_i$ is oriented such that the orientation on its $a$-side matches the orientation of $\vec{a}$. Hence some $c_i$ is nonseparating. We will rename $c_i$ to $d$. The nonseparating bicorn $d$ has the property that its $a$-side intersects $\beta^*$ exactly twice, at its endpoints.

\noindent \textbf{Step 2.1 --- Analysis of arcs joining non-adjacent sides of $c$}

Now we look at the consecutive points of intersection of the $a$-side of $d$ with the $B$-sides of $c$. Orient $d$ and enumerate the points of intersection $y_1,\ldots,y_n$ in the order that they appear along $d$, so that 
\begin{itemize}
\item $\vec{d}|[y_n,y_1]$ is the $B$-side of $d$ and 
\item $\vec{d}|[y_i,y_{i+1}]$ is contained in the $a$-side of $d$ for $i<n$.
\end{itemize} If some arc $\delta_i=\vec{d}|[y_i,y_{i+1}]$ with $i<n$ joins two $B$-sides of $c$ which are not adjacent, then we may express $[c]$ as $[c']+[c'']$ where $c'$ is a $2k_1$-corn and $c''$ is a $2k_2$-corn where $k_1,k_2<k$ and $c,c',$ and $c''$ are oriented appropriately (see Figure \ref{cut2kcorn}). Furthermore, $i(c,c')$ and $i(c,c'')$ are both at most 1. Since $c$ is nonseparating, we see that either $c'$ or $c''$ is nonseparating. If for instance $c'$ is nonseparating then we have $d(c,c')\leq 3<2k+1$ and this completes the induction in this case since $k_1<k$. So we may suppose without loss of generality that each arc $\delta_i$ for $i=1,\ldots,n-1$ either 
\begin{itemize}
\item joins a $B$-side of $c$ to itself,
\item is parallel to an $A$-side of $c$, or
\item is equal to an $A$-side of $c$.
\end{itemize}

\begin{figure}[h]
\centering
\def\svgwidth{0.7\textwidth}
\begingroup%
  \makeatletter%
  \providecommand\color[2][]{%
    \errmessage{(Inkscape) Color is used for the text in Inkscape, but the package 'color.sty' is not loaded}%
    \renewcommand\color[2][]{}%
  }%
  \providecommand\transparent[1]{%
    \errmessage{(Inkscape) Transparency is used (non-zero) for the text in Inkscape, but the package 'transparent.sty' is not loaded}%
    \renewcommand\transparent[1]{}%
  }%
  \providecommand\rotatebox[2]{#2}%
  \newcommand*\fsize{\dimexpr\f@size pt\relax}%
  \newcommand*\lineheight[1]{\fontsize{\fsize}{#1\fsize}\selectfont}%
  \ifx\svgwidth\undefined%
    \setlength{\unitlength}{1607.12713767bp}%
    \ifx\svgscale\undefined%
      \relax%
    \else%
      \setlength{\unitlength}{\unitlength * \real{\svgscale}}%
    \fi%
  \else%
    \setlength{\unitlength}{\svgwidth}%
  \fi%
  \global\let\svgwidth\undefined%
  \global\let\svgscale\undefined%
  \makeatother%
  \begin{picture}(1,0.67958038)%
    \lineheight{1}%
    \setlength\tabcolsep{0pt}%
    \put(0,0){\includegraphics[width=\unitlength,page=1]{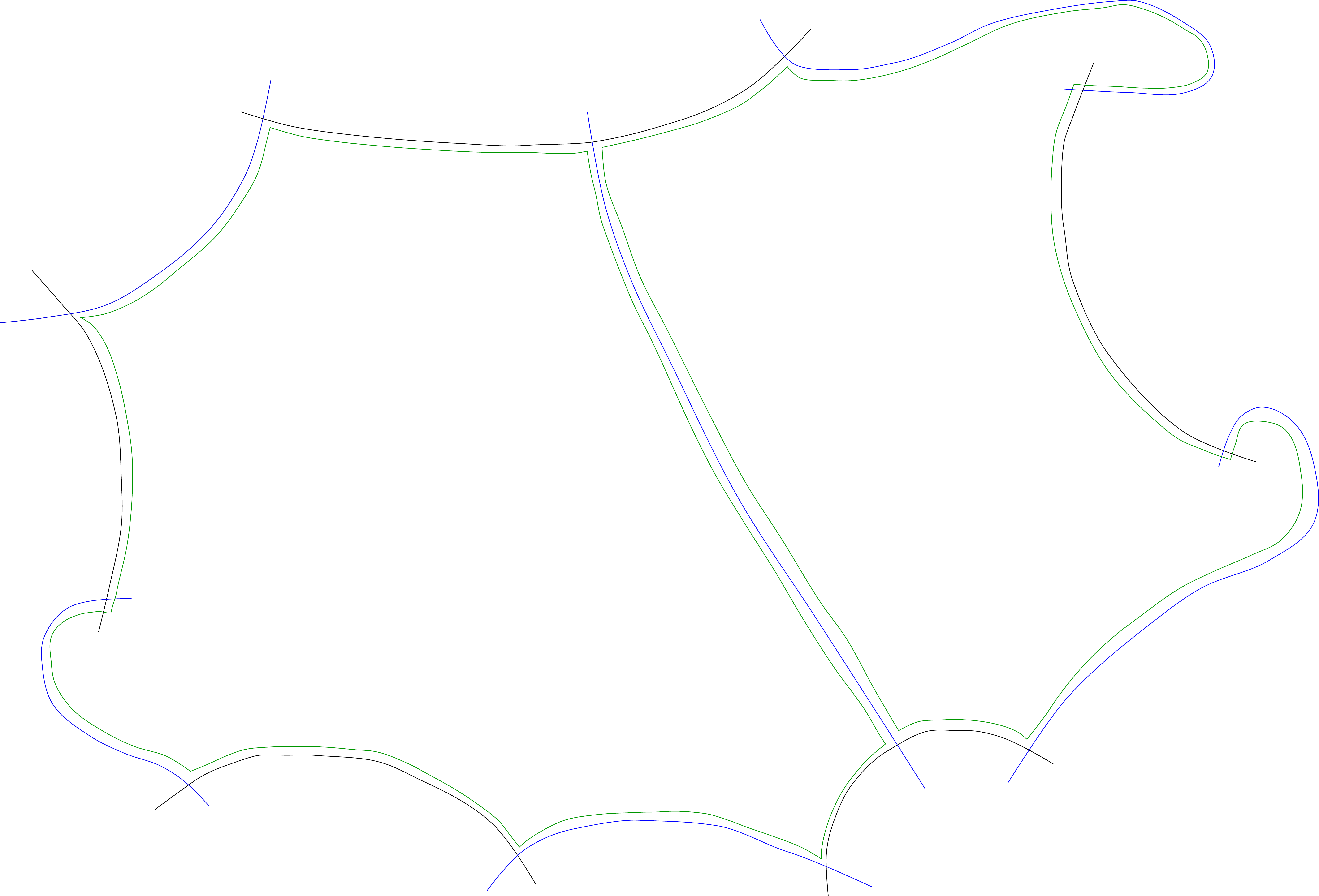}}%
    \put(0.30659852,0.54229092){\color[rgb]{0,0.58823529,0}\makebox(0,0)[lt]{\lineheight{1.25}\smash{\begin{tabular}[t]{l}$c'$\end{tabular}}}}%
    \put(0.51213289,0.56114728){\color[rgb]{0,0.58823529,0}\makebox(0,0)[lt]{\lineheight{1.25}\smash{\begin{tabular}[t]{l}$c''$\end{tabular}}}}%
  \end{picture}%
\endgroup%

\caption{If an arc of $a$ joins two non-adjacent $B$-sides of $c$, then we find a uniformly close nonseparating $2l$-corn with $l<k$.}
\label{cut2kcorn}
\end{figure}

\noindent \textbf{Step 2.2 --- Analysis of arcs joining sides of $c$ to themselves}

We now rule out (without loss of generality) several other behaviors of the arcs $\delta_i$. First we discuss the case of arcs joining a $B$-side $\beta$ to itself. For this purpose, orient $\beta$.
\begin{enumerate}[(i)]
\item If $\delta_i$ joins the left side of $\beta$ to the right side of $\beta$ then $c'=\delta_i\cup \beta|[y_i,y_{i+1}]$ intersects $c$ exactly once. Hence we have $c'\in \B(A,B)$ and $d(c,c')\leq 3$. See the left side of Figure \ref{samesidebicorn}.
\item If $\delta_i$ and $\delta_j$ both join the right side of $\beta$ to itself or the left side of $\beta$ to itself and $\beta|[y_i,y_{i+1}]$ contains \textit{exactly one} of the endpoints $\{y_j,y_{j+1}\}$ of $\delta_j$, then the bicorns $c'=\delta_i\cup \beta|[y_i,y_{i+1}]$ and $c''=\delta_j\cup \beta|[y_j,y_{j+1}]$ intersect each other exactly once and are disjoint from $c$. Hence $c'\in \B(A,B)$ and $d(c,c')=1$. See the right side of Figure \ref{samesidebicorn}.
\end{enumerate}
In either case we have found a nonseparating bicorn which is $3$-close to $c$ and the induction step is therefore complete. So we may assume without loss of generality that neither case (i) nor case (ii) above occurs.

\begin{figure}[h]
\centering
\begin{tabular}{c c}

\def\svgwidth{0.4\textwidth}
\begingroup%
  \makeatletter%
  \providecommand\color[2][]{%
    \errmessage{(Inkscape) Color is used for the text in Inkscape, but the package 'color.sty' is not loaded}%
    \renewcommand\color[2][]{}%
  }%
  \providecommand\transparent[1]{%
    \errmessage{(Inkscape) Transparency is used (non-zero) for the text in Inkscape, but the package 'transparent.sty' is not loaded}%
    \renewcommand\transparent[1]{}%
  }%
  \providecommand\rotatebox[2]{#2}%
  \newcommand*\fsize{\dimexpr\f@size pt\relax}%
  \newcommand*\lineheight[1]{\fontsize{\fsize}{#1\fsize}\selectfont}%
  \ifx\svgwidth\undefined%
    \setlength{\unitlength}{1223.20445143bp}%
    \ifx\svgscale\undefined%
      \relax%
    \else%
      \setlength{\unitlength}{\unitlength * \real{\svgscale}}%
    \fi%
  \else%
    \setlength{\unitlength}{\svgwidth}%
  \fi%
  \global\let\svgwidth\undefined%
  \global\let\svgscale\undefined%
  \makeatother%
  \begin{picture}(1,0.34245892)%
    \lineheight{1}%
    \setlength\tabcolsep{0pt}%
    \put(0,0){\includegraphics[width=\unitlength,page=1]{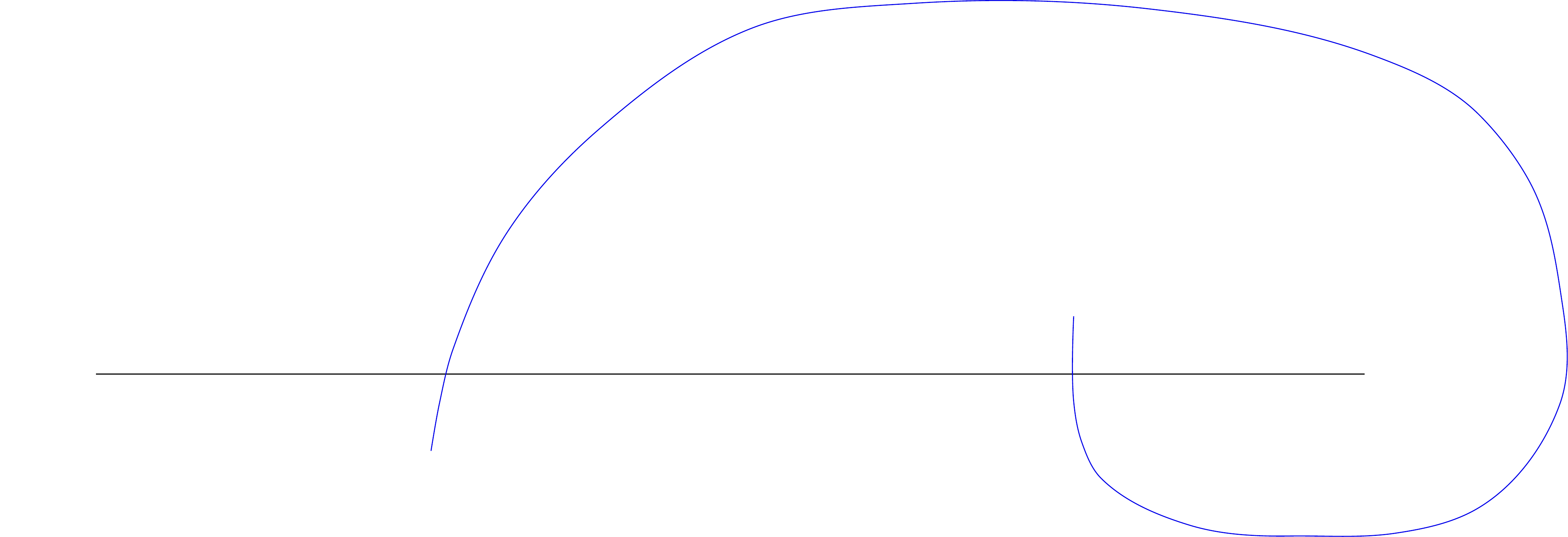}}%
    \put(-0.0018562,0.10387984){\color[rgb]{0,0,0}\makebox(0,0)[lt]{\lineheight{1.25}\smash{\begin{tabular}[t]{l}$\beta$\end{tabular}}}}%
    \put(0,0){\includegraphics[width=\unitlength,page=2]{bsidea.pdf}}%
    \put(0.46238114,0.27205636){\color[rgb]{0,0.58823529,0}\makebox(0,0)[lt]{\lineheight{1.25}\smash{\begin{tabular}[t]{l}$c'$\end{tabular}}}}%
    \put(0,0){\includegraphics[width=\unitlength,page=3]{bsidea.pdf}}%
  \end{picture}%
\endgroup%
 &

\def\svgwidth{0.4\textwidth}
\begingroup%
  \makeatletter%
  \providecommand\color[2][]{%
    \errmessage{(Inkscape) Color is used for the text in Inkscape, but the package 'color.sty' is not loaded}%
    \renewcommand\color[2][]{}%
  }%
  \providecommand\transparent[1]{%
    \errmessage{(Inkscape) Transparency is used (non-zero) for the text in Inkscape, but the package 'transparent.sty' is not loaded}%
    \renewcommand\transparent[1]{}%
  }%
  \providecommand\rotatebox[2]{#2}%
  \newcommand*\fsize{\dimexpr\f@size pt\relax}%
  \newcommand*\lineheight[1]{\fontsize{\fsize}{#1\fsize}\selectfont}%
  \ifx\svgwidth\undefined%
    \setlength{\unitlength}{1223.20445143bp}%
    \ifx\svgscale\undefined%
      \relax%
    \else%
      \setlength{\unitlength}{\unitlength * \real{\svgscale}}%
    \fi%
  \else%
    \setlength{\unitlength}{\svgwidth}%
  \fi%
  \global\let\svgwidth\undefined%
  \global\let\svgscale\undefined%
  \makeatother%
  \begin{picture}(1,0.34245892)%
    \lineheight{1}%
    \setlength\tabcolsep{0pt}%
    \put(0,0){\includegraphics[width=\unitlength,page=1]{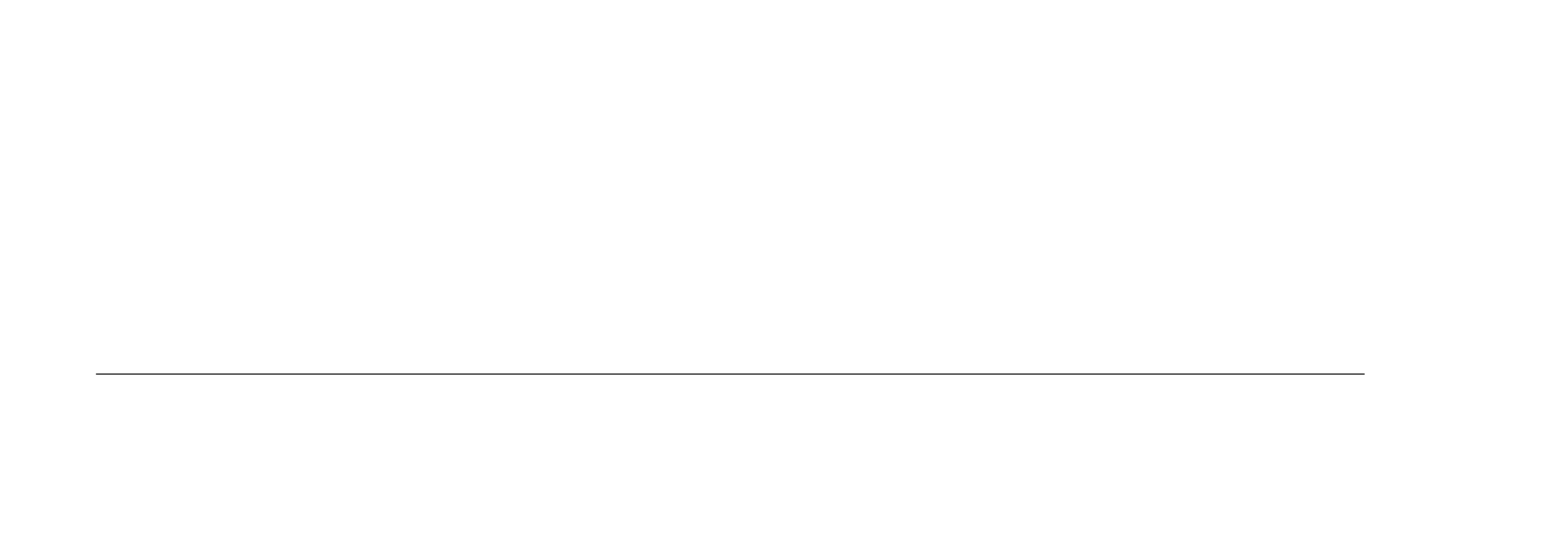}}%
    \put(-0.0018562,0.10387984){\color[rgb]{0,0,0}\makebox(0,0)[lt]{\lineheight{1.25}\smash{\begin{tabular}[t]{l}$\beta$\end{tabular}}}}%
    \put(0,0){\includegraphics[width=\unitlength,page=2]{bsideb.pdf}}%
    \put(0.38540458,0.25401665){\color[rgb]{0,0.58823529,0}\makebox(0,0)[lt]{\lineheight{1.25}\smash{\begin{tabular}[t]{l}$c'$\end{tabular}}}}%
    \put(0.63416564,0.26365175){\color[rgb]{0,0.58823529,0}\makebox(0,0)[lt]{\lineheight{1.25}\smash{\begin{tabular}[t]{l}$c''$\end{tabular}}}}%
    \put(0,0){\includegraphics[width=\unitlength,page=3]{bsideb.pdf}}%
  \end{picture}%
\endgroup%

\end{tabular}
\caption{Uniformly close bicorns in cases (i) and (ii).}
\label{samesidebicorn}
\end{figure}

\noindent \textbf{Step 2.3 --- Analysis of arcs joining adjacent sides of $c$}

Now we discuss the case of arcs $\delta_i$ which are parallel to $A$-sides of $c$. For this purpose, orient the curve $c$. In each of the behaviors (iii)-(vi) below, we will find a nonseparating bicorn or 4-corn which is $3$-close to $c$. Thus the induction step will be complete if any of the behaviors (iii)-(vi) occur.

\begin{enumerate}[(i)]
\setcounter{enumi}{2}
\item If the arc $\delta_i$ is not one of the $A$-sides of $c$ but is parallel to one of the $A$-sides of $c$ and joins the left side of $c$ to the right side of $c$ then we find a bicorn (if $\delta_i$ shares an endpoint with an $A$-side of $c$) or a 4-corn (otherwise) $c'$ with $i(c,c')=1$ (see Figure \ref{parallel1}). Hence $c'$ is nonseparating and $d(c,c')\leq 3$.

\item Consider $B$-sides $\beta'$ and $\beta''$ separated by the $A$-sides $\alpha'$ and $\alpha''$ and the $B$-side $\beta$. Orient the sides $\alpha'$, $\alpha''$, $\beta'$, $\beta''$, and $\beta$ to agree with the orientation of $c$ and suppose the sides have been named so that the concatenation $\beta' * \alpha' * \beta * \alpha'' * \beta''$ is a positively-oriented subpath of $c$. If
\begin{itemize}

\item $\delta_i$ joins the left side of $\beta$ to the left side of $\beta'$ and $\delta_j$ joins the left side of $\beta$ to the left side of $\beta''$,

\item and $(\delta_i\cap \beta) > (\delta_j\cap \beta)$ with respect to the orientation of $\beta$,

\end{itemize}
then we find two 4-corns or bicorns $c'$ and $c''$ between $A$ and $B$ with $i(c',c'')=1$ and $i(c',c)=i(c'',c)=0$ (see Figure \ref{parallel2}). Note that we are using that $k\geq 3$ here, so that $\beta'\neq \beta''$. Hence $c'$ is nonseparating and $d(c,c')=1$.

\item Consider $\delta_i$ and $\delta_j$ which join the left side of $c$ to itself. Suppose that $\delta_i$ joins a $B$-side $\beta$ of $c$ to itself and $\delta_j$ joins $\beta$ to an adjacent $B$-side $\beta'$ of $c$. Suppose that $\delta_i$ and $\delta_j$ have the property that $\beta|[y_i,y_{i+1}]$ contains the point $\delta_j\cap \beta$. Then we find a bicorn $c'$ and a 4-corn or bicorn $c''$ with $i(c',c'')=1$ and $i(c,c')=i(c,c'')=0$ (see Figure \ref{parallel3}). We have $c'\in \B(A,B)$ and $d(c,c')=1$.

\item If $\delta_i$ and $\delta_j$ both join the left side of $c$ to itself and are both parallel to the $A$-side $\alpha$ of $c$ and are not nested then we again find two 4-corns or bicorns $c'$ and $c''$ between $A$ and $B$ with $i(c,c')=i(c,c'')=0$ and $i(c',c'')=1$. See Figure \ref{parallel4}. Hence $c'$ is nonseparating and $d(c,c')=1$.

\end{enumerate}

Thus we may assume without loss of generality that none of the behaviors (iii)-(vi) occur.

\begin{figure}[h]
\centering
\begin{tabular}{c c c}

\def\svgwidth{0.33\textwidth}
\begingroup%
  \makeatletter%
  \providecommand\color[2][]{%
    \errmessage{(Inkscape) Color is used for the text in Inkscape, but the package 'color.sty' is not loaded}%
    \renewcommand\color[2][]{}%
  }%
  \providecommand\transparent[1]{%
    \errmessage{(Inkscape) Transparency is used (non-zero) for the text in Inkscape, but the package 'transparent.sty' is not loaded}%
    \renewcommand\transparent[1]{}%
  }%
  \providecommand\rotatebox[2]{#2}%
  \newcommand*\fsize{\dimexpr\f@size pt\relax}%
  \newcommand*\lineheight[1]{\fontsize{\fsize}{#1\fsize}\selectfont}%
  \ifx\svgwidth\undefined%
    \setlength{\unitlength}{928.15795898bp}%
    \ifx\svgscale\undefined%
      \relax%
    \else%
      \setlength{\unitlength}{\unitlength * \real{\svgscale}}%
    \fi%
  \else%
    \setlength{\unitlength}{\svgwidth}%
  \fi%
  \global\let\svgwidth\undefined%
  \global\let\svgscale\undefined%
  \makeatother%
  \begin{picture}(1,0.81446456)%
    \lineheight{1}%
    \setlength\tabcolsep{0pt}%
    \put(0,0){\includegraphics[width=\unitlength,page=1]{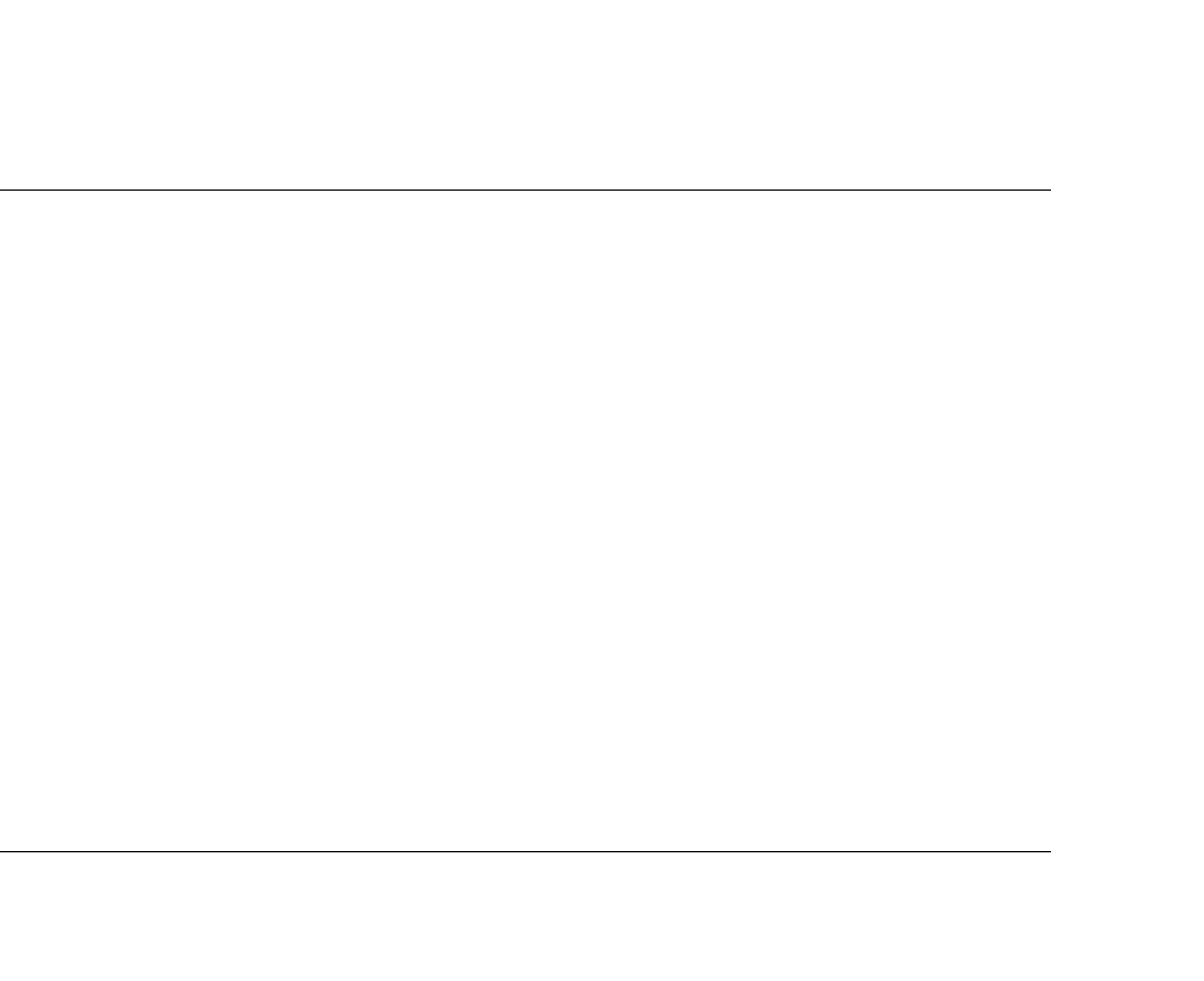}}%
    \put(0.6949248,0.4073293){\color[rgb]{0,0,1}\makebox(0,0)[lt]{\lineheight{1.25}\smash{\begin{tabular}[t]{l}$d$\end{tabular}}}}%
    \put(0,0){\includegraphics[width=\unitlength,page=2]{oppside1.pdf}}%
    \put(0.72840125,0.69822803){\color[rgb]{0,0.58823529,0}\makebox(0,0)[lt]{\lineheight{1.25}\smash{\begin{tabular}[t]{l}$c'$\end{tabular}}}}%
    \put(0.6949248,0.4073293){\color[rgb]{0,0,1}\makebox(0,0)[lt]{\lineheight{1.25}\smash{\begin{tabular}[t]{l}$d$\end{tabular}}}}%
  \end{picture}%
\endgroup%
 &

\def\svgwidth{0.33\textwidth}
\begingroup%
  \makeatletter%
  \providecommand\color[2][]{%
    \errmessage{(Inkscape) Color is used for the text in Inkscape, but the package 'color.sty' is not loaded}%
    \renewcommand\color[2][]{}%
  }%
  \providecommand\transparent[1]{%
    \errmessage{(Inkscape) Transparency is used (non-zero) for the text in Inkscape, but the package 'transparent.sty' is not loaded}%
    \renewcommand\transparent[1]{}%
  }%
  \providecommand\rotatebox[2]{#2}%
  \newcommand*\fsize{\dimexpr\f@size pt\relax}%
  \newcommand*\lineheight[1]{\fontsize{\fsize}{#1\fsize}\selectfont}%
  \ifx\svgwidth\undefined%
    \setlength{\unitlength}{928.15795898bp}%
    \ifx\svgscale\undefined%
      \relax%
    \else%
      \setlength{\unitlength}{\unitlength * \real{\svgscale}}%
    \fi%
  \else%
    \setlength{\unitlength}{\svgwidth}%
  \fi%
  \global\let\svgwidth\undefined%
  \global\let\svgscale\undefined%
  \makeatother%
  \begin{picture}(1,0.81446456)%
    \lineheight{1}%
    \setlength\tabcolsep{0pt}%
    \put(0,0){\includegraphics[width=\unitlength,page=1]{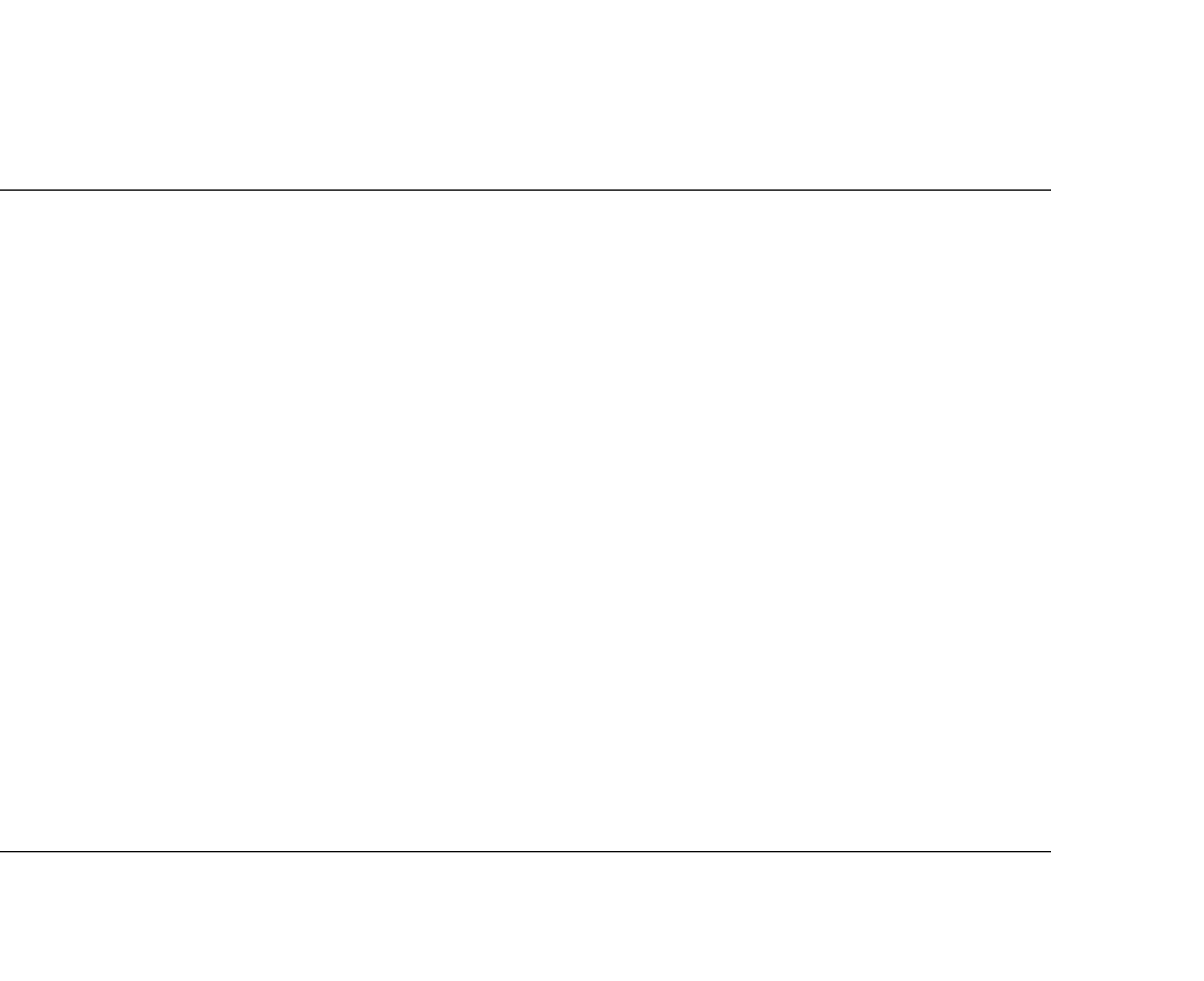}}%
    \put(0.6949248,0.4073293){\color[rgb]{0,0,1}\makebox(0,0)[lt]{\lineheight{1.25}\smash{\begin{tabular}[t]{l}$d$\end{tabular}}}}%
    \put(0,0){\includegraphics[width=\unitlength,page=2]{oppside2.pdf}}%
    \put(0.70185097,0.69822803){\color[rgb]{0,0.58823529,0}\makebox(0,0)[lt]{\lineheight{1.25}\smash{\begin{tabular}[t]{l}$c'$\end{tabular}}}}%
    \put(0.6949248,0.4073293){\color[rgb]{0,0,1}\makebox(0,0)[lt]{\lineheight{1.25}\smash{\begin{tabular}[t]{l}$d$\end{tabular}}}}%
    \put(0,0){\includegraphics[width=\unitlength,page=3]{oppside2.pdf}}%
  \end{picture}%
\endgroup%
 &

\def\svgwidth{0.33\textwidth}
\begingroup%
  \makeatletter%
  \providecommand\color[2][]{%
    \errmessage{(Inkscape) Color is used for the text in Inkscape, but the package 'color.sty' is not loaded}%
    \renewcommand\color[2][]{}%
  }%
  \providecommand\transparent[1]{%
    \errmessage{(Inkscape) Transparency is used (non-zero) for the text in Inkscape, but the package 'transparent.sty' is not loaded}%
    \renewcommand\transparent[1]{}%
  }%
  \providecommand\rotatebox[2]{#2}%
  \newcommand*\fsize{\dimexpr\f@size pt\relax}%
  \newcommand*\lineheight[1]{\fontsize{\fsize}{#1\fsize}\selectfont}%
  \ifx\svgwidth\undefined%
    \setlength{\unitlength}{928.15795898bp}%
    \ifx\svgscale\undefined%
      \relax%
    \else%
      \setlength{\unitlength}{\unitlength * \real{\svgscale}}%
    \fi%
  \else%
    \setlength{\unitlength}{\svgwidth}%
  \fi%
  \global\let\svgwidth\undefined%
  \global\let\svgscale\undefined%
  \makeatother%
  \begin{picture}(1,0.81446456)%
    \lineheight{1}%
    \setlength\tabcolsep{0pt}%
    \put(0,0){\includegraphics[width=\unitlength,page=1]{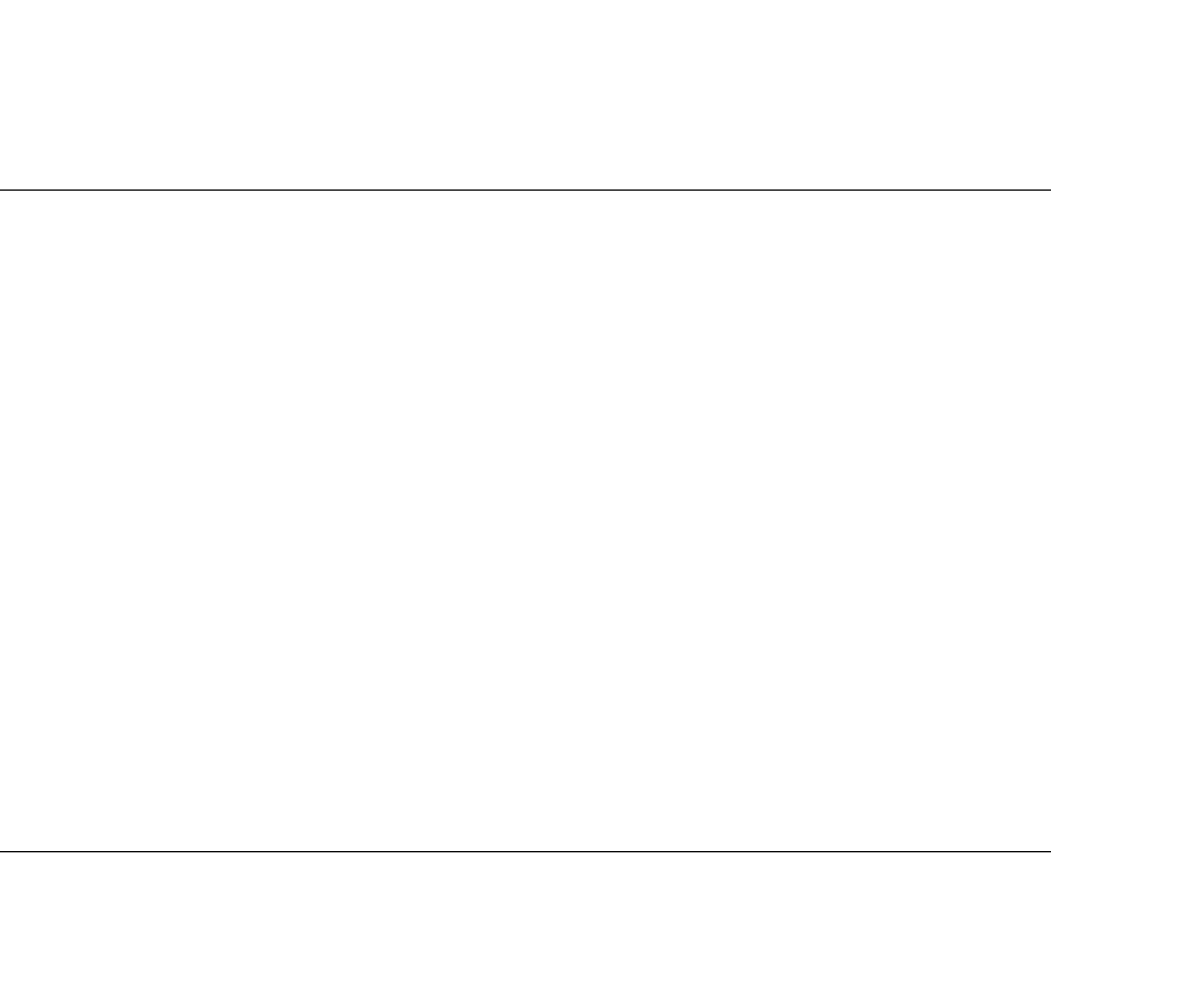}}%
    \put(0.63258935,0.4073293){\color[rgb]{0,0,1}\makebox(0,0)[lt]{\lineheight{1.25}\smash{\begin{tabular}[t]{l}$d$\end{tabular}}}}%
    \put(0,0){\includegraphics[width=\unitlength,page=2]{oppside3.pdf}}%
    \put(0.69072005,0.69294747){\color[rgb]{0,0.58823529,0}\makebox(0,0)[lt]{\lineheight{1.25}\smash{\begin{tabular}[t]{l}$c'$\end{tabular}}}}%
    \put(0,0){\includegraphics[width=\unitlength,page=3]{oppside3.pdf}}%
  \end{picture}%
\endgroup%

\end{tabular}
\caption{If an arc of $d$ is parallel to an $A$-side of $c$ and joins the left side of $c$ to the right side, then we find a uniformly close nonseparating 4-corn or bicorn $c'$.}
\label{parallel1}
\end{figure}

\begin{figure}[h]

\centering
\def\svgwidth{0.5\textwidth}
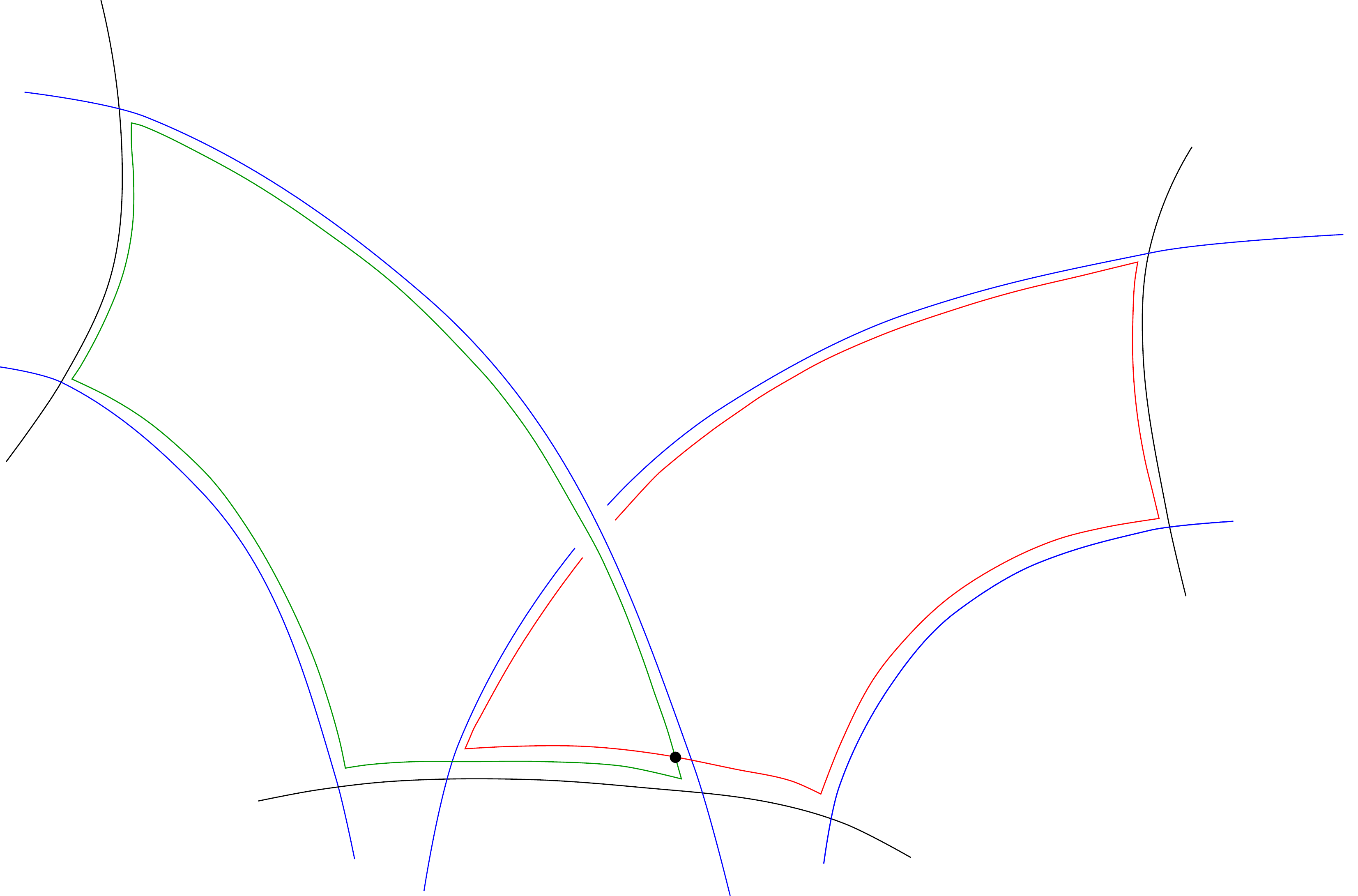

\caption{Behavior (iv) and uniformly close nonseparating 4-corns or bicorns $c'$ and $c''$.}
\label{parallel2}
\end{figure}

\begin{figure}[h]

\centering
\def\svgwidth{0.35\textwidth}
\begingroup%
  \makeatletter%
  \providecommand\color[2][]{%
    \errmessage{(Inkscape) Color is used for the text in Inkscape, but the package 'color.sty' is not loaded}%
    \renewcommand\color[2][]{}%
  }%
  \providecommand\transparent[1]{%
    \errmessage{(Inkscape) Transparency is used (non-zero) for the text in Inkscape, but the package 'transparent.sty' is not loaded}%
    \renewcommand\transparent[1]{}%
  }%
  \providecommand\rotatebox[2]{#2}%
  \newcommand*\fsize{\dimexpr\f@size pt\relax}%
  \newcommand*\lineheight[1]{\fontsize{\fsize}{#1\fsize}\selectfont}%
  \ifx\svgwidth\undefined%
    \setlength{\unitlength}{573.14224844bp}%
    \ifx\svgscale\undefined%
      \relax%
    \else%
      \setlength{\unitlength}{\unitlength * \real{\svgscale}}%
    \fi%
  \else%
    \setlength{\unitlength}{\svgwidth}%
  \fi%
  \global\let\svgwidth\undefined%
  \global\let\svgscale\undefined%
  \makeatother%
  \begin{picture}(1,1.31188645)%
    \lineheight{1}%
    \setlength\tabcolsep{0pt}%
    \put(0,0){\includegraphics[width=\unitlength,page=1]{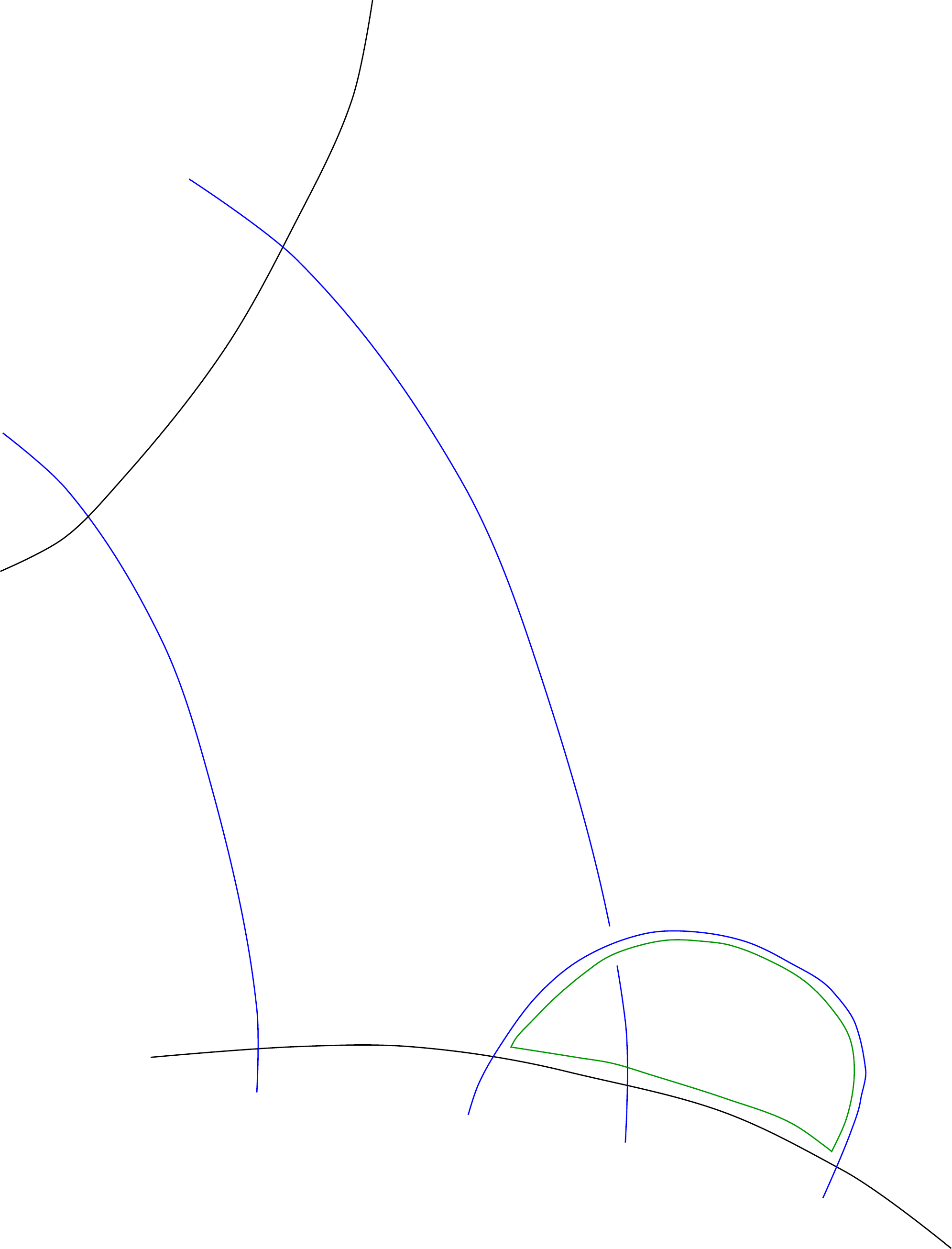}}%
    \put(0.72849018,0.25168741){\color[rgb]{0,0.58823529,0}\makebox(0,0)[lt]{\lineheight{1.25}\smash{\begin{tabular}[t]{l}$c'$\end{tabular}}}}%
    \put(0.51049585,0.12211673){\color[rgb]{0,0,0}\makebox(0,0)[lt]{\lineheight{1.25}\smash{\begin{tabular}[t]{l}$\beta$\end{tabular}}}}%
    \put(0.12715626,0.92845195){\color[rgb]{0,0,0}\makebox(0,0)[lt]{\lineheight{1.25}\smash{\begin{tabular}[t]{l}$\beta'$\end{tabular}}}}%
    \put(0,0){\includegraphics[width=\unitlength,page=2]{adjacentandsame.pdf}}%
    \put(0.13747357,0.75213893){\color[rgb]{1,0,0}\makebox(0,0)[lt]{\lineheight{1.25}\smash{\begin{tabular}[t]{l}$c''$\end{tabular}}}}%
  \end{picture}%
\endgroup%
 \\

\caption{Behavior (v) and a uniformly close nonseparating bicorn $c'$}
\label{parallel3}
\end{figure}

\begin{figure}[h]

\centering
\def\svgwidth{0.4\textwidth}
\begingroup%
  \makeatletter%
  \providecommand\color[2][]{%
    \errmessage{(Inkscape) Color is used for the text in Inkscape, but the package 'color.sty' is not loaded}%
    \renewcommand\color[2][]{}%
  }%
  \providecommand\transparent[1]{%
    \errmessage{(Inkscape) Transparency is used (non-zero) for the text in Inkscape, but the package 'transparent.sty' is not loaded}%
    \renewcommand\transparent[1]{}%
  }%
  \providecommand\rotatebox[2]{#2}%
  \newcommand*\fsize{\dimexpr\f@size pt\relax}%
  \newcommand*\lineheight[1]{\fontsize{\fsize}{#1\fsize}\selectfont}%
  \ifx\svgwidth\undefined%
    \setlength{\unitlength}{909.20239066bp}%
    \ifx\svgscale\undefined%
      \relax%
    \else%
      \setlength{\unitlength}{\unitlength * \real{\svgscale}}%
    \fi%
  \else%
    \setlength{\unitlength}{\svgwidth}%
  \fi%
  \global\let\svgwidth\undefined%
  \global\let\svgscale\undefined%
  \makeatother%
  \begin{picture}(1,0.75303995)%
    \lineheight{1}%
    \setlength\tabcolsep{0pt}%
    \put(0,0){\includegraphics[width=\unitlength,page=1]{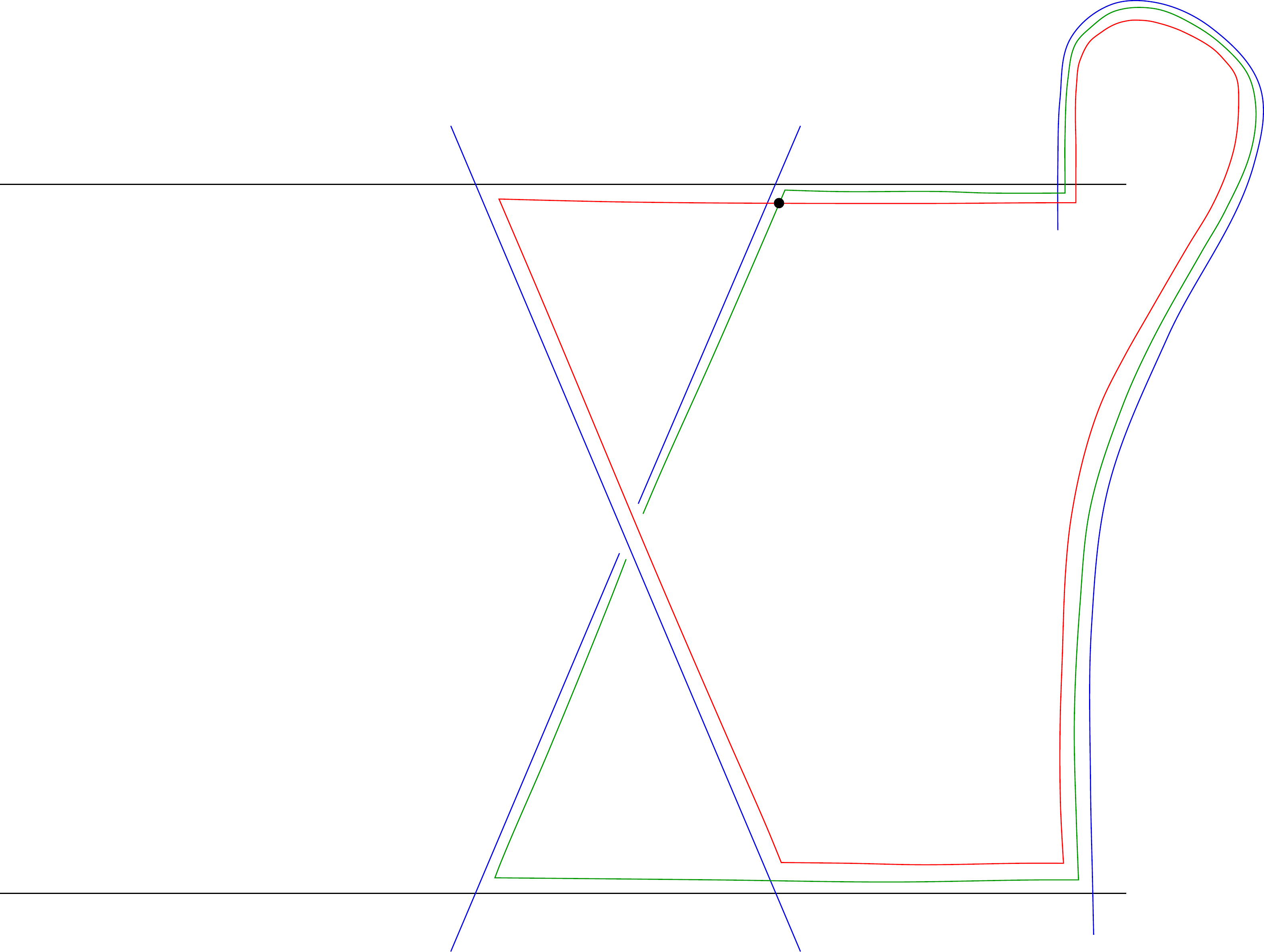}}%
    \put(0.55857433,0.40584594){\color[rgb]{0,0.58823529,0}\makebox(0,0)[lt]{\lineheight{1.25}\smash{\begin{tabular}[t]{l}$c'$\end{tabular}}}}%
    \put(0.44544537,0.50836908){\color[rgb]{1,0,0}\makebox(0,0)[lt]{\lineheight{1.25}\smash{\begin{tabular}[t]{l}$c''$\end{tabular}}}}%
    \put(0.3853456,0.20079967){\color[rgb]{0,0,1}\makebox(0,0)[lt]{\lineheight{1.25}\smash{\begin{tabular}[t]{l}$d$\end{tabular}}}}%
    \put(0.87660394,0.23664969){\color[rgb]{0,0,1}\makebox(0,0)[lt]{\lineheight{1.25}\smash{\begin{tabular}[t]{l}$\alpha$\end{tabular}}}}%
  \end{picture}%
\endgroup%

\caption{Behavior (vi) and uniformly close nonseparating 4-corns or bicorns $c'$ and $c''$.}
\label{parallel4}
\end{figure}

\noindent \textbf{Step 3 --- Analysis of the pattern of intersection of $d$ with $c$}

Finally then recall that we are assuming without loss of generality that every arc $\delta_i$ for $i=1,\ldots,n-1$ either
\begin{enumerate}[(1)]
\item is an $A$-side of $c$,
\item joins the left (or right) of a $B$-side of $c$ to itself, or
\item is parallel to an $A$-side of $c$ and joins the left or (right) side of $c$ to itself.
\end{enumerate}

We now consider \textit{only} the arcs $\delta_i$ which join the left side of $c$ to itself. We are further assuming
\begin{itemize}
\item if two such arcs $\delta_i$ and $\delta_j$ both join a given $B$-side of $c$ to itself then they are nested;
\item if two such arcs $\delta_i$ and $\delta_j$ are both parallel to a given $A$-side of $c$ then they are nested;
\item if $\delta_i$ has type (2) above and $\delta_j$ has type (3) above then they do not have the configuration shown in Figure \ref{parallel3};
\item if $\delta_i$ and $\delta_j$ both have type (3) above then they do not have the configuration shown in Figure \ref{parallel2}.
\end{itemize}

We find that the arcs $\delta_j$ joining the left side of $c$ to itself must form a pattern as shown in Figure \ref{nonseparatingcurve}. Formally, this implies the following properties. Note that the endpoints of $\beta_i$ are $q_i$ and $p_{i+1}$. Then if $\beta_i$ is oriented from $q_i$ to $p_{i+1}$ there are points $q_i<r_i<s_i<p_{i+1}$ such that for any arc $\delta_j$ joining the left side of $c$ to itself:
\begin{itemize}
\item if $\delta_j$ joins $\beta_i$ to $\beta_{i-1}$ (i.e. is parallel to $\alpha_i$) then $\delta_j\cap \beta_i<r_i$;
\item if $\delta_j$ joins $\beta_i$ to $\beta_{i+1}$ (i.e. is parallel to $\alpha_{i+1}$) then $\delta_j\cap \beta_i>s_i$.
\end{itemize}

From this property and the ruling out of behaviors (iii)-(vi), we may define a set of arcs $\delta'_i$ for $i=1,\ldots,k$ as follows. Either $\delta'_i$ is parallel to $\alpha_i$ and is an ``outermost'' arc with this property or $\delta'_i$ is $\alpha_i$ itself if there is no such arc parallel to $\alpha_i$. Formally:
\begin{itemize}
\item for each $i=1,\ldots,k$, either
\begin{enumerate}[(1)]
\item $\delta_i'=\alpha_i$, or
\item $\delta_i'$ is equal to some $\delta_j$ which is parallel to the $A$-side $\alpha_i$ of $c$ and joins the left side of $c$ to itself,
\end{enumerate}
\item denoting by $z_i$ the endpoint of $\delta'_i$ on $\beta_i$ and $w_i$ the endpoint of $\delta'_{i+1}$ on $\beta_i$ (indices being taken mod $k$), if any $\delta_j$ intersects the interior of the arc $\beta_i|[z_i,w_i]$ and joins the left side of $c$ to itself, then $\delta_j$ must join $\beta_i|[z_i,w_i]$ to itself.
\end{itemize}

Moreover, for each $i=1,\ldots,k$, there is a (possibly empty) collection of ``outermost'' arcs $\{\delta''^j_i\}_{j=1}^{m_i}$ of $d$ joining the left side of $\beta_i$ to itself. Formally:
\begin{itemize}
\item each $\delta''^j_i$ is equal to some $\delta_l$ joining the left side $\beta_i$ to itself;
\item the $\delta''^j_i$ are arranged such that if $u_i^j$ and $v_i^j$ are the endpoints of $\delta''^j_i$ we have \[z_i<u_i^1<v_i^1<u_i^2<v_i^2< \ldots < u_i^{m_i} <v_i^{m_i}<w_i\] in the orientation on $\beta_i$;
\item each arc \[\beta_i|[z_i,u_i^1], \beta_i|[v_i^1,u_i^2], \beta_i|[v_i^2,u_i^3], \ldots, \beta_i|[v_i^{m_i-1},u_i^{m_i}], \beta_i|[v_i^{m_i},w_i]\] does not meet any $\delta_j$ in its interior. 
\end{itemize}

\begin{figure}[h]
\centering
\def\svgwidth{1.15\textwidth}
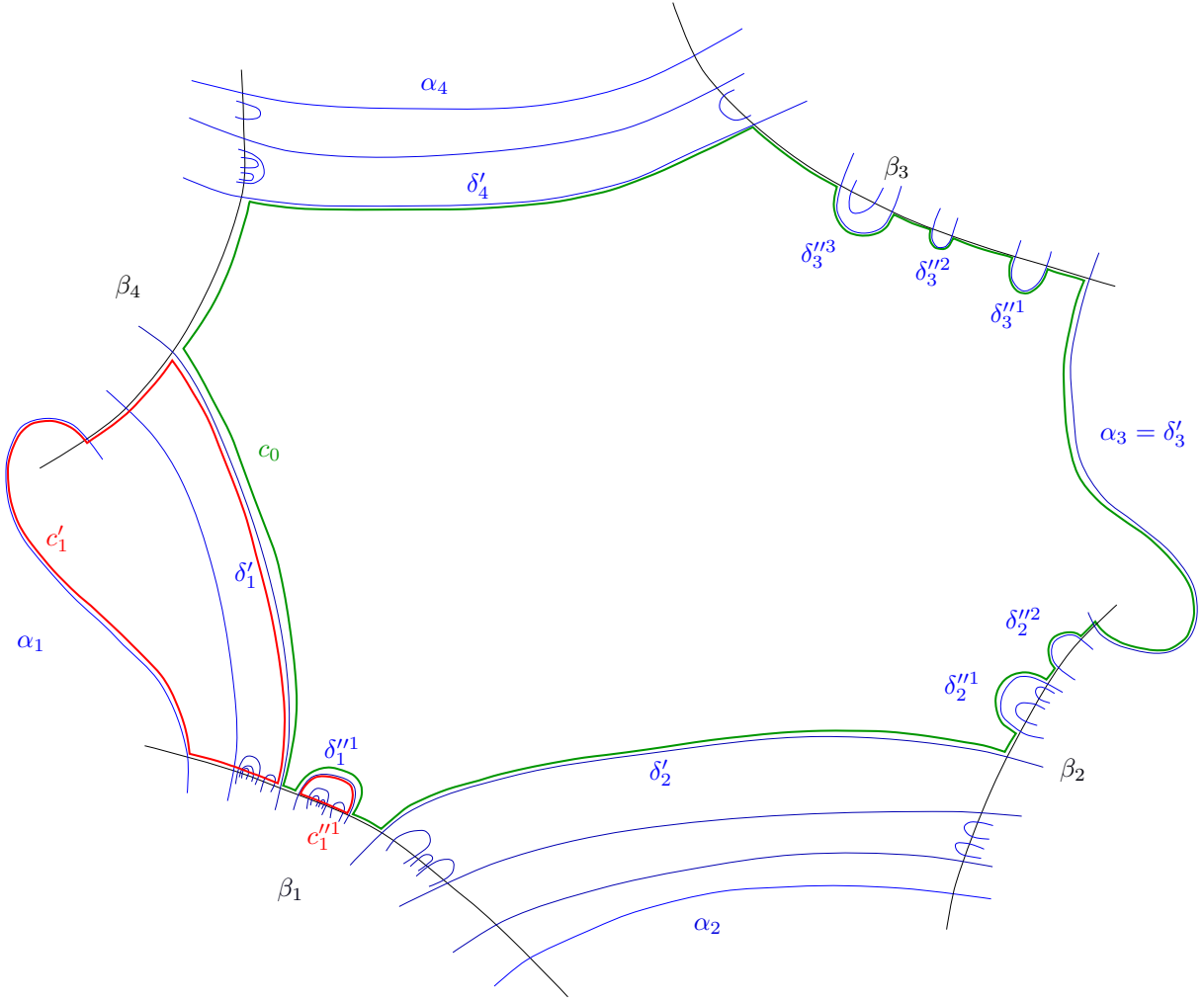

\caption{The arrangement of the arcs $\delta_i$ after behaviors (i)-(vi) have been ruled out. This also shows the construction of the nonseparating curve $c_0$.}
\label{nonseparatingcurve}
\end{figure} 

\noindent \textbf{Step 4 --- Construction of a path to $d$}

Each $\delta_i'$ which is not equal to $\alpha_i$ defines either a bicorn together with a subarc of $\beta_{i-1}$ or $\beta_i$ (if it shares an endpoint with $\alpha_i$) or a 4-corn together with $\alpha_i$ and subarcs of $\beta_{i-1}$ and $\beta_i$ (otherwise) which we call $c_i'$. Namely, recall that $z_i$ is the endpoint of $\delta'_i$ on $\beta_i$ and $w_{i-1}$ is the endpoint of $\delta'_i$ on $\beta_{i-1}$ and we set \[c_i'=\alpha_i\cup \beta_i|[q_i,z_i] \cup \delta_i' \cup \beta_{i-1}|[w_{i-1},p_i]\] (noting that the arc $\beta_i|[q_i,z_i]$ is degenerate if $q_i=z_i$ and similarly for $\beta_{i-1}|[w_{i-1},p_i]$). Moreover, each $\delta''^j_i$ defines a bicorn together with a subarc of $\beta_i$ which we call $c''^j_i$. Namely, \[c''^j_i=\delta_i''^j\cup \beta_i|[u_i^j,v_i^j].\] If any of the curves $c_i'$ or $c_i''^j$ is nonseparating, then the induction is complete by the usual arguments. Otherwise, we orient each $c_i'$ and each $c''^j_i$ such that the resulting orientation matches that of $c$ along their overlap. The arcs where the $c_i'$ and $c''^j_i$ overlap with $c$ are moreover pairwise disjoint. We obtain a simple closed curve $c_0$ by replacing each of these arcs of $c$ with the corresponding complementary subarcs of $c'_i$ or $c''^j_i$. In $H_1(\overline{S})$, we have \[[c_0]=[c] + \sum_{i=1}^k [c'_i]+ \sum_{i=1}^k \sum_{j=1}^{m_i} [c''^j_i]\] when $c_0$ is oriented appropriately. Moreover, since each homology class in the sum other than $[c]$ is 0, we see that $c_0$ (which is homologous to $c$) is nonseparating. We have both $i(c_0,c)=0$ and $i(c_0,d)=0$ so in this case we have $d(c,d)\leq 2$ and $d\in \B(A,B)$.

\noindent \textbf{Base case}

Now we consider the base case $k=2$. We claim that we may set $D(2)=8$. Let $c$ be a nonseparating 4-corn and choose $a\in A$. If $a$ intersects each of the $B$-sides of $c$ at most once then we see that $i(a,c)\leq 2$ and $c$ is $5$-close to $a\in \B(A,B)$. Otherwise we easily find an arc of $a$ intersecting some $B$-side $\beta^*$ of $c$ exactly at its endpoints, and for which the resulting bicorn $d$ of $a$ with $\beta^*$ is nonseparating, as before. As before, we consider the consecutive points of intersection $y_1,\ldots,y_n$ of the $a$-side of $d$ with the $B$-sides of $c$. Again, we choose the numbering such that $\vec{d}|[y_i,y_{i+1}]$ is contained in the $a$-side of $d$ for $i<n$ and $\vec{d}|[y_n,y_1]$ is the $B$-side of $d$. Each arc $\delta_i=\vec{d}|[y_i,y_{i+1}]$ for $i=2,\ldots,n-1$ joins the other $B$-side $\beta\neq \beta^*$ of $c$ to itself. If any of these arcs have the behavior described in points (i) or (ii) above then we find a nonseparating bicorn $c'$ with $i(c,c')\leq 1$ and thus the proof of the theorem is complete. Otherwise we orient $c$ and consider the arcs $\delta_i$ which join the left side of $c$ to itself. There is a collection of ``outermost'' arcs $\delta'_1,\ldots,\delta'_m$ such that any $\delta_i$ joining the left side of $\beta$ to itself is nested inside one of the $\delta'_j$. Each $\delta'_j$ defines a bicorn $c'_j$ together with a subarc of $\beta$. Namely if $u_j$ and $v_j$ are the endpoints of $\delta'_j$ then $c'_j=\delta'_j\cup \beta|[u_j,v_j]$. If any of these bicorns $c'_j$ is nonseparating then the proof is complete. Otherwise we define a simple closed curve $c_0$ by removing the arc of $c'_j$ along $\beta$ (that is $\beta|[u_j,v_j]$) from $c$ and replacing it by $\delta_j'$. Again, $c_0$ is disjoint from $c$, homologous to $c$, and thus nonseparating. We have $i(c_0,d)\leq 3$ (see the proof of Claim 2.5 from \cite{nonsepbicorns}) so $d(c,d)\leq d(c,c_0)+d(c_0,d)\leq 1+7=8$. See Figure \ref{4corn} for a summary. So $d\in \B(A,B)$ and is 8-close to $c$. We may take $D(2)=8$ and this completes the proof of the base case.

\begin{figure}[h]
\centering
\def\svgwidth{0.8\textwidth}
\begingroup%
  \makeatletter%
  \providecommand\color[2][]{%
    \errmessage{(Inkscape) Color is used for the text in Inkscape, but the package 'color.sty' is not loaded}%
    \renewcommand\color[2][]{}%
  }%
  \providecommand\transparent[1]{%
    \errmessage{(Inkscape) Transparency is used (non-zero) for the text in Inkscape, but the package 'transparent.sty' is not loaded}%
    \renewcommand\transparent[1]{}%
  }%
  \providecommand\rotatebox[2]{#2}%
  \newcommand*\fsize{\dimexpr\f@size pt\relax}%
  \newcommand*\lineheight[1]{\fontsize{\fsize}{#1\fsize}\selectfont}%
  \ifx\svgwidth\undefined%
    \setlength{\unitlength}{837.04185185bp}%
    \ifx\svgscale\undefined%
      \relax%
    \else%
      \setlength{\unitlength}{\unitlength * \real{\svgscale}}%
    \fi%
  \else%
    \setlength{\unitlength}{\svgwidth}%
  \fi%
  \global\let\svgwidth\undefined%
  \global\let\svgscale\undefined%
  \makeatother%
  \begin{picture}(1,0.72383745)%
    \lineheight{1}%
    \setlength\tabcolsep{0pt}%
    \put(0,0){\includegraphics[width=\unitlength,page=1]{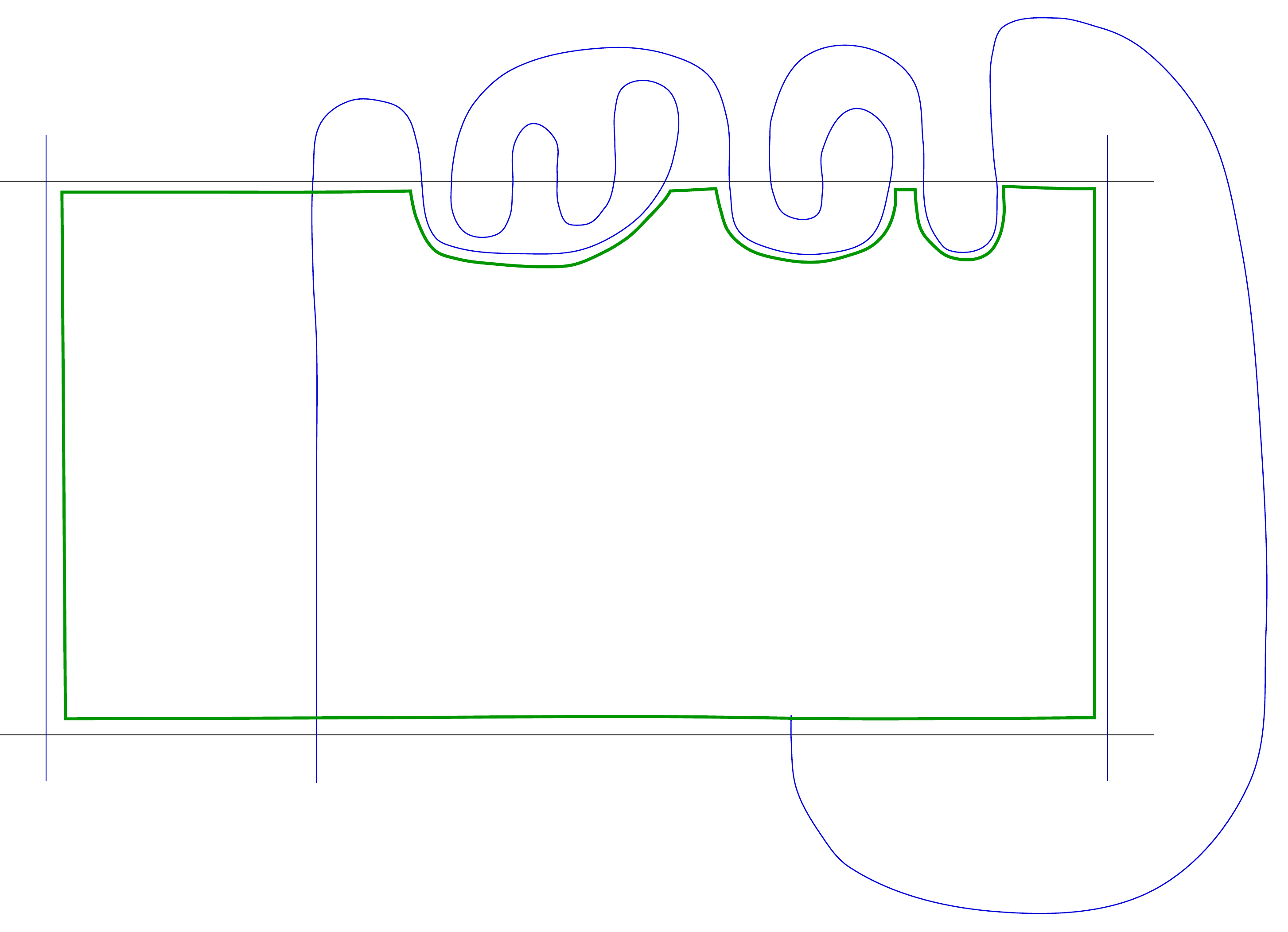}}%
    \put(0.37032143,0.48996821){\color[rgb]{0,0,1}\makebox(0,0)[lt]{\lineheight{1.25}\smash{\begin{tabular}[t]{l}$\delta'_1$\end{tabular}}}}%
    \put(0.57397095,0.48996821){\color[rgb]{0,0,1}\makebox(0,0)[lt]{\lineheight{1.25}\smash{\begin{tabular}[t]{l}$\delta'_2$\end{tabular}}}}%
    \put(0.71245272,0.48996821){\color[rgb]{0,0,1}\makebox(0,0)[lt]{\lineheight{1.25}\smash{\begin{tabular}[t]{l}$\delta'_3$\end{tabular}}}}%
    \put(0.05805868,0.35465439){\color[rgb]{0,0.58823529,0}\makebox(0,0)[lt]{\lineheight{1.25}\smash{\begin{tabular}[t]{l}$c_0$\end{tabular}}}}%
    \put(0,0){\includegraphics[width=\unitlength,page=2]{4corn.pdf}}%
    \put(0.20759809,0.4409781){\color[rgb]{0.77254902,0,0}\makebox(0,0)[lt]{\lineheight{1.25}\smash{\begin{tabular}[t]{l}$d$\end{tabular}}}}%
    \put(0,0){\includegraphics[width=\unitlength,page=3]{4corn.pdf}}%
  \end{picture}%
\endgroup%

\caption{The proof of the base case of Theorem \ref{2kcornsclose}.}
\label{4corn}
\end{figure} 
\end{proof}

\begin{rem}
\label{Dcalculation}
In fact, the above proof shows that we may take \[D(k)=8+\sum_{i=3}^k (2i+1)=k^2+2k\]
\end{rem}

\begin{defn}
If $A$ and $B$ are a pair of multicurves, such that each $a\in A$ and each $b\in B$ is nonseparating and $k\geq 1$ then we denote by $A_k(A,B)$ the set of \textit{nonseparating} $2k$-corns between $A$ and $B$. We denote by $\B_k(A,B)$ the subgraph of $\NS(S)$ spanned by $\bigcup_{l=1}^k A_l(A,B)$.
\end{defn}

Let $A$ and $B$ be multicurves such that each $a\in A$ and each $b\in B$ is nonseparating. For any choices of curves $a,a'\in A$ and $b,b'\in B$ and for any choices of geodesics $\gamma$ from $a$ to $b$ and $\gamma'$ from $a'$ to $b'$ in $\NS(S)$, the geodesics $\gamma$ and $\gamma'$ are uniformly Hausdorff close (with constant depending only on the constant of hyperbolicity of $\NS(S)$) by the Morse Lemma. We denote by $[A,B]$ the union of all geodesics from curves in $A$ to curves in $B$ in $\NS(S)$.

\begin{cor}
\label{2kcornshausclose}
Let $k\geq 1$. Then there exists $E=E(k)$ with the following property. Let $S$ be a finite genus orientable surface and $A$ and $B$ be multicurves on $S$ such that each $a\in A$ and each $b\in B$ is nonseparating. Then $\B_k(A,B)$ is $E$-Hausdorff close to $[A,B]$. Moreover, the optimal $E(k)$ grows at most quadratically with $k$.
\end{cor}

\begin{proof}
First of all, by the Morse Lemma and Corollary \ref{bicornshausclose}, there exists $E_1>0$ independent of $S,k,A,$ and $B$ such that $[A,B]$ is $E_1$-Hausdorff close to $\B(A,B)$. In particular, $[A,B]$ is contained in the $E_1$-neighborhood of $\B(A,B)\subset \B_k(A,B)$.

On the other hand, by Theorem \ref{2kcornsclose}, any element of $\B_k(A,B)$ is $E_2(k)$-close to $\B(A,B)$ where \[E_2(k)=\max \{D(l):1\leq l \leq k\}.\] Therefore any element of $\B_k(A,B)$ is $(E_1+E_2(k))$-close to $[A,B]$ by the previous paragraph.

Taking $E(k)=E_1+E_2(k)$, we have that $[A,B]$ is $E(k)$-Hausdorff close to $\B_k(A,B)$, as claimed. Since by Remark \ref{Dcalculation} we may take $D(k)=k^2+k$, we may also take $E_2(k)=k^2+k$ and this shows that the optimal $E(k)$ grows at most quadratically.
\end{proof}

\section{Nontrivially separating bicorns are disjoint from nonseparating bicorns}
\label{nontrivsepsec}

\begin{defn}
Let $c$ be a separating simple closed curve in $S$. We say that $c$ is \textit{trivially separating} if it bounds a punctured disk. Otherwise we say that $c$ is \textit{nontrivially separating}.
\end{defn}

Thus, if $c$ is nontrivially separating then each component of $S\setminus c$ contains a nonseparating simple closed curve. We denote by $\NS^\perp(c)$ the subgraph spanned by nonseparating simple closed curves disjoint from $c$. It is a diameter two subgraph of $\NS(S)$.

\begin{lem}
Let $a\in \NS(S)^0$ and let $b$ be a simple geodesic. Suppose that the curve $c$ is a bicorn between $a$ and $b$ and that $c$ is nontrivially separating. Then there is a bicorn between $a$ and $b$ which lies in $\NS^\perp(c)$.
\label{nontrivsep}
\end{lem}

\begin{rem}
Note that in the statement of Lemma \ref{nontrivsep} it is not necessary for $b$ to be closed or even complete. This generality is used in the proof of Theorem \ref{boundarythm}.
\end{rem}

\begin{proof}[Proof of Lemma \ref{nontrivsep}]
Denote by $\beta$ the $b$-side of $c$ and by $\alpha$ the $a$-side of $c$. Denote by $p$ and $q$ the corners of $c$. Orient $a$ in such a way that the subarc $\alpha$ is oriented from $p$ to $q$ and denote by $\vec{a}$ the resulting oriented curve. Consider the consecutive points of intersection of $\vec{a}$ with $\beta$ in the order that they appear along $\vec{a}$, starting with $p$. Denote them by \[p_1=p, p_2=q, p_3,\ldots, p_n.\] For each $i$, $\vec{a}|[p_i,p_{i+1}]\cup \beta|[p_i,p_{i+1}]$ (indices being taken modulo $n$) is a bicorn between $a$ and $b$ which we denote by $c_i$ (in particular $c_1=c$). Note also that each $c_i$ is isotopic into one of the components of $S\setminus c$ (or both in the case of $c=c_1$). Orienting each $c_i$ appropriately, we have \[\sum_{i=1}^n [c_i] = \left[\sum_{i=1}^n c_i\right]= \left[\sum_{i=1}^n \vec{a}|[p_i,p_{i+1}] \right] +\left[\sum_{i=1}^n \beta|[p_i,p_{i+1}]\right]= [\vec{a}] + \left[\sum_{i=1}^n \beta|[p_i,p_{i+1}]\right].\] Note that $\sum \beta|[p_i,p_{i+1}]$ is a singular $1$-cycle contained in $\beta$. In particular, it is null-homologous so that we have simply $\sum_{i=1}^n [c_i]=[\vec{a}]$ and hence some $c_i$ is nonseparating. For this $i$ we have $c_i \in \NS^\perp(c)$, as we previously remarked, and also $c_i \in \B(a,b)$.
\end{proof}

\section{Lamination bicorns}
\label{lambicornsec}

In this section we analyze bicorns with the leaves of a lamination and their properties in homology.

\begin{lem}
\label{bicornhomology}
Let $L$ be a minimal lamination filling an essential finite type subsurface $V\subset S$. Let $a$ be either a simple closed curve or a simple bi-infinite geodesic which intersects $V$ essentially and does not spiral onto $L$. Then homology classes represented by bicorns between $a$ and $L$ generate the first homology $H_1(V)$.
\end{lem}

\begin{proof}
Since $a$ intersects $V$ essentially and does not spiral onto $L$, we have $a\cap L\neq \emptyset$.

Choose a compact subarc $\alpha\subset a$ contained in $V$ such that $\alpha$ intersects $L$ in its interior and not at its endpoints. Orient $\alpha$. Let $p\in \alpha \cap L$ and let $l_p$ be the leaf of $L$ through $p$. Denote by $l_p^+$ the half leaf of $l_p$ based at $p$ and leaving $\alpha$ to the right. Similarly denote by $l_p^-$ the half leaf leaving to the left. By minimality of $L$, both rays $l_p^+$ and $l_p^-$ necessarily intersect $\alpha$ in their interiors. Let $q^r$ be the first point of $l_p^+\cap \alpha$ after $p$ (orienting $l_p^+$ to start from $p$) and denote by $l_p^r$ the subarc $l|[p,q^r]$. Similarly, let $q^\ell$ be the first point on $l_p^-\cap \alpha$ after $p$ and denote by $l_p^\ell$ the subarc $l|[p,q^\ell]$.

Since every half leaf of $L$ intersects $\alpha$, the set $\mathcal{A}=\{l_p^r\}_{p\in L\cap \alpha} \cup \{l_p^\ell\}_{p\in L\cap \alpha}$ is a partition of $L$ into compact arcs. However, note that each arc $\sigma\in \mathcal{A}$ is represented exactly twice as an arc $l_p^r$ or $l_p^\ell$ (once for each of its endpoints). Define an equivalence relation on arcs in $\mathcal{A}$ by $\sigma\sim \tau$ if $\sigma$ and $\tau$ are isotopic through arcs with endpoints on $\operatorname{int}(\alpha)$ and intersecting $\alpha$ nowhere in their interiors.

Note the following fact: if $l_p^r\in \mathcal{A}$ and $q\in L\cap \alpha$ is sufficiently close to $p$ then $l_p^r\sim l_q^r$. If $r$ is replaced by $\ell$ then the fact remains true. The fact holds because if $q\in L\cap \alpha$ is very close to $p$, then $l_p^+$ and $l_q^+$ fellow travel for a very long time. In particular they will fellow travel long enough that the first intersections of $l_p^+$ and $l_q^+$ with $\alpha$ are very close together and $l_p^r$ and $l_q^r$ will stay very close together on their entire lengths.

We claim that there are finitely many $\sim$-classes of arcs in $\mathcal{A}$. To see this, for each $p\in L\cap \alpha$ we may choose an open neighborhood $U_p$ of $p$ in $\alpha$ such that for any $q\in L\cap U_p$ we have $l_p^r\sim l_q^r$ and $l_p^\ell\sim l_q^\ell$. Since $L\cap \alpha$ is compact, we may choose finitely many points $p_1,\ldots, p_k$ of $L\cap \alpha$ such that $U_{p_1}\cup \ldots \cup U_{p_k}$ covers $L\cap \alpha$. Then given $\sigma \in \mathcal{A}$ we have either $\sigma=l_q^r$ or $\sigma=l_q^\ell$ for some $q\in L\cap \alpha$. We have $q\in U_{p_i}$ for some $i$ and hence $\sigma=l_q^r\sim l_{p_i}^r$ in the first case and $\sigma=l_q^\ell \sim l_{p_i}^\ell$ in the second case. Hence every $\sigma \in \mathcal{A}$ is equivalent under $\sim$ to some $l_{p_i}^r$ or $l_{p_i}^\ell$.

Now, choose one representative of each $\sim$-class. Denote by $\Sigma=\{\sigma_1,\ldots,\sigma_n\} \subset \mathcal{A}$ the set of representatives. The union $\alpha \cup \bigcup_{i=1}^n \sigma_i$ is an embedded graph $\Gamma$ in $V$. Since $L$ fills $V$, the components of $V\setminus \Gamma$ are all disks or once-punctured disks. Choosing as a base point $x\in \Gamma$, we thus see that loops contained in $\Gamma$ based at $x$ generate $\pi_1(V,x)$. In fact, we may choose a finite generating set using the arcs $\sigma_i$. For $\sigma_i \in \Sigma$ denote by $s_i$ and $t_i$ the endpoints of $\sigma_i$ on $\alpha$. Consider the loop \[\gamma_i=\alpha|[x,s_i] * \sigma_i * \alpha|[t_i,x]\] where $\alpha|[x,s_i]$ is oriented from $x$ to $s_i$, $\sigma_i$ from $s_i$ to $t_i$, $\alpha|[t_i,x]$ from $t_i$ to $x$, and $*$ denotes concatenation of paths. The loops $\gamma_i$ clearly generate $\pi_1(\Gamma,x)$ and therefore also $\pi_1(V,x)$. Note also that $\gamma_i$ is freely homotopic to the bicorn $c_i=\alpha|[s_i,t_i]\cup \sigma_i$ between $a$ and a leaf of $L$. Now since $\pi_1(V,x)\to H_1(V)$ is surjective, the homology classes $[\gamma_i]$ generate $H_1(V)$. We have $[\gamma_i]=[c_i]$ and $c_i$ is a bicorn between $a$ and $L$.
\end{proof}

\begin{cor}
\label{nonsepbicorn}
Let $S,V,L,$ and $a$ be as in the statement of Lemma \ref{bicornhomology}. Suppose that $V$ contains a curve which is nonseparating in $S$ (possibly separating or inessential in $V$). Then there is a bicorn between $a$ and $L$ which is nonseparating and contained in $V$.
\end{cor}

\begin{proof}
Choose a curve $c$ in $V$ which is nonseparating in $S$. In $H_1(V)$, we have $[c]=\sum_{i=1}^k [c_i]$ where the $c_i$ are bicorns between $a$ and $L$ and all the curves have been oriented. Since $c$ is homologous to the sum of the $c_i$ in $V$, it is also homologous to the sum of the $c_i$ in $\overline{S}$. Thus, we have $[c]=\sum_{i=1}^k [c_i]$ in $H_1(\overline{S})$. Since $[c]\neq 0$ in $H_1(\overline{S})$, we have $[c_i]\neq 0$ in $H_1(\overline{S})$ for some $i$ and hence this $c_i$ is a nonseparating bicorn between $a$ and $L$ which is contained in $V$.
\end{proof}

\begin{cor}
\label{nontrivsepbicorn}
Let $S,V,L,$ and $a$ be as in the statement of Lemma \ref{bicornhomology}. Suppose that $V$ contains a nontrivially separating curve (possibly inessential in $V$) but no nonseparating curve. Then there is a bicorn between $a$ and $L$ which is nontrivially separating in $S$ and contained in $V$.
\end{cor}

\begin{proof}
Since $V$ contains no noseparating curve, we have $\genus(V)=0$.

First we claim that at least two boundary components of $V$ are nontrivially separating. To see this, denote the components of $\partial V$ by $c_1,\ldots,c_n$. Since each $c_i$ is separating, $S\setminus V$ is the union of components $W_1,\ldots,W_n$ where $W_i$ is bounded by $c_i$ for each $i$. If only one of the components $W_i$ has positive genus, then it is easy to see that any simple closed curve $d$ in $V$ is trivially separating, see Figure \ref{trivsepboundarycomponents}. Hence at least two components, say $W_1$ and $W_2$ have positive genus. Therefore we see that $c_1$ and $c_2$ are both nontrivially separating.

\begin{figure}[h]
\centering

\def\svgwidth{0.5\textwidth}
\begingroup%
  \makeatletter%
  \providecommand\color[2][]{%
    \errmessage{(Inkscape) Color is used for the text in Inkscape, but the package 'color.sty' is not loaded}%
    \renewcommand\color[2][]{}%
  }%
  \providecommand\transparent[1]{%
    \errmessage{(Inkscape) Transparency is used (non-zero) for the text in Inkscape, but the package 'transparent.sty' is not loaded}%
    \renewcommand\transparent[1]{}%
  }%
  \providecommand\rotatebox[2]{#2}%
  \newcommand*\fsize{\dimexpr\f@size pt\relax}%
  \newcommand*\lineheight[1]{\fontsize{\fsize}{#1\fsize}\selectfont}%
  \ifx\svgwidth\undefined%
    \setlength{\unitlength}{1525.08591059bp}%
    \ifx\svgscale\undefined%
      \relax%
    \else%
      \setlength{\unitlength}{\unitlength * \real{\svgscale}}%
    \fi%
  \else%
    \setlength{\unitlength}{\svgwidth}%
  \fi%
  \global\let\svgwidth\undefined%
  \global\let\svgscale\undefined%
  \makeatother%
  \begin{picture}(1,1.12488558)%
    \lineheight{1}%
    \setlength\tabcolsep{0pt}%
    \put(0,0){\includegraphics[width=\unitlength,page=1]{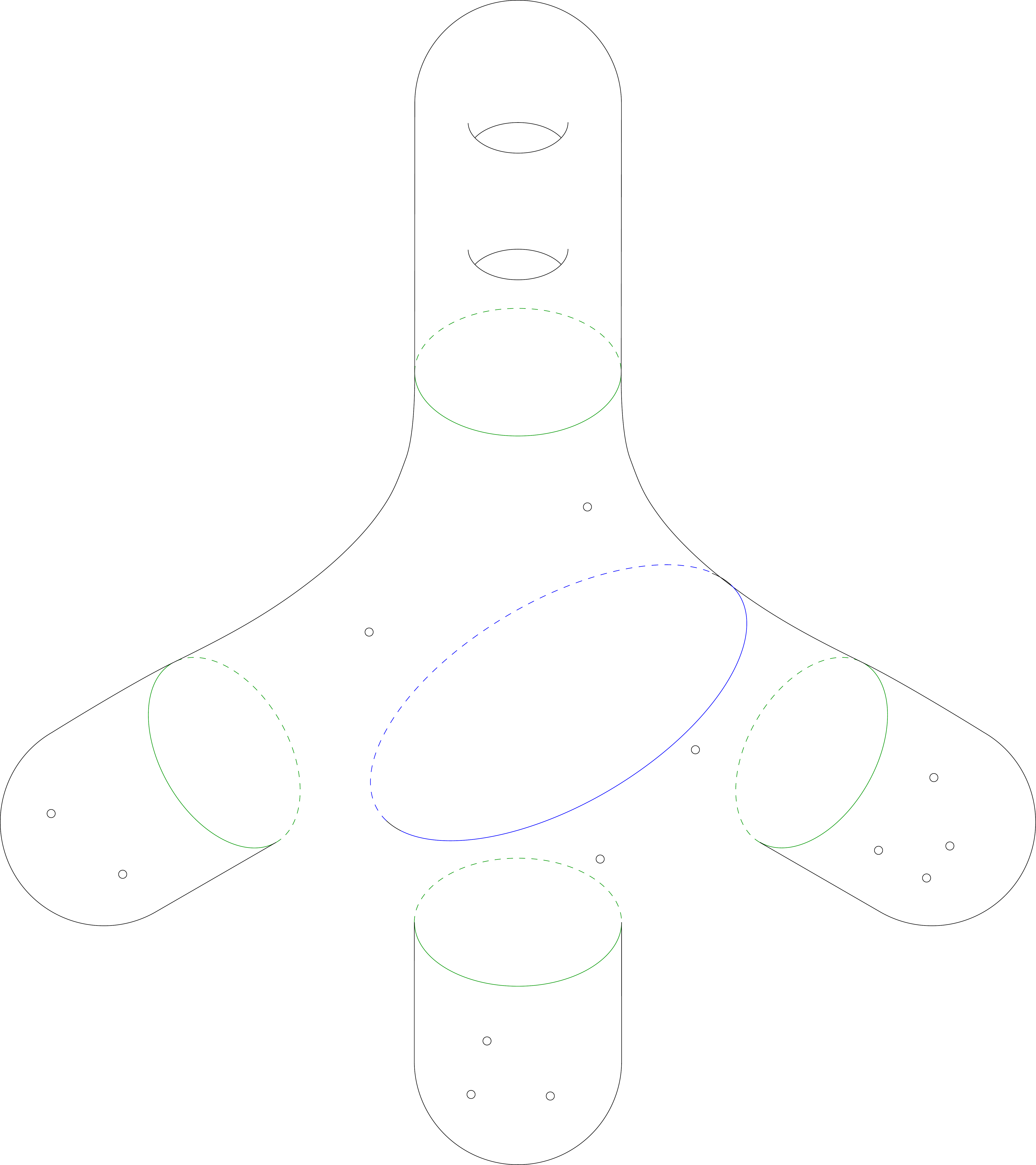}}%
    \put(0.30609995,0.75708867){\color[rgb]{0,0.58823529,0}\makebox(0,0)[lt]{\lineheight{1.25}\smash{\begin{tabular}[t]{l}$c_1$\end{tabular}}}}%
    \put(0.26816299,0.26250294){\color[rgb]{0,0.58823529,0}\makebox(0,0)[lt]{\lineheight{1.25}\smash{\begin{tabular}[t]{l}$c_2$\end{tabular}}}}%
    \put(0.62224138,0.20630004){\color[rgb]{0,0.58823529,0}\makebox(0,0)[lt]{\lineheight{1.25}\smash{\begin{tabular}[t]{l}$c_3$\end{tabular}}}}%
    \put(0.87936975,0.48871971){\color[rgb]{0,0.54901961,0}\makebox(0,0)[lt]{\lineheight{1.25}\smash{\begin{tabular}[t]{l}$c_4$\end{tabular}}}}%
    \put(0.73324216,0.5674038){\color[rgb]{0,0,1}\makebox(0,0)[lt]{\lineheight{1.25}\smash{\begin{tabular}[t]{l}$d$\end{tabular}}}}%
    \put(0,0){\includegraphics[width=\unitlength,page=2]{nontriviallyseparating.pdf}}%
  \end{picture}%
\endgroup%

\caption{At least two components of $S\setminus V$ have positive genus. Otherwise every simple closed curve in $V$ is trivially separating. Small circles denote punctures here.}
\label{trivsepboundarycomponents}
\end{figure}

%
%
%

Denote by $\overline{V}$ the surface $V$ with the punctures of $S$ filled in. We consider the first homology of $\overline{V}$. Denote by $d_1,\ldots,d_k$ ($k\geq 2$) the boundary components of $V$ which are nontrivially separating and by $e_1,\ldots,e_l$ the boundary components of $V$ which are trivially separating. After orienting the curves, the first homology $H_1(\overline{V})$ is freely generated (as an abelian group) by $[d_2],\ldots, [d_k], [e_1],\ldots,[e_l]$ (note that \[[d_1]=[d_2]+\ldots+[d_k]+[e_1]+\ldots+[e_l]\] in $H_1(\overline{V})$ when each  $d_i$ and each $e_j$ is oriented appropriately). In particular, the subspace of $H_1(\overline{V})$ spanned by $[e_1],\ldots, [e_l]$ is a proper subspace. Moreover, a simple closed curve in $V$ is trivially separating in $S$ if and only if its homology class in $H_1(\overline{V})$ is contained in this subspace. Since bicorns between $a$ and $L$ generate $H_1(V)$ and the homomorphism $H_1(V)\to H_1(\overline{V})$ induced by inclusion is surjective, there exists some bicorn $c$ between $a$ and $L$ whose homology class in $H_1(\overline{V})$ is not contained in the span of $[e_1],\ldots,[e_l]$. We have that $c$ is nontrivially separating and contained in $V$ up to isotopy.
\end{proof}

\section{Finite covers}
\label{coversec}

In this section we prove Theorem \ref{coverthm} and discuss related results. The theorem is broken up into three results: Theorem \ref{quasiconvex}, Lemma \ref{infdiam}, and Lemma \ref{notcoarselyequal}.

Let $\pi:\tilde{S}\to S$ be a degree $k<\infty$ cover where $S$ and $\tilde{S}$ have finite positive genus (but possibly infinitely many punctures). Each simple closed curve $c$ on $S$ lifts to a collection of simple closed curves $\tilde{c}\subset \tilde{S}$ (the components of the preimage) for which the restrictions $\pi|\tilde{c}:\tilde{c}\to c$ may be covers of degree greater than one.

\begin{lem}
Let $a\in \NS(S)^0$ and $\tilde{a}$ a lift of $a$ to $\tilde{S}$. Then $\tilde{a}$ is nonseparating.
\end{lem}

\begin{proof}
Endow $S$ with an orientation and endow $\tilde{S}$ with the pullback orientation. Let $\vec{a}$ be $a$ with an orientation and choose an oriented curve $\vec{b}\subset S$ with $\hat{i}(\vec{a},\vec{b})=1$.

There exists a lift $\tilde{b}\subset \tilde{S}$ of $b$ with $i(\tilde{a},\tilde{b})\neq 0$. Denote by $\vec{\tilde{a}}$ and $\vec{\tilde{b}}$ the lifts $\tilde{a}$ and $\tilde{b}$ endowed with the pullback orientations from $\vec{a}$ and $\vec{b}$. If $\vec{\tilde{a}}$ and $\vec{\tilde{b}}$ intersect at a point $p=\tilde{a}(s)=\tilde{b}(t)$ then $\{\vec{\tilde{a}}'(s),\vec{\tilde{b}}'(t)\}$ is a positively oriented basis of $T_p(\tilde{S})$ since $\pi$ is an orientation-preserving diffeomorphism locally at $p$. This proves that $\hat{i}(\vec{\tilde{a}},\vec{\tilde{b}})>0$ and hence $\tilde{a}$ is nonseparating.
\end{proof}

Hence we may define a map $\Phi:\NS(S)^0\to 2^{\NS(\tilde{S})^0}$ (where if $X$ is a set then $2^X$ denotes its power set) by sending a nonseparating simple closed curve $a$ to the finite collection of its lifts (in other words $\Phi(a)$ is the full preimage of $a$). Note that if $a,b\in \NS(S)^0$ and $c$ is a nonseparating bicorn between $a$ and $b$ then $\Phi(c)$ is a multicurve and each curve in $\Phi(c)$ is a nonseparating $2l$-corn between the multicurves $A=\Phi(a)$ and $B=\Phi(b)$, for some $l\leq k$. Thus, $\Phi(\B(a,b))\subset \B_k(\Phi(a),\Phi(b))$.

In all that follows, we conflate $\Phi(\NS(S)^0)$ with $\bigcup \Phi(\NS(S)^0)$, thus considering it as a subset of $\NS(\tilde{S})^0$ (rather than as a set of subsets of $\NS(\tilde{S})^0$).

\begin{thm}
There exists $F=F(k)$ such that the image $\Phi(\NS(S)^0)$ is $F$-quasiconvex in $\NS(\tilde{S})$.
\label{quasiconvex}
\end{thm}

\begin{proof}
Consider $\tilde{a}\in \Phi(a)$ and $\tilde{b}\in \Phi(b)$ and a geodesic $[\tilde{a},\tilde{b}]$. We wish to show that $[\tilde{a},\tilde{b}]$ is uniformly close to $\Phi(\NS(S)^0)$.

Choose a sequence $c_0=a,c_1,\ldots,c_n=b$ with $c_i\in \B(a,b)$ and $i(c_i,c_{i+1})\leq 2$ for all $i$ (see the proof of Corollary \ref{bicornshausclose} for the existence of such a sequence). Choose any sequence of lifts $\widetilde{c_0},\widetilde{c_1},\ldots,\widetilde{c_n}$ with $\widetilde{c_0}=\tilde{a},\widetilde{c_n}=\tilde{b}$. We have $i(\widetilde{c_i},\widetilde{c_{i+1}})\leq 2k$ for all $i$ and hence $d(\widetilde{c_i},\widetilde{c_{i+1}})\leq 4k+1$ for all $i$ by Lemma \ref{intersectionnumber}.

We have $\widetilde{c_i}\in \B_k(\Phi(a),\Phi(b))$ for all $i$ and hence $\widetilde{c_i}$ is $E(k)$-close to $[\Phi(a),\Phi(b)]$ for all $i$ by Corollary \ref{2kcornshausclose}. By the Morse Lemma, $[\Phi(a),\Phi(b)]$ is $M$-close to $[\tilde{a},\tilde{b}]$ where $M$ only depends on the constant of hyperbolicity of $\NS(\tilde{S})$ (and in particular may be taken to be independent of $S$ and $\tilde{S}$ since the nonseparating curve graphs are uniformly hyperbolic). For each $i$ we may choose $d_i\in [\tilde{a},\tilde{b}]$ with $d(\widetilde{c_i},d_i)\leq E(k)+M$. We may take $d_0=\tilde{a}$ and $d_n=\tilde{b}$. We have \[d(d_i,d_{i+1})\leq d(d_i,\widetilde{c_i})+d(\widetilde{c_i},\widetilde{c_{i+1}})+d(\widetilde{c_{i+1}},d_{i+1})\leq 2E(k)+2M+4k +1.\]

The sequence $d_0,d_1,\ldots,d_n$ is a $(2E(k)+2M+4k+1)$-coarse path contained in $[\tilde{a},\tilde{b}]$. Hence for each vertex $d\in [a,b]$ we have $d(d,d_i)\leq 2E(k)+2M+4k+1$ for some $i$ and therefore \[d(d,\widetilde{c_i})\leq d(d,d_i)+d(d_i,\widetilde{c_i})\leq (2E(k)+2M+4k+1)+E(k)+M=3E(k)+3M+4k+1.\] Since $\widetilde{c_i}\in \Phi(\NS(S)^0)$ for all $i$, we may take $F(k)=3E(k)+3M+4k+1$.
\end{proof}

A Theorem of Rafi-Schleimer (\cite{covers}, Theorem 8.1), states that if $\tilde{S}\to S$ is a degree $k<\infty$ cover of finite type surfaces and $\Psi:\mathcal{C}(S)^0\to 2^{\mathcal{C}(\tilde{S})^0}$ is defined analogously to our map $\Phi$, then $\Psi$ is a quasi-isometric embedding.

One may wonder whether our map $\Phi$ is in fact a quasi-isometric embedding, which would make Theorem \ref{quasiconvex} a trivial consequence. However this is typically not the case. See Figure \ref{cover}. In the figure a 4-times punctured genus 3 surface $\tilde{S}$ is a $\Z/2\Z$-cover of a 4-times punctured genus 1 surface $S$. The green curve bounds a one-holed torus $\tilde{V}$ which projects to a witness $V$ for $\NS(S)$. If $\phi$ is a pseudo-Anosov supported on $V$ and $c$ is a nonseparating curve in $V$ then $\{\phi^n(c)\}_{n\in \Z}$ is infinite diameter in $\NS(S)$. However, the lifts $\{\Phi(\phi^n(c))\}_{n\in \Z}$ have diameter two. Of course this argument works whenever a finite type witness for $\NS(S)$ lifts to a subsurface of $\NS(\tilde{S})$ which is not a witness.

\begin{figure}[h]
\centering
\def\svgwidth{0.8\textwidth}
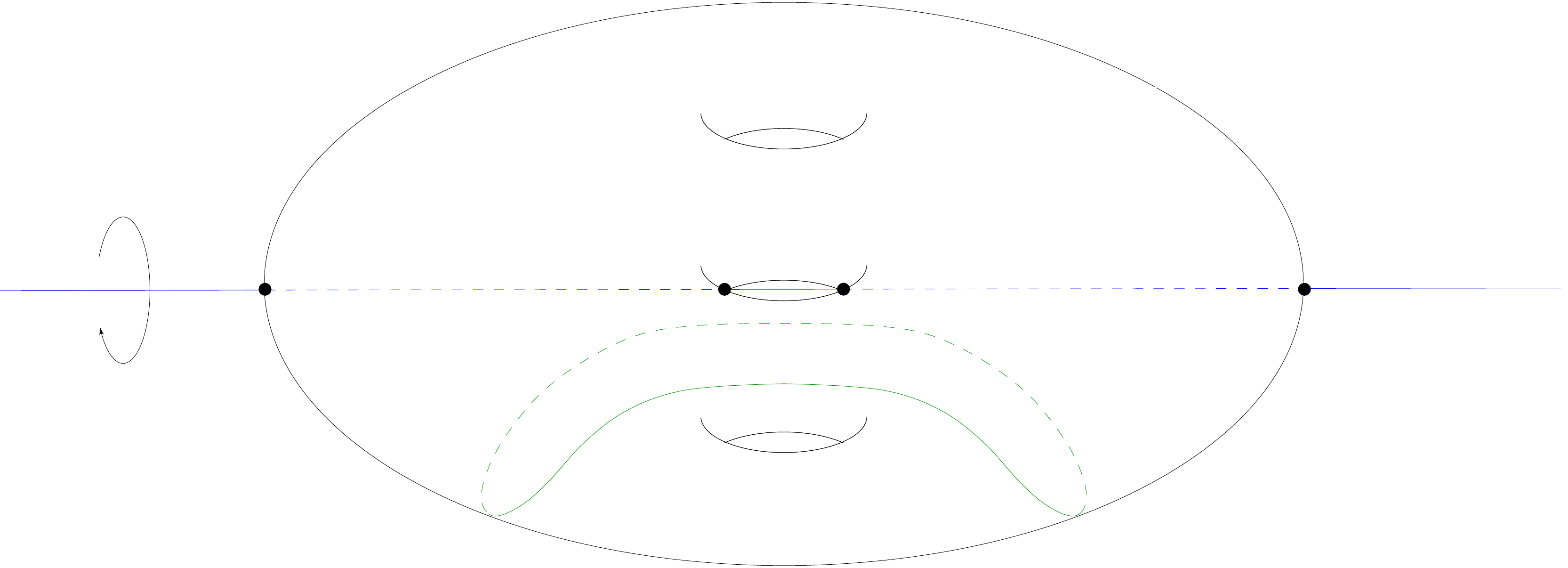

\caption{A $\Z/2\Z$-cover $\tilde{S}$ of the 4-times punctured torus $S$. The green curve bounds a subsurface $\tilde{V}\subset \tilde{S}$ which is not a witness for $\NS(\tilde{S})$, although $\tilde{V}$ is a lift of a subsurface $V\subset S$ which is a witness for $\NS(S)$.}
\label{cover}
\end{figure}

Since $\Phi$ is not a quasi-isometric embedding, the answers to the following questions are not immediately obvious:
\begin{enumerate}[(i)]
\item Is the image of $\Phi$ infinite-diameter?
\item Is the image of $\Phi$ coarsely equal to $\NS(\tilde{S})$?
\end{enumerate}

Recall that given a metric space $X$ and a subspace $Y\subset X$ we say that $Y$ is \textit{coarsely equal} to $X$ if there exists $R>0$ such that $X$ is equal to the $R$-neighborhood of $Y$.

If the answer to (i) were no or the answer to (ii) were yes then Theorem \ref{quasiconvex} would become trivial. Thus we should show that this is not the case.

\begin{lem}
\label{infdiam}
Let $\pi:\tilde{S}\to S$ be a finite degree cover. Then $\Phi(\NS(S)^0)$ has infinite diameter in $\NS(\tilde{S})$.
\end{lem}

\begin{proof}
We find a finite type witness $V\subset S$ such that $V$ lifts to a witness $\tilde{V}$ for $\NS(\tilde{S})$ (that is, $\tilde{V}$ is a connected component of the preimage of $V$). This will prove the statement. For in this case we may choose a pseudo-Anosov $\phi$ supported on $V$. Up to replacing $\phi$ by a finite power, $\phi$ lifts to a pseudo-Anosov $\tilde{\phi}$ supported on $\tilde{V}$. For $c$ a nonseparating curve in $V$ and $\tilde{c}$ a lift of $c$ contained in $\tilde{V}$ we have for each $n$ that $\tilde{\phi}^n(\tilde{c})$ is a lift of $\phi^n(c)$. Since $\tilde{\phi}$ acts loxodromically on $\NS(\tilde{S})$, $\{\tilde{\phi}^n(\tilde{c})\}_{n\in \Z}\subset \Phi(\NS(S))$ is an infinite diameter subset.

Choose an exhaustion of $V_1\subset V_2\subset \ldots$ of $S$ by finite type essential subsurfaces. Let $\widetilde{V_i}$ by the full preimage of $V_i$ in $\tilde{S}$, a possibly disconnected subsurface of $\tilde{S}$. Then $\widetilde{V_1}\subset \widetilde{V_2}\subset \ldots$ is an exhaustion of $\tilde{S}$. Choose a finite type witness $W\subset \tilde{S}$ for $\NS(\tilde{S})$. For large enough $n$, we have $W\subset \widetilde{V_n}$. Since $W$ is connected, it is contained in a connected component $\tilde{V}$ of $\widetilde{V_n}$. Hence this component $\tilde{V}$ (which must in fact be equal to all of $\widetilde{V_n}$) is a witness for $\NS(\tilde{S})$. Setting $V=V_n$ we have the desired witness $V$ for $\NS(S)$.
\end{proof}

\begin{lem}
\label{notcoarselyequal}
Let $\pi:\tilde{S}\to S$ be a finite degree cover. Then $\Phi(\NS(S)^0)$ is not coarsely equal to $\NS(\tilde{S})$.
\end{lem}

\begin{proof}
We show that there is a point of $\partial \NS(\tilde{S})$ which does not lie in the limit set of $\Phi(\NS(S)^0)$. This will prove the statement. Recall that the \textit{limit set} of a subset $Y$ of a hyperbolic metric space $X$ consists of the points of $\partial X$ which are limits of sequences in $Y$.

As in the proof of Lemma \ref{infdiam}, we may choose a finite type witness $V\subset S$ for $\NS(S)$ such that the preimage $\tilde{V}=\pi^{-1}(V)$ is a witness for $\NS(\tilde{S})$. We may choose a lamination $\tilde{L} \in \ELW(\tilde{V})$ which is \textit{not} the lift of a lamination on $S$ (in other words, there are geodesics in the image $\pi(\tilde{L})$ which intersect transversely). Since the inclusion $\NS(\tilde{V})\hookrightarrow \NS(\tilde{S})$ is an isometric embedding, by Corollary \ref{boundarypoint} in the next section, we may choose a sequence $\{c_n\}_{n=1}^\infty \subset \NS(\tilde{V})$ with $[\{c_n\}]\in \partial \NS(\tilde{V})\subset \partial \NS(\tilde{S})$ and $c_n\toCH \tilde{L}$.

Suppose that $[\{c_n\}]=[\{\widetilde{a_n}\}]$ where $a_n\in \NS(S)^0$ and $\widetilde{a_n}$ is a lift of $a_n$ for each $n$. We will consider the geodesic representatives of the $a_n$. For any $n$, $a_n$ intersects $V$ essentially. Moreover, we may construct a compact subsurface $V_0\subset V$ with the property that every simple geodesic which intersects $V$ essentially also intersects $V_0$. To construct $V_0$, it suffices to choose certain open neighborhoods of the punctures of $V$ and half-open collar neighborhoods of the boundary components of $V$ which are all pairwise disjoint. We remove these neighborhoods from $V$ to form $V_0$ (see the proof of Lemma 4.2 in \cite{wwpd} for more details). Hence we have that $a_n$ intersects $V_0$ for each $n$. By compactness of the unit tangent bundle $T^1(V_0)$, we may pass to a subsequence and choose tangent vectors $(p_n,v_n)$ along $a_n$ which converge to $(p,v)\in T^1(V_0)$. The bi-infinite geodesic $l$ in $S$ passing through $(p,v)$ is simple and lifts to a bi-infinite geodesic $\tilde{l}$ in $\tilde{S}$.

We claim that $\tilde{l}$ intersects $\tilde{L}$ transversely. To see this, note that $\tilde{l}$ intersects $\tilde{V}$. Since $\tilde{L}$ fills $\tilde{V}$, $\tilde{l}$ must either intersect $\tilde{L}$ transversely or spiral onto $\tilde{L}$ on at least one end. In the latter case, there is a ray $\tilde{l}_0$ contained in $\tilde{l}$ with $\overline{\tilde{l}_0} \setminus \tilde{l}_0=\tilde{L}$. The projection $\pi(\tilde{l}_0)$ is a ray $l_0$ contained in $l$. Moreover, for any leaf $\tilde{m}$ of $\tilde{L}$, the projection $m=\pi(\tilde{m})$ has a self-intersection at some point $q$ (which follows from minimality of $\tilde{L}$). If $m_0$ is a compact subsegment of $m$ self-intersecting at $q$ then for any choice of $\epsilon>0$, there exists a subsegment of $l_0$ which is $\epsilon$-Hausdorff close to $m_0$. Hence we see that $l_0$ also has a self-intersection and this contradicts that $l$ is simple. Thus $\tilde{l}$ intersects $\tilde{L}$ transversely as claimed.

Since $c_n\toCH \tilde{L}$ and the Hausdorff topology on laminations on $\tilde{V}$ is compact, we may pass to a further subsequence to assume that $c_n\toH \tilde{M}$ where $\tilde{M}$ is a geodesic lamination contained in $\tilde{V}$ with $\tilde{L}\subset \tilde{M}$. By Corollary \ref{nonsepbicorn}, there is a nonseparating bicorn $d$ between $\tilde{l}$ and $\tilde{L}$. Let $\lambda$ and $\mu$ be the $\tilde{l}$-arc of $d$ and the $\tilde{L}$-arc of $d$, respectively. For any $\epsilon>0$ and any $n$ large enough, there is a lift $\widetilde{a_n}'$ of $a_n$ (possibly different from $\widetilde{a_n}$) and a segment of $\widetilde{a_n}'$ of length $1/\epsilon$ which is $\epsilon$-Hausdorff close to a subsegment of $\tilde{l}$ with midpoint contained in $\lambda$. Since, $d(\widetilde{a_n},\widetilde{a_n}')\leq 1$, we still have $[\{\widetilde{a_n}'\}]=[\{c_n\}]$. Also for $n$ large enough, there is a segment of $c_n$ of length $1/\epsilon$ which is $\epsilon$-Hausdorff close to a subsegment of the leaf of $\tilde{L}$ containing $\mu$ so that the midpoint of this subsegment is in $\mu$. Hence we see that $d\in \B(\widetilde{a_n}',c_n)$ for all $n$ large enough. Since $\B(\widetilde{a_n}',c_n)$ is $E(1)$-Hausdorff close to $[\widetilde{a_n}',c_n]$, this proves that $(\widetilde{a_n}'\cdot c_n)_d\leq E(1)+2\delta$ for $n$ sufficiently large, where $\delta$ is the constant of hyperbolicity of $\NS(\tilde{S})$. This contradicts that $(\widetilde{a_n}' \cdot c_n)_d\to \infty$ as $n\to \infty$.
\end{proof}

\section{Infinite bicorn paths}
\label{infpathsec}

In the remaining sections of the paper, we prove Theorem \ref{boundarythm} using the technique of \textit{infinite bicorn paths}. Recall the statement:

\boundarythm*

For the rest of the paper, $S$ is assumed to be \textit{finite type}. Thus, $S$ has finite positive genus, finitely many punctures, and no boundary components. As usual, we will replace curves by their geodesic representatives. In particular, simple closed curves are pairwise in minimal position.

In this and the next section, we define a map $F:\ELW(S)\hookrightarrow \partial \NS(S)$. We ultimately prove that $F$ is a homeomorphism. The hardest part of the proof is showing that $F$ is surjective in Section \ref{surjsec}. For this we apply the tools we developed in Sections \ref{2kcornsec} through \ref{nontrivsepsec}.

The following arguments are adapted from arguments in \cite{phothesis} Chapter 5. They are necessarily more complicated due to the fact that we must frequently verify that certain curves are nonseparating. First we construct inductively an infinite coarse path based on a lamination in $\ELW(S)$. The construction takes as input the following data:
\begin{itemize}
\item a nonseparating curve $a$;
\item a leaf $l$ of a lamination $L\in \ELW(S)$;
\item an orientation $\vec{l}$ on $l$;
\item a point $p\in a \cap l$ (note that such a point exists since $L$ is minimal and fills a witness).
\end{itemize}
The output is a path $\B(a,\vec{l},p)=\{a_0,a_1,\ldots\}$ uniquely determined by the data. The curves $a_i$ are bicorns between $a$ and $l$, constructed to have longer and longer $l$-arcs and shorter and shorter $a$-arcs.

\begin{const}[Construction of $\B(a,\vec{l},p)$: Step 1]
By minimality of $L$, the oriented ray $\vec{l}|[p,\infty)$ intersects $a$ infinitely many times in its interior. Let $p_1$ be the first point of $\vec{l}|(p,\infty)\cap a$ after $p$ along $l$. Orient $a$ to obtain the oriented curve $\vec{a}$ such that $\vec{a}|[p,p_1]$ leaves $p$ to the right of $l$ and $\vec{a}|[p_1,p]$ leaves $p$ to the left of $l$. Both of the curves $a_1'=\vec{a}|[p,p_1]\cup l|[p,p_1]$ and $a_1''=\vec{a}|[p_1,p]\cup l|[p,p_1]$ are simple and we have $[a_1']+[a_1'']=[a]$ when the curves $a_1'$ and $a_1''$ are oriented to agree with $\vec{a}$ along their overlap with $\vec{a}$. Hence one of the curves $a_1'$ or $a_1''$ is nonseparating. If $a_1'$ is nonseparating then we set $a_1=a_1'$ and otherwise we set $a_1=a_1''$. We also set $a_0=a$. We have $i(a_0,a_1)\leq 1$ by construction.
\label{step1}
\end{const}

See Figure \ref{firstbicorn}.

\begin{figure}[h]
\centering

\def\svgwidth{0.5\textwidth}
\begingroup%
  \makeatletter%
  \providecommand\color[2][]{%
    \errmessage{(Inkscape) Color is used for the text in Inkscape, but the package 'color.sty' is not loaded}%
    \renewcommand\color[2][]{}%
  }%
  \providecommand\transparent[1]{%
    \errmessage{(Inkscape) Transparency is used (non-zero) for the text in Inkscape, but the package 'transparent.sty' is not loaded}%
    \renewcommand\transparent[1]{}%
  }%
  \providecommand\rotatebox[2]{#2}%
  \newcommand*\fsize{\dimexpr\f@size pt\relax}%
  \newcommand*\lineheight[1]{\fontsize{\fsize}{#1\fsize}\selectfont}%
  \ifx\svgwidth\undefined%
    \setlength{\unitlength}{675.35402794bp}%
    \ifx\svgscale\undefined%
      \relax%
    \else%
      \setlength{\unitlength}{\unitlength * \real{\svgscale}}%
    \fi%
  \else%
    \setlength{\unitlength}{\svgwidth}%
  \fi%
  \global\let\svgwidth\undefined%
  \global\let\svgscale\undefined%
  \makeatother%
  \begin{picture}(1,0.78132227)%
    \lineheight{1}%
    \setlength\tabcolsep{0pt}%
    \put(0,0){\includegraphics[width=\unitlength,page=1]{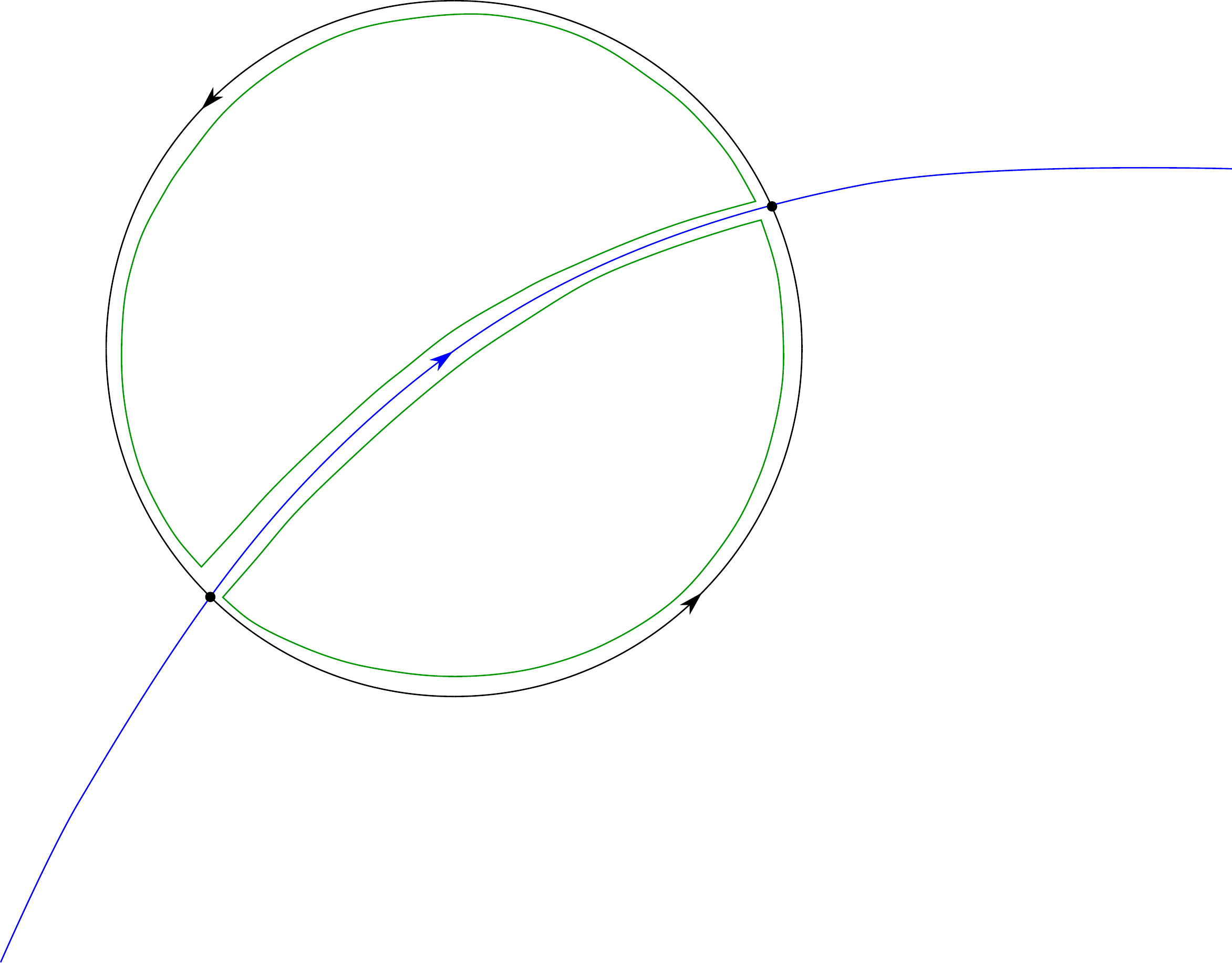}}%
    \put(0.07145002,0.09457065){\color[rgb]{0,0,1}\makebox(0,0)[lt]{\lineheight{1.25}\smash{\begin{tabular}[t]{l}$\vec{l}$\end{tabular}}}}%
    \put(0.56055624,0.23030887){\color[rgb]{0,0,0}\makebox(0,0)[lt]{\lineheight{1.25}\smash{\begin{tabular}[t]{l}$\vec{a}$\end{tabular}}}}%
    \put(0.16792507,0.22245627){\color[rgb]{0,0,0}\makebox(0,0)[lt]{\lineheight{1.25}\smash{\begin{tabular}[t]{l}$p$\end{tabular}}}}%
    \put(0.65142229,0.58367693){\color[rgb]{0,0,0}\makebox(0,0)[lt]{\lineheight{1.25}\smash{\begin{tabular}[t]{l}$p_1$\end{tabular}}}}%
    \put(0.43603604,0.27181561){\color[rgb]{0,0.58823529,0}\makebox(0,0)[lt]{\lineheight{1.25}\smash{\begin{tabular}[t]{l}$a_1'$\end{tabular}}}}%
    \put(0.15895065,0.59601676){\color[rgb]{0,0.58823529,0}\makebox(0,0)[lt]{\lineheight{1.25}\smash{\begin{tabular}[t]{l}$a_1''$\end{tabular}}}}%
  \end{picture}%
\endgroup%

\caption{The first bicorn $a_1$ is defined to be $a_1'$ if $a_1'$ is nonseparating and otherwise it is $a_1''$.}
\label{firstbicorn}
\end{figure}
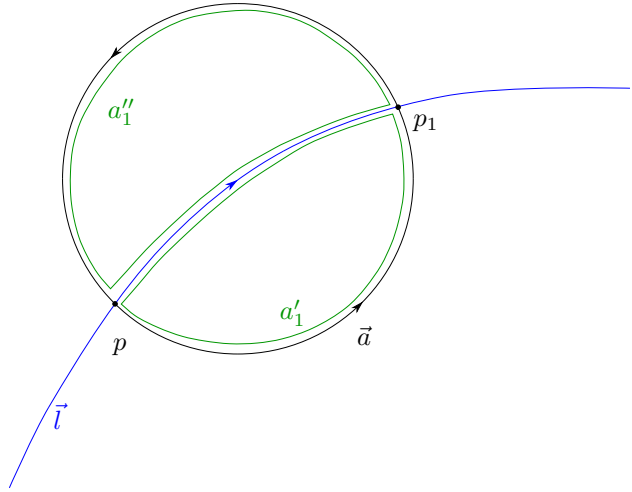

\begin{const}[Construction of $\B(a,\vec{l},p)$: Step $n$]
Now suppose that nonseparating bicorns $a_0,a_1,\ldots,a_n$ have been constructed with $a$-arcs $\alpha_0\supset \alpha_1 \supset \ldots \supset \alpha_n$ and $l$-arcs $\lambda_0\subset \lambda_1 \subset \ldots \subset \lambda_n$. We suppose that $i(a_i,a_{i+1})\leq 2$ for all $i$. Denote by $p_n$ and $q_n$ the endpoints of $\lambda_n$ with $p_n<q_n$ in the orientation on $\vec{l}$. Since $L$ is minimal and fills a witness, both rays $\vec{l}|[q_n,\infty)$ and $\vec{l}|(-\infty,p_n]$ intersect the nonseparating curve $a_n$ in their interiors. Consequently, they must intersect the arc $\alpha_n$. Let $q_{n+1}^*$ be the \textit{first} point of $\vec{l}|(q_n,\infty) \cap \alpha_n$ along $\vec{l}$ and $p_{n+1}^*$ be the \textit{last} point of $\vec{l}|(-\infty,p_n)\cap \alpha_n$ along $\vec{l}$.

We define arcs of $a$ and $l$ and bicorns between $a$ and $l$ by  (see Figure \ref{bicornconstruction}):
\begin{itemize}
\item $\lambda_{n+1}'=\vec{l}|[p_n,q_{n+1}^*]$, $\alpha_{n+1}'=\alpha_n|[p_n,q_{n+1}^*]$, $a_{n+1}'=\alpha_{n+1}'\cup \lambda_{n+1}'$;
\item $\lambda_{n+1}''=\vec{l}|[p_{n+1}^*,q_n]$, $\alpha_{n+1}''=\alpha_n|[p_{n+1}^*,q_n]$, $a_{n+1}''=\alpha_{n+1}''\cup\lambda_{n+1}''$;
\item $\lambda_{n+1}'''=\lambda_{n+1}'\cup \lambda_{n+1}''=\vec{l}|[p_{n+1}^*,q_{n+1}^*]$, $\alpha_{n+1}'''=\alpha_{n+1}'\cap\alpha_{n+1}''=\alpha_n|[p_{n+1}^*,q_{n+1}^*]$, $a_{n+1}'''=\alpha_{n+1}'''\cup \lambda_{n+1}'''$.
\end{itemize}

By the argument of \cite{nonsepbicorns} Claim 3.4, one of $a_{n+1}',a_{n+1}'',$ and $a_{n+1}'''$ is nonseparating. We define $a_{n+1}$ to be the first element of the ordered set $\{a_{n+1}',a_{n+1}'',a_{n+1}'''\}$ which is nonseparating (where the elements are ordered from left to right) and let $\alpha_{n+1}$ and $\lambda_{n+1}$ be the $a$-side and $l$-side of $a_{n+1}$, respectively. Clearly we have $\alpha_{n+1}\subset \alpha_n$ and $\lambda_{n+1}\supset \lambda_n$. Moreover, $i(a_n,a_{n+1})\leq 2$. This completes the inductive step of the construction of the infinite bicorn path $\B(a,\vec{l},p)$. It is a 5-coarse path in $\NS(S)$.
\label{stepn}
\end{const}

\begin{figure}[h]
\centering
\def\svgwidth{0.8\textwidth}
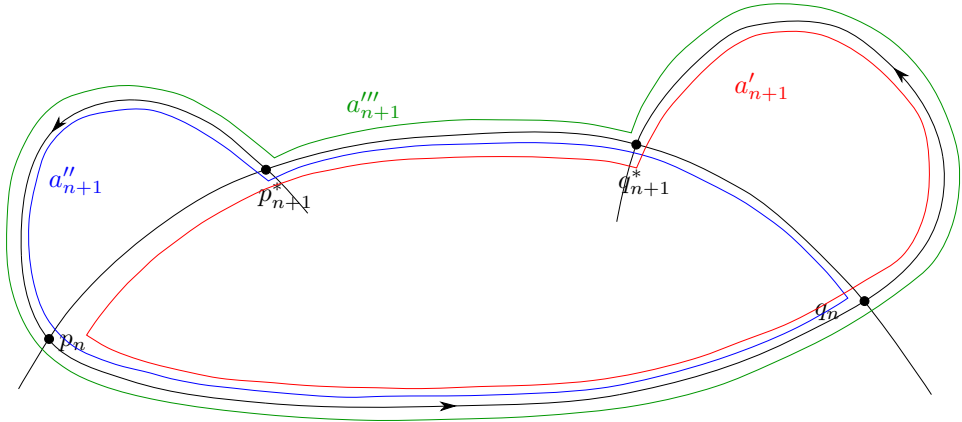

\caption{The construction of the curves $a_{n+1}'$, $a_{n+1}''$, and $a_{n+1}'''$ from $a_n$.}
\label{bicornconstruction}
\end{figure}

Now we equip the geodesic lamination $L$ with a transverse measure. Denote by $\mu \in \ML(S)$ the resulting measured lamination.

\begin{lem}
The transverse measure $i(a_n,\mu)$ converges to 0 
\end{lem}

\begin{proof}
The transverse measure of $a_n$ is at most that of the transversal $\alpha_n$. By compactness $\bigcap_{n=0}^\infty \alpha_n$ is an arc or possibly a point which we call $\tilde{\alpha}$. Since $\mu$ has no atoms, if $\tilde{\alpha}$ happens to be a point, there is nothing to prove. So suppose that $\tilde{\alpha}$ is an arc. If $\tilde{\alpha}$ intersects $\mu$ nowhere in its interior then there is again nothing to prove.

Otherwise, $\tilde{\alpha}$ intersects $\mu$ in a Cantor set of points. Note that the union $\tilde{\lambda}=\bigcup_{n=0}^\infty \lambda_n$ contains one of the rays $\vec{l}|[p,\infty)$ or $\vec{l}|(-\infty,p]$. This follows because there is a lower bound to the lengths of the arcs $\vec{l}|[p_{n+1},p_n]$ and $\vec{l}|[q_n,q_{n+1}]$ by injectivity radius considerations. By minimality, both rays $\vec{l}|[p,\infty)$ and $\vec{l}|(-\infty,p]$ must intersect $\tilde{\alpha}$ in its interior. Suppose without loss of generality that $\tilde{\lambda}$ contains $\vec{l}|[p,\infty)$. Then $\vec{l}|[p,\infty)$ intersects the interior of $\tilde{\alpha}$ at a point $r$. Consequently, for some $n$ we have $q_n<r<q_{n+1}$ in the orientation on $\vec{l}$. But $r\in \alpha_n$ and this contradicts that $q_{n+1}$ is the \textit{first} point of $\vec{l}\cap \alpha_n$ after $q_n$.
\end{proof}

\begin{lem}
Let $\mu\in \ML(S)$ with underlying lamination $L\in\ELW(S)$. If $\{c_n\}\subset \NS(S)^0$ with $i(c_n,\mu)\to 0$ then $c_n\toCH L$.
\label{intchconv}
\end{lem}

\begin{proof}
Suppose that $L$ fills the witness $V\subset S$. Consider a Hausdorff convergent subsequence $c_{n_i}\toH M$ with $M$ a geodesic lamination. Suppose that $L\not\subset M$ and that there exists a leaf $l$ of $M$ intersecting $V$. Then there is a compact arc $\lambda\subset l$ with endpoints in complementary regions of $L$ and intersecting $L$ tranvsersely. Then for all $i$ large enough $c_{n_i}$ contains an arc homotopic to $\lambda$ through a sequence of arcs transverse to $L$. Hence for all such $i$, $i(c_{n_i},\mu)$ is at least as large as the measure of $\lambda$. This contradicts that $i(c_{n_i},\mu)\to 0$. Hence we have that either $L\subset M$ or $M$ is disjoint from $V$. If $M$ is disjoint from $V$, then by compactness of $M$, $c_{n_i}$ is contained in the complement of $V$ for $i$ sufficiently large. But the complement of $V$ consists of a set of punctured disks and this contradicts that $c_{n_i}$ is nonseparating. So we must have $L\subset M$. Since the Hausdorff convergent subsequence $\{c_{n_i}\}$ was arbitrary, this proves that $c_n\toCH L$, as claimed.
\end{proof}

\begin{cor}
Let $a\in \NS(S)^0$, $L\in \ELW(S)$, $\vec{l}$ be a leaf of $L$ endowed with an orientation, and $p\in a\cap \vec{l}$. Write $\B(a,\vec{l},p)=\{a_n\}_{n=0}^\infty$. Then $a_n\toCH L$.
\label{bicornchconv}
\end{cor}

The proof of the following lemma is nearly identical to that of \cite{phothesis} Proposition 5.2.5.

\begin{lem}
Suppose that $L\in \ELW(S)$ and let $\mu$ be a measured lamination with underlying lamination $L$. If $\{c_n\}\subset \NS(S)^0$ with $i(c_n,\mu)\to 0$ then $[\{c_n\}]\in \partial \NS(S)$.
\label{intnumconvto0}
\end{lem}

\begin{proof}
The proof is a variation on Luo's proof of infinite diameter of the curve complex (see \cite{mm} Section 4.3).

Let $V$ be the witness filled by $L$. We have by Lemma \ref{intchconv} that $c_n\toCH L$.  Choose a curve $a\in \NS(S)^0$. If $d(a,c_n)\not\to \infty$ then up to taking a subsequence, we have $d(a,c_n)\leq M$ for all $n$ and some $M\in \N$. Up to taking a further subsequence, there exists $N\in \N$ such that $d_{\NS(S)}(a,c_n)=N$ for all $n$ and we may suppose that $c_n\toH L^N$ where $L^N$ is a lamination with $L\subset L^N$. For each $n$, choose a path \[c_0^n=a, c_1^n,\ldots, c_N^n=c_n\] in $\NS(S)$. Since $c_0^n,\ldots,c_N^n$ is a path in $\NS(S)$ we have that $i(c_{N-1}^n,c_N^n)=0$ for all $n$. Taking a further subsequence we may suppose that for each $i$ there exists a lamination $L^i$ with $c_i^n\toH L^i$.

Since $i(c_{N-1}^n,c_N^n)=0$, $L^{N-1}$ and $L^N$ do not intersect transversely. Hence since $L\subset L^N$ we have that either $L\subset L^{N-1}$ or $L^{N-1}$ does not intersect $V$.  Since $c_{N-1}^n$ is nonseparating, it can never be contained in the complement of $V$ and therefore we have that $L\subset L^{N-1}$. An induction proof yields also that $L\subset L^{N-2}, L\subset L^{N-3}, \ldots, L\subset L^0$. However, we have $c_0^n=a$ for all $n$ so that necessarily $L^0=a$. This is a contradiction.

Now suppose that $\liminf_{m,n\to \infty} (c_m\cdot c_n)_a =k<\infty$. We have $d(a,[c_m,c_n])\leq (c_m \cdot c_n)_a+2\delta$ for all $m,n$, so there are pairs $(m,n)$ with $d(a,[c_m,c_n])\leq k+3\delta$ and $m,n$ both arbitrarily large. By Corollary \ref{2kcornshausclose}, $[c_m,c_n]$ and $\B(c_m,c_n)$ are $E(1)$-Hausdorff close, so we may choose $c_{m,n}\in \B(c_m,c_n)$ with $d(a,c_{m,n})\leq k+3\delta+E(1)$ for all $m,n$. However, we have \[i(c_{m,n},\mu)\leq i(c_m,\mu)+i(c_n,\mu)\] and as long as $m$ and $n$ are sufficiently large, $i(c_m,\mu)$ and $i(c_n,\mu)$ are as close as desired to 0. Thus $\lim_{m,n\to \infty} i(c_{m,n},\mu)=0$. By the work in the previous paragraphs, this implies that $\lim_{m,n\to \infty} d(a,c_{m,n})=\infty$. This contradicts the definition of $c_{m,n}$.

Thus $[\{c_n\}]\in \partial \NS(S)$.
\end{proof}

\begin{cor}
Let $a\in \NS(S)^0$, $L\in \ELW(S)$, $\vec{l}$ be a leaf of $L$ endowed with an orientation, and $p\in a\cap \vec{l}$. Then $[\B(a,\vec{l},p)]\in \partial \NS(S)^0$.
\label{boundarypoint}
\end{cor}

\begin{cor}
Let $a\in \NS(S)^0$, $L\in \ELW(S)$. Let $l,l'$ be leaves of $L$ and $p\in l\cap a$, $p'\in a\cap l'$. Choose orientations on $l$ and $l'$ and let $\vec{l}$ and $\vec{l'}$ be the corresponding oriented geodesics. Then $[\B(a,\vec{l},p)]=[\B(a,\vec{l'},p')]$.
\label{welldef}
\end{cor}

\begin{proof}
Choose a transverse measure on $L$ and define $\mu$ to be the resulting measured lamination. Write $\{a_n\}=\B(a,\vec{l},p)$, $\{a'_n\}=\B(a,\vec{l'},p')$. Define $b_n$ by $b_n=a_{n/2}$ if $n$ is even and $b_n=a'_{(n-1)/2}$ if $n$ is odd. Then $i(b_n,\mu)\to 0$ so $[\{b_n\}]\in \partial \NS(S)$ by Lemma \ref{intnumconvto0}. Hence $[\{a_n\}]=[\{a'_n\}]$.
\end{proof}

\section{The map from $\ELW(S)$ to $\partial \NS(S)$}
\label{mapsec}

In this section we define a map $F:\ELW(S)\to \partial \NS(S)$ and show that it is continuous and injective. It will take more work to show that $F$ is surjective and open. We relegate these proofs to the next section.

Choose $a\in \NS(S)^0$. For $L\in \ELW(S)$ we choose a leaf $l$ of $L$, orient it to obtain the oriented geodesic $\vec{l}$, and choose a point $p\in l\cap a$. We define $F(L)=[\B(a,\vec{l},p)]$. The fact that this is a well-defined map is Corollary \ref{welldef}.

To prove that $F$ is continuous, we need the following lemma:

\begin{lem}
\label{bicornrestrict}
Let $a\in \NS(S)^0$, $L\in \ELW(S)$, $\vec{l}$ a leaf of $L$ endowed with an orientation, and $p\in a\cap \vec{l}$. Write $\B(a,\vec{l},p)=\{a_n\}_{n=0}^\infty$ where $a_0=a$. Then given any $j\in \Z_{\geq 0}$, we have that $a_j\in \B(a,a_k)$ for any $k$ sufficiently large.
\end{lem}

\begin{proof}
We have by Lemma \ref{bicornchconv} that $a_k\toCH L$. Thus, for any $\epsilon>0$, we have that for all $k$ sufficiently large, there exists a subsegment of $a_k$ which is $\epsilon$-Hausdorff close to a subsegment of $l$ of length $1/\epsilon$ centered at $p$. This clearly implies that $a_j\in \B(a,a_k)$ for all such $k$, since $a_j$ is a bicorn between $a$ and $l$.
\end{proof}

\begin{lem}
The map $F:\ELW(S)\to \partial \NS(S)$ is continuous.
\end{lem}

\begin{proof}
Consider $\{L_k\}_{k=1}^\infty\subset \ELW(S)$ with $L_k\toCH L\in \ELW(S)$. Choose a leaf $l\in L$. For $p\in a\cap l$, there exist leaves $l_k\subset L_k$ and points $p^k \in a\cap l_k$ with $(l_k,p^k)\to (l,p)$ locally uniformly near $p$. That is, the tangent lines to $l_k$ at $p^k$ converge to the tangent line to $l$ at $p$.

We may orient $l$ to obtain the oriented geodesics $\vec{l}$ and $\vec{l}_k$ such that $\vec{l}_k\to \vec{l}$ as oriented geodesics. Denote $\B(a,\vec{l},p)=\{a_0,a_1,\ldots\}$ and $\B(a,\vec{l}_k,p^k)=\{a^k_0,a^k_1,\ldots\}$.

Letting $a_1'$ and $a_1''$ be the bicorns constructed in Construction \ref{step1} for $\vec{l},a,$ and $p$ and $a'^k_1$ and $a''^k_1$ be the bicorns constructed in Construction \ref{step1} for $\vec{l}_k,a$, and $p^k$, we have that for all $k$ large enough, $a'^k_1$ is homotopic to $a_1'$ and $a''^k_1$ is homotopic to $a_1''$. Therefore we have $a_1=a^k_1$ for all such $k$.

Suppose for induction that $a_n=a^k_n$ for all $k\geq K(n)$. Recall the notation of Construction \ref{stepn}. We let $\alpha_n$ and $\lambda_n$ be the $a$-side and $l$-side of $a_n$, respectively, with endpoints $p_n$ and $q_n$ such that $p_n<q_n$ in the orientation on $\vec{l}$. We let $p_{n+1}^*$ and $q_{n+1}^*$ be the first points of $l\cap \alpha_n$ before $p_n$ and after $q_n$, respectively. Similarly, define $\alpha_n^k,\lambda_n^k,p_n^k,q_n^k,p_{n+1}^{k*},$ and $q_{n+1}^{k*}$ for $a_n^k$. If $k$ is large enough then $\lambda_n^k$ may be made arbitrarily close to $\lambda_n$ and $p_{n+1}^{k^*}$ and $q_{n+1}^{k^*}$ may be made arbitrarily close to $p_{n+1}^*$ and $q_{n+1}^*$, respectively. Thus if $k$ is large enough, we have that $a_{n+1}'^k$ is homotopic to $a_{n+1}'$ where $a_{n+1}'$ and $a_{n+1}'^k$ are defined as in Construction \ref{stepn}. Similarly $a_{n+1}''^k$ and $a_{n+1}'''^k$ are homotopic to $a_{n+1}''$ and $a_{n+1}'''$, respectively. We have a convention about how to choose $a_{n+1}$ and $a_{n+1}^k$, respectively from a specific ordered set of nonseparating curves. Moreover, by what we just argued there is $K(n+1)>0$ such that this ordered set is the same for $(a,\vec{l},p)$ and $(a,\vec{l}_k,p^k)$ whenever $k\geq K(n+1)$. Therefore we have $a_{n+1}^k=a_{n+1}$ whenever $k\geq K(n+1)$.

By what we just argued, there exists an increasing function $K:\N\to \N$ such that we have $a^k_n=a_n$ whenever $k\geq K(n)$. Hence when $k\geq K(n)$, the coarse paths $\B(a,\vec{l}_k,p^k)$ and $\B(a,\vec{l},p)$ agree for length at least $n$. To see that this implies that $F$ is continuous, we argue as follows.

By the previous paragraph, for any $k\geq K(n)$, we we have that for all $\ell$ and $m$ large enough compared to $n$, that \[a_0=a_0^k, a_1=a_1^k,\ldots,a_n=a_n^k\] all lie in $\B(a,a_\ell)\cap \B(a,a_m^k)$, by Lemma \ref{bicornrestrict}. Note that by Corollary \ref{2kcornshausclose}, $\B(a,a_\ell)$ is $E(1)$-Hausdorff close to $[a,a_\ell]$ and $\B(a,a_m^k)$ is $E(1)$-Hausdorff close to $[a,a_m^k]$. Hence, $[a,a_\ell]$ and $[a,a_m^k]$ have segments of length at least $d(a,a_n)-E(1)$ which are $2E(1)$-Hausdorff close. For a fixed $k\geq K(n)$, this implies a lower bound on $(a_\ell\cdot a_m^k)_a$, independent of $\ell$ and $m$ as long as $\ell,m\gg 0$. This lower bound goes to infinity as $d(a,a_n)\to \infty$. In other words \[\liminf_{\ell,m\to \infty} (a_\ell,a_m^k)\to \infty\] as $k\to \infty$. This proves that $[\B(a,\vec{l}_k,p^p)]\to [\B(a,\vec{l},p)]$ as $k\to \infty$.
\end{proof}

\begin{lem}
\label{injectivity}
The map $F:\ELW(S) \to \partial \NS(S)$ is injective.
\end{lem}

\begin{proof}
We again adapt the proof of Luo of infinite diameter of the curve complex.

Suppose that $L\neq M\in \ELW(S)$ and $F(L)=[\B(a,\vec{l},p)]=[\B(a,\vec{m},q)]=F(M)$ where $l$ is a leaf of $L$, $p\in l\cap a$ and $m$ is a leaf of $M$, $q\in m\cap a$. Denote $\B(a,\vec{l},p)=\{a_0,a_1,\ldots\}$ and $\B(a,\vec{m},q)=\{b_0,b_1,\ldots\}$. Since $\B(a,\vec{l},p)$ and $\B(a,\vec{m},q)$ are coarse paths in $\NS(S)$ defining the same point of $\partial \NS(S)$, they are in fact Hausdorff close. That is, there exists $N\in \N$ such that each $a_n\in \B(a,\vec{l},p)$ is distance $\leq N$ from some $b_k \in \B(a,\vec{m},q)$ and vice versa. Up to passing to a subsequence of $\{a_n\}_{n=0}^\infty$ and replacing $N$ by a possibly smaller non-negative integer, we may suppose without loss of generality that each $a_n$ is distance exactly $N$ from some curve $b_{k_n}\in \B(a,\vec{m},q)$.

For each $n$ we may choose a path $x^0_n=a_n, x^1_n,\ldots,x^N_n=b_{k_n}$. We may pass to a further subsequence to assume that each sequence of curves $\{x^i_n\}_{n=0}^\infty$ Hausdorff converges to $L^i\in \GL(S)$.

By Corollary \ref{bicornchconv}, we have $L\subset L^0$. Since $i(x^0_n,x^1_n)=0$ for all $n$, $L^0$ and $L^1$ do not intersect transversely. Let $V\subset S$ be the witness filled by $L$. Since $L^0\pitchfork L^1=\emptyset$ either $L\subset L^1$ or $L^1$ does not intersect $V$. If $L^1$ does not intersect $V$ then $x^1_n$ is contained in the complement of $V$ for all $n$ large enough. This contradicts that $x^1_n$ is nonseparating. Hence $L\subset L^1$. Using the same argument, we see by induction that $L\subset L^N$. But we also have $M\subset L^N$ and $L\pitchfork M\neq \emptyset$. This is a contradiction.
\end{proof}

\section{Surjectivity}
\label{surjsec}

In this section we finish the proof of Theorem \ref{boundarythm}. Surjectivity of the map $F$ is the hardest part of the proof, and we relegate Subsections \ref{casei} through \ref{caseiii} to elements of the proof of surjectivity. In Subsection \ref{finishsec} we finish the proof that $F$ is a homeomorphism and derive a corollary.

Consider a point $[\{a_n\}]\in \partial \NS(S)$. We wish to show that $[\{a_n\}]$ is equal to $F(L)$ for some $L\in \ELW(S)$. To find the lamination $L$, we first pass to a subsequence to assume that $a_n\toH M$ for $M\in \GL(S)$. We will derive a contradiction to the fact that $[\{a_n\}]\in \partial \NS(S)$ unless $M$ contains an ending lamination $L$ filling a witness for $\NS(S)$. With this ending lamination $L$ in hand, we will finally prove that $F(L)=[\{a_n\}]$.

So assume for contradiction that $M$ contains no such ending lamination $L\in \ELW(S)$. By the structure theorem for geodesic laminations on finite type surfaces (see \cite{CEG}, Theorem I.4.2.8), $M$ is necessarily a union of a finite nonempty set of minimal sublaminations along with a finite, possibly empty set of isolated leaves. There are three possibilities:
\begin{enumerate}[(i)]
\item $M$ contains a minimal sublamination filling a subsurface which is not a witness but contains a nonseparating simple closed curve;
\item $M$ contains a minimal sublamination filling a subsurface containing only separating curves but at least one nontrivially separating curve;
\item every minimal sublamination of $M$ is contained in a punctured disk of $S$ (that is, a planar subsurface of $S$ bounded by a trivially separating curve).
\end{enumerate}

We will deal with Cases (i)-(iii) separately in the next three sections, deriving a contradiction in each case. Case (iii) is substantially more involved than the other two.

\subsection{Case (i)}
\label{casei}

Let $M_0\subset M$ be a minimal sublamination filling a subsurface $V\subset S$ which contains a nonseparating simple closed curve $d$. Since $d(d,a_n)\to \infty$ as $n\to \infty$, if $n$ is large enough, $a_n$ intersects $d$, and in particular $V$, essentially. By Corollary \ref{nonsepbicorn}, for any such $n$ there is a nonseparating bicorn $c_n$ between $a_n$ and $M_0$, contained in $V$ up to isotopy. Moreover, we may choose a nonseparating simple closed curve $e$ disjoint from $V$, since $V$ is not a witness.

Now if $c_n$ is a bicorn between $a_n$ and a leaf $l_n$ of $M_0$, for any $m\gg n$ (i.e. for any $m$ sufficiently large compared to $n$) there is a subarc of $a_m$ which is Hausdorff close to a subarc of $l_n$ containing the $l_n$-side of $c_n$. Hence, for $m$ large enough in comparison to $n$, $c_n\in \B(a_n,a_m)$.

Thus, for any $m$ large enough in comparison to $n$, we have \[d(a,[a_n,a_m])\leq d(a,\B(a_n,a_m))+E(1) \leq d(a,c_n)+E(1) \leq d(a,e)+d(e,c_n)+E(1)\leq d(a,e)+1+E(1).\] Hence we have \[(a_n\cdot a_m)_a \leq d(a,[a_n,a_m])+2\delta \leq d(a,e)+1+E(1)+2\delta.\] Since $n$ may be taken to be arbitrarily large, and $m$ may be taken to be arbitrarily large as well, this  contradicts that $(a_n\cdot a_m)_a\to \infty$ as $n,m\to \infty$.

\subsection{Case (ii)}
\label{caseii}

Let $M_0\subset M$ be a minimal sublamination filling a subsurface $V\subset S$ which contains no nonseparating curve but at least one nontrivially separating curve. Then $V$ must have the following form:
\begin{enumerate}[(i)]
\item $\genus(V)=0$,
\item each boundary component of $V$ bounds a component of $S\setminus V$,
\item at least two components of $S\setminus V$ have positive genus.
\end{enumerate}

Denote by $\NS^\perp(V)$ the subgraph of $\NS(S)$ spanned by nonseparating curves which are disjoint from some \textit{nontrivially separating} curve in $V$.

\begin{figure}[h]
\centering

\def\svgwidth{0.7\textwidth}
\begingroup%
  \makeatletter%
  \providecommand\color[2][]{%
    \errmessage{(Inkscape) Color is used for the text in Inkscape, but the package 'color.sty' is not loaded}%
    \renewcommand\color[2][]{}%
  }%
  \providecommand\transparent[1]{%
    \errmessage{(Inkscape) Transparency is used (non-zero) for the text in Inkscape, but the package 'transparent.sty' is not loaded}%
    \renewcommand\transparent[1]{}%
  }%
  \providecommand\rotatebox[2]{#2}%
  \newcommand*\fsize{\dimexpr\f@size pt\relax}%
  \newcommand*\lineheight[1]{\fontsize{\fsize}{#1\fsize}\selectfont}%
  \ifx\svgwidth\undefined%
    \setlength{\unitlength}{1586.72872709bp}%
    \ifx\svgscale\undefined%
      \relax%
    \else%
      \setlength{\unitlength}{\unitlength * \real{\svgscale}}%
    \fi%
  \else%
    \setlength{\unitlength}{\svgwidth}%
  \fi%
  \global\let\svgwidth\undefined%
  \global\let\svgscale\undefined%
  \makeatother%
  \begin{picture}(1,0.97933967)%
    \lineheight{1}%
    \setlength\tabcolsep{0pt}%
    \put(0,0){\includegraphics[width=\unitlength,page=1]{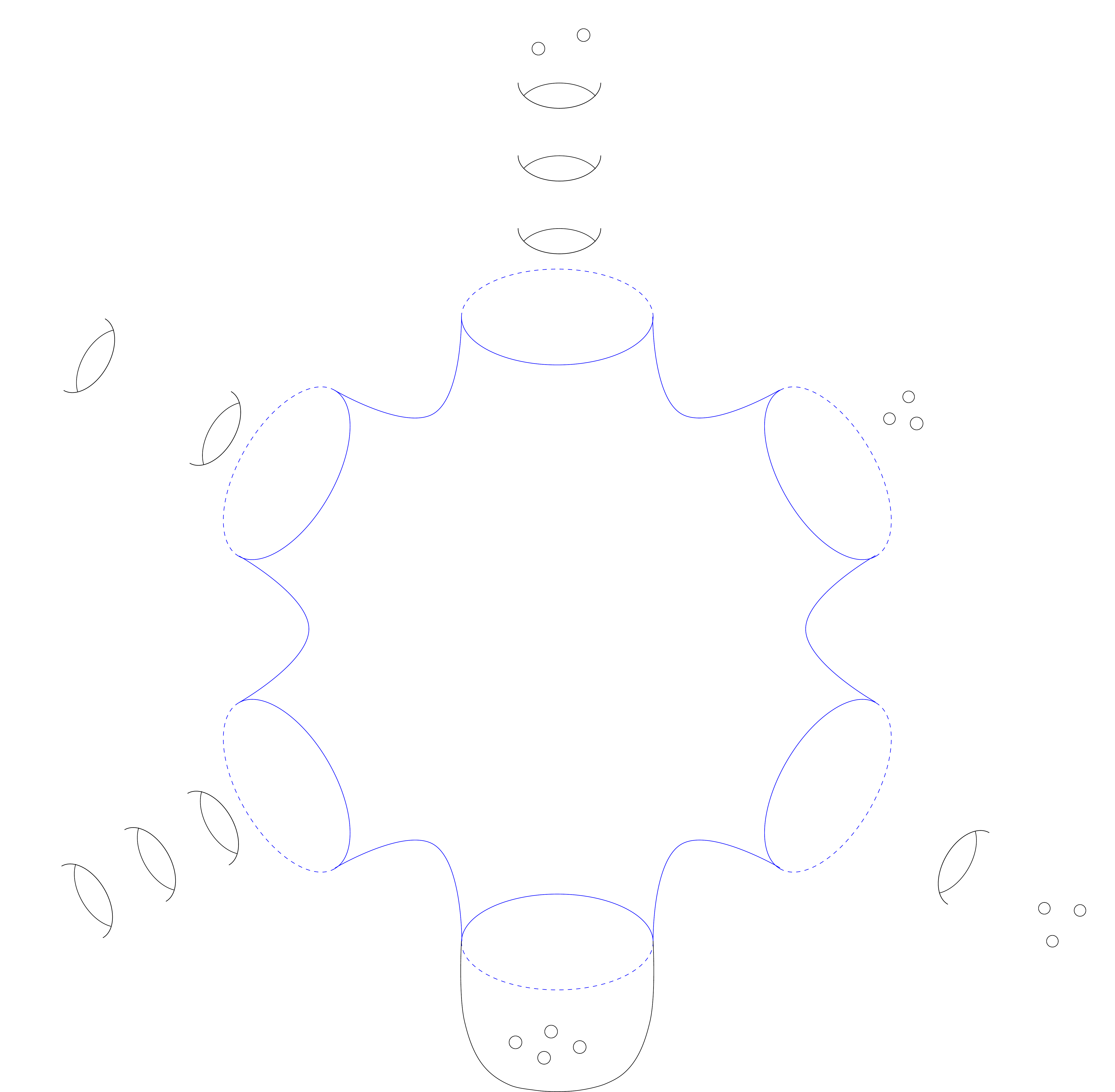}}%
    \put(0.4817294,0.40663361){\color[rgb]{0,0,1}\makebox(0,0)[lt]{\lineheight{1.25}\smash{\begin{tabular}[t]{l}$V$\end{tabular}}}}%
    \put(0,0){\includegraphics[width=\unitlength,page=2]{sepsubsurface.pdf}}%
  \end{picture}%
\endgroup%

\caption{The subsurface $V$ filled by $M_0$.}
\end{figure}

\begin{lem}
\label{nsperpdiam}
The subgraph $\NS^\perp(V)$ has diameter at most four.
\end{lem}

\begin{proof}
Denote by $S_1,\ldots,S_k$ the components of $S\setminus V$. Since at least two of the $S_i$ have positive genus, the subgraph of $\NS(S)$ spanned by nonseparating curves \textit{contained in} some $\NS(S_i)$ has diameter two. Hence it remains to show that every curve in $\NS^\perp(V)$ is disjoint from a nonseparating curve contained in some $\NS(S_i)$.

If $a$ is nonseparating and $c\subset V$ is nontrivially separating and $i(a,c)=0$, then each component of $S\setminus c$ contains some subsurface $S_i$ which contains a nonseparating curve. The curve $a$ is contained in one of the components of $S\setminus c$, so it is disjoint from a nonseparating curve which is contained in one of the $S_i$ in the component of $S\setminus c$ \textit{which does not contain} $a$. This proves the lemma.
\end{proof}

As in the proof of Lemma \ref{nsperpdiam}, we have that the subgraph of $\NS(S)$ consisting of curves disjoint from $V$ has diameter two. Hence if $n$ is sufficiently large, $a_n$ intersects $V$ essentially. By Lemma \ref{nontrivsepbicorn}, there is a nontrivially separating bicorn $c_n$ between $a_n$ and $M_0$ contained in $V$ up to isotopy. 

As in Case (i), we see that if $m$ is sufficiently large compared to $n$, $c_n$ is a bicorn between $a_n$ and $a_m$. By Lemma \ref{nontrivsep}, there is a nonseparating bicorn $d_{n,m}\in \B(a_n,a_m)$ which is disjoint from $c_n$. In particular, $d_{n,m}$ is contained in $\NS^\perp(V)$, a subgraph of $\NS(S)$ of diameter at most four.

Therefore we have for $n\gg 0$ and $m\gg n$ that \[d(a,[a_n,a_m])\leq d(a,\B(a_n,a_m))+E(1)\leq d(a,d_{n,m})+E(1)\leq d(a,\NS^\perp(V))+4+E(1).\] Again, this contradicts that $(a_n\cdot a_m)_a\to \infty$ as $n,m\to \infty$.

%

\subsection{Case (iii)}
\label{caseiii}

The closure of every isolated leaf of $M$ contains a minimal sublamination of $M$. Therefore the set of minimal sublaminations of $M$ is nonempty.

If $L$ is a minimal sublamination of $M$, then it fills a subsurface of $S$ bounded by some nonempty set of trivially separating curves $C(L)=\{c_1,\ldots,c_s\}$, considered up to isotopy (thus, for example if $L$ is a simple closed geodesic then $s=1$). Let $\mathscr{C}$ be the union of the sets $C(L)$ where $L$ ranges over the set of minimal sublaminations of $M$. Then $\mathscr{C}$ carries a partial order $<$ defined by $c<d$ if the punctured disk bounded by $d$ contains $c$. Denote by $\{d_1,\ldots,d_t\}$ the set of elements of $\mathscr{C}$ which are maximal with respect to $<$.

We will consider the geodesic representatives of the $d_i$. The set $\{d_1,\ldots,d_t\}$ has the following properties:
\begin{itemize}
\item Each $d_i$ bounds a punctured disk $D_i$.
\item The subsurface $V=S\setminus \coprod_{i=1}^t D_i$ is a witness for $\NS(S)$.
\item The only leaves of $M$ intersecting $V$ are isolated.
\end{itemize}

We will now isotope the curves $d_i$ slightly (maintaining the notation $d_i$ for the isotoped curves) to curves which may not be geodesic but with a certain desirable property. Note that each $d_i$ is isotopic to a boundary component of exactly one subsurface of $S$ which is filled by a minimal sublamination of $M$. We denote this subsurface by $C_i$ and by $M_i$ the minimal sublamination filling it. If $M_i$ is \textit{not} a simple closed geodesic then we do not isotope $d_i$ at all (so it remains a geodesic). On the other hand, if $M_i$ is a simple closed geodesic then we have $d_i=M_i$. In this case, we take a small collar neighborhood of $d_i=M_i$ with the property that any simple closed geodesic that meets this neighborhood other than $M_i$ itself must intersect $M_i$ (see \cite{borthwick} Lemma 13.14). We isotope $d_i$ to the boundary component of this collar neighborhood which is contained in the component of $S\setminus M_i$ which is \textit{not} a punctured disk. We then have:

\begin{itemize}
\item Any simple closed geodesic which crosses some $d_i$ intersects the minimal lamination $M_i$.
\end{itemize}

Moreover since each $a_n$ is connected and not contained in a punctured disk, we see that:
\begin{itemize}
\item For each $i=1,\ldots, t$, there exists at least one isolated leaf of $M$ crossing $d_i$.
\end{itemize}

Consider the quotient surface $\hat{S}$ defined by collapsing $D_i$ to a point, which we call $p_i$, for $i=1,\ldots, t$. The lamination $M$ projects to a finite graph $\Gamma$ embedded in $\hat{S}$. The graph $\Gamma$ has as vertices the points $p_i$ and as edges the images of the isolated leaves of $M$ that intersect $V$. We have $\genus(\hat{S})=\genus(S)$.

We claim that $\Gamma$ fills $\hat{W}$, a subsurface of $\hat{S}$ of full genus (in the sense that $\Gamma$ intersects every essential simple closed curve in $\hat{W}$). For otherwise there is a nonseparating simple closed curve $\hat{c}$ in $\hat{S}$ disjoint from $\Gamma$. The curve $\hat{c}$ lifts to a nonseparating curve $c$ in $S$ with $c \cap M=\emptyset$. Therefore we have $i(a_n,c)=0$ for all large enough $n$, which contradicts that $a_n$ converges to a point of $\partial \NS(S)$.

We now claim that the first homology of $\hat{W}$ is generated by the homology classes of simple closed curves contained in $\Gamma$. By the fact that $\Gamma$ fills $\hat{W}$, we see that each component of $\hat{W}\setminus \Gamma$ is a disk or annulus. Thus, any closed curve in $\hat{W}$ is freely homotopic (and thus homologous) to a curve in $\Gamma$. This proves that the natural map $H_1(\Gamma)\to H_1(\hat{W})$ induced by inclusion is a surjection. Now, the first homology of any finite graph is generated by homology classes of simple closed curves (as can be seen by taking a spanning tree, and a cycle through each edge not contained in the spanning tree). This shows that $H_1(\Gamma)$, and therefore also $H_1(\hat{W})$, is generated by homology classes of simple closed curves contained in $\Gamma$, as claimed. By Fact \ref{homfact} and surjectivity of the map $H_1(\hat{W})\to H_1(\overline{\hat{W}})$ (where $\overline{\hat{W}}$ denotes $\hat{W}$ with the punctures filled in), this proves that there is a simple closed curve $\hat{c}$ in $\Gamma$ which is nonseparating in $\hat{S}$. We denote $\hat{c}=\widehat{l_1} \cup \ldots \cup \widehat{l_k}$ where the $\widehat{l_i}$ are distinct edges of $\Gamma$, occurring in order along $\hat{c}$. Denote also by $\hat{e}$ a simple closed curve in $\hat{S}$ intersecting $\hat{c}$ exactly once, disjoint from the points $p_i$. We may lift $\hat{e}$ to a simple closed curve $e$ in $S$.

The $\widehat{l_i}$ are images of isolated leaves $l_i$ of $M$. Denote by $D'_i$ and $D'_{i+1}$ the disks containing the ends of $l_i$, indices being taken modulo $k$. Since the curve $\hat{c}$ is simple in $\Gamma$, no $D_j$ is equal to $D'_i$ for multiple values of $i$. If we have $D'_i=D_j$ then we denote by $d'_i$ the boundary component $d_j$ and by $M'_i$ the sublamination $M_j$. The geodesic $l_i$ intersects the boundary components $d'_i$ and $d'_{i+1}$ in one point each. Denote by $\lambda_i$ the subarc of $l_i$ which goes between $d'_i$ and $d'_{i+1}$. For large enough $n$, $a_n$ has a subarc $\alpha_n^i$ with endpoints on $d'_i$ and $d'_{i+1}$ such that $\lambda_i$ and $\alpha_n^i$ are homotopic through a sequence of arcs with endpoints on $d'_i$ and $d'_{i+1}$. We may extend $\alpha^i_n$ to a longer arc $\alpha'^i_n$ with endpoints in $D'_i$ and $D'_{i+1}$ which intersects $M'_i$ and $M'_{i+1}$ in its interior. See Figure \ref{constructing2kcorn}.

Now, by minimality of the lamination $M'_i$, any leaf of $M'_i$ intersects both $\alpha'^i_n$ and $\alpha'^{i-1}_n$ (indices still being taken modulo $k$). Consider such a leaf $l$ of $M'_i$ and an arc $\lambda$ of $l$ intersecting $\alpha'^i_n$ and $\alpha'^{i-1}_n$ at its endpoints and nowhere in its interior. For any $m$ sufficiently large, $a_m$ contains an arc which is extremely close to the leaf $l$ along a long segment containing $\lambda$. In particular, for $m$ large enough compared to $n$, there exists an arc $\beta^i_m$ of $a_m$ contained in $D'_i$ and intersecting $\alpha'^i_n$ and $\alpha'^{i-1}_n$ at its endpoints and nowhere in its interior. Having defined the arcs $\beta^i_m$, we may remove segments at the ends of the arcs $\alpha'^j_n$ (in particular segments contained in the disks $D'_j$ and $D'_{j+1}$) so that the resulting arc $\alpha''^j_{n,m}$ intersects $\beta^j_m$ and $\beta^{j+1}_m$ at its endpoints and nowhere in its interior. Again, see Figure \ref{constructing2kcorn}.

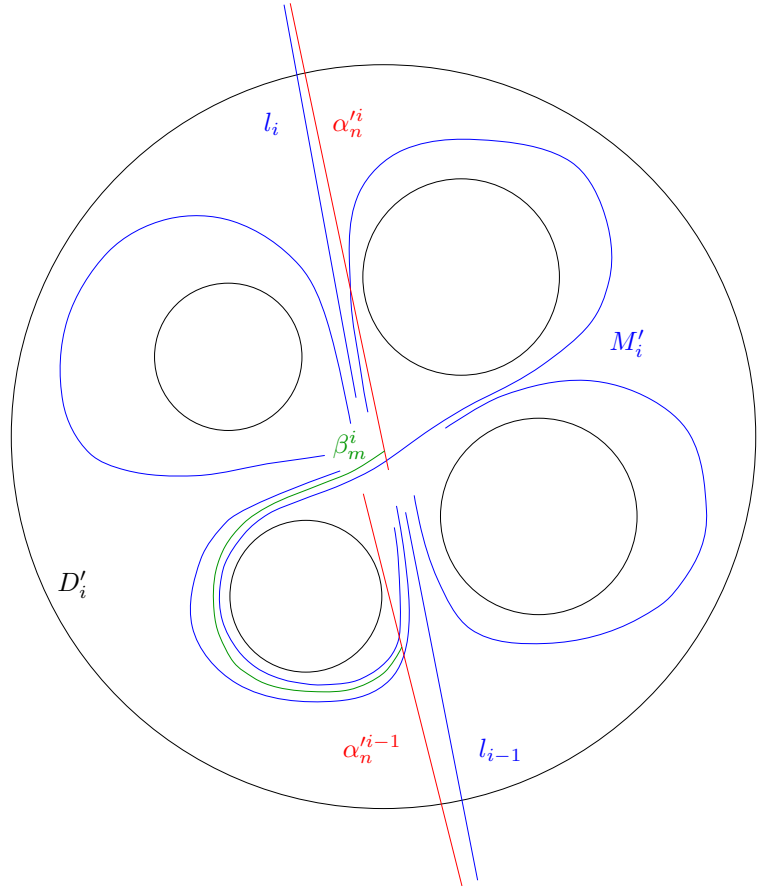
\begin{figure}[h]
\centering
\def\svgwidth{0.6\textwidth}
\begingroup%
  \makeatletter%
  \providecommand\color[2][]{%
    \errmessage{(Inkscape) Color is used for the text in Inkscape, but the package 'color.sty' is not loaded}%
    \renewcommand\color[2][]{}%
  }%
  \providecommand\transparent[1]{%
    \errmessage{(Inkscape) Transparency is used (non-zero) for the text in Inkscape, but the package 'transparent.sty' is not loaded}%
    \renewcommand\transparent[1]{}%
  }%
  \providecommand\rotatebox[2]{#2}%
  \newcommand*\fsize{\dimexpr\f@size pt\relax}%
  \newcommand*\lineheight[1]{\fontsize{\fsize}{#1\fsize}\selectfont}%
  \ifx\svgwidth\undefined%
    \setlength{\unitlength}{714.75121686bp}%
    \ifx\svgscale\undefined%
      \relax%
    \else%
      \setlength{\unitlength}{\unitlength * \real{\svgscale}}%
    \fi%
  \else%
    \setlength{\unitlength}{\svgwidth}%
  \fi%
  \global\let\svgwidth\undefined%
  \global\let\svgscale\undefined%
  \makeatother%
  \begin{picture}(1,1.18572649)%
    \lineheight{1}%
    \setlength\tabcolsep{0pt}%
    \put(0,0){\includegraphics[width=\unitlength,page=1]{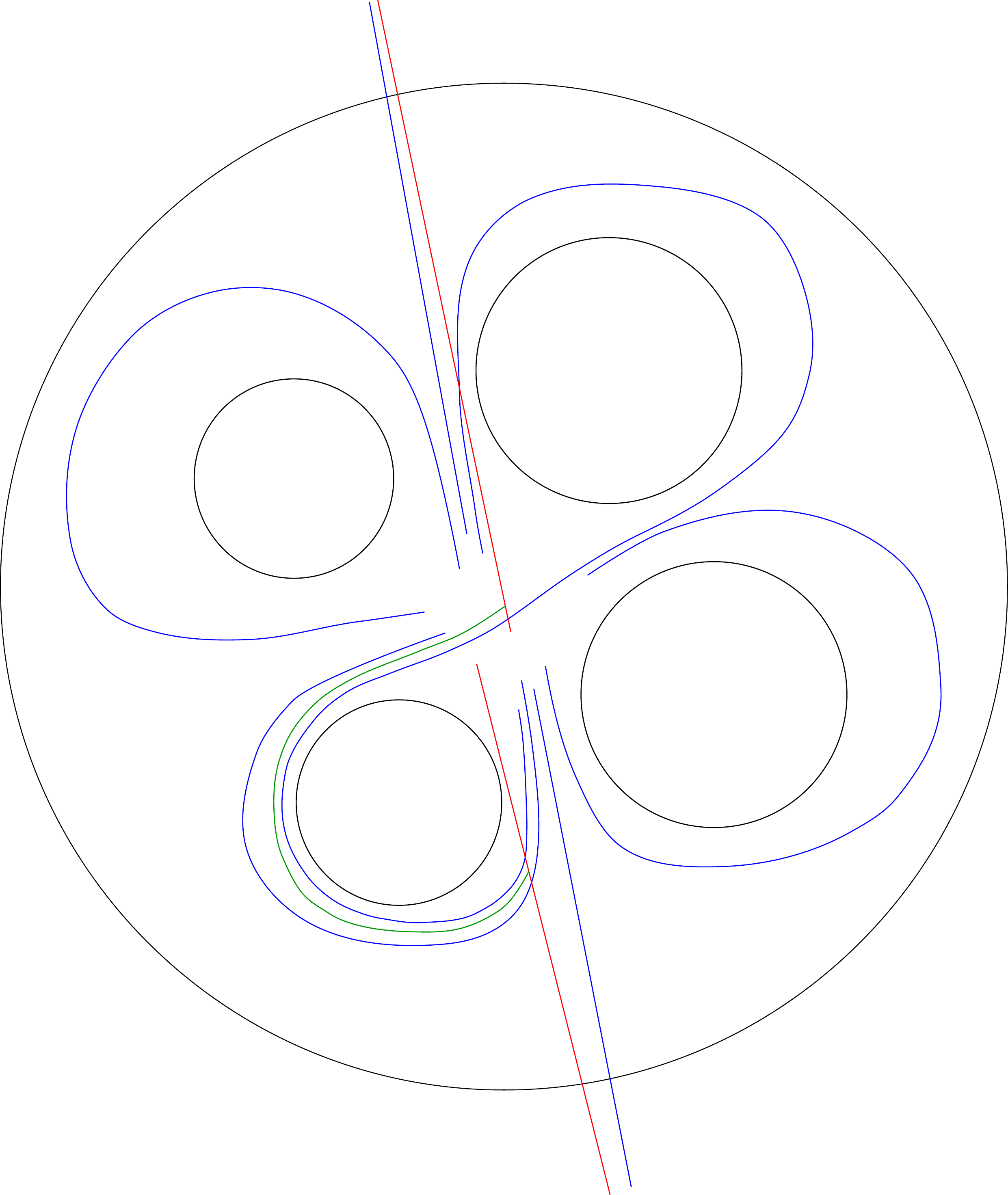}}%
    \put(0.06228529,0.39676935){\color[rgb]{0,0,0}\makebox(0,0)[lt]{\lineheight{1.25}\smash{\begin{tabular}[t]{l}$D_i'$\end{tabular}}}}%
    \put(0.80310245,0.7199587){\color[rgb]{0,0,1}\makebox(0,0)[lt]{\lineheight{1.25}\smash{\begin{tabular}[t]{l}$M'_i$\end{tabular}}}}%
    \put(0.33930478,1.01376723){\color[rgb]{0,0,1}\makebox(0,0)[lt]{\lineheight{1.25}\smash{\begin{tabular}[t]{l}$l_i$\end{tabular}}}}%
    \put(0.43164457,1.01428984){\color[rgb]{1,0,0}\makebox(0,0)[lt]{\lineheight{1.25}\smash{\begin{tabular}[t]{l}$\alpha'^i_n$\end{tabular}}}}%
    \put(0.62786662,0.17378963){\color[rgb]{0,0,1}\makebox(0,0)[lt]{\lineheight{1.25}\smash{\begin{tabular}[t]{l}$l_{i-1}$\end{tabular}}}}%
    \put(0.44528565,0.17378963){\color[rgb]{1,0,0}\makebox(0,0)[lt]{\lineheight{1.25}\smash{\begin{tabular}[t]{l}$\alpha'^{i-1}_n$\end{tabular}}}}%
    \put(0.43178121,0.58232262){\color[rgb]{0,0.58823529,0}\makebox(0,0)[lt]{\lineheight{1.25}\smash{\begin{tabular}[t]{l}$\beta_m^i$\end{tabular}}}}%
  \end{picture}%
\endgroup%

\caption{The construction of the arcs $\alpha'^i_n$ of $a_n$ and $\beta^i_m$ of $a_m$, for $m\gg n$. Such arcs are concatenated to form a nonseparating $2k$-corn $c_{n,m}$.}
\label{constructing2kcorn}
\end{figure}

The union \[c_{n,m}=\beta^1_m \cup \alpha''^1_{n,m} \cup \beta^2_m \cup \alpha''^2_{n,m}\cup \ldots \cup \beta^k_m \cup \alpha''^k_{n,m}\] is a $2k$-corn between $a_n$ and $a_m$. Note that exactly one of the arcs $\lambda_i$ intersects the curve $e$ exactly once. Hence we have $i(c_{n,m},e)=1$. This proves both that $c_{n,m}\in \B_k(a_n,a_m)$ and that $d(c_{n,m},e)\leq 3$. Finally, note that $k$ is bounded by the number of punctures of $S$. Hence by Corollary \ref{2kcornshausclose}, $c_{n,m}$ is $E(P)$-close to $[a_n,a_m]$ where $P$ is the number of punctures of $S$.

Thus for any $n$ sufficiently large and any $m$ sufficiently large compared to $n$, we have \[d(a,[a_n,a_m])\leq d(a,\B_P(a_n,a_m))+E(P)\leq d(a,c_{n,m})+E(P)\leq d(a,e)+d(e,c_{n,m})+E(P)\leq d(a,e)+3+E(P).\] Once again this contradicts that $(a_n\cdot a_m)_a\to \infty$ as $n,m\to \infty$.


\subsection{Finishing the proof}
\label{finishsec}

In this section we finish the proof of Theorem \ref{boundarythm} by showing that the map $F$ is a homeomorphism.

Recall that we started off with $[\{a_n\}]\in \partial \NS(S)$ and passed to a subsequence to assume that $a_n\toH M$ where $M \in \GL(S)$. By the work in Sections \ref{casei}-\ref{caseiii}, we find $L\subset M$ such that $L\in \ELW(S)$.

The proof of the following lemma is essentially identical to the last two paragraphs of the proof of Lemma 5.2.9 in \cite{phothesis}.

\begin{lem}
\label{surjectivity}
We have $F(L)=[\{a_n\}]$.
\end{lem}

\begin{proof}
Let $\vec{l}$ be a leaf of $L$ endowed with an orientation, $a\in \NS(S)^0$, $p\in a\cap \vec{l}$ and $\B(a,\vec{l},p)=\{b_n\}_{n=0}^\infty$. We wish to show that $[\{a_n\}]=[\{b_n\}]$.

If $[\{a_n\}]\neq [\{b_n\}]$ then there exists $k>0$ such that $(a_n\cdot b_m)_a\leq k$ for all $n,m$. Hence we also have $d(a,[a_n,b_m])\leq k+2\delta$ for all $n,m$.

Given $b_m\in \B(a,\vec{l},p)$, there exists $n_m$ large enough that $a_{n_m}$ is close to $l$ on a long enough segment centered at $p$ that $b_m \in \B(a,a_{n_m})$. There is $z\in [a,a_{n_m}]$ which is $E(1)$-close to $b_m$ by Corollary \ref{2kcornshausclose}.

\begin{figure}[h]
\centering

\def\svgwidth{0.7\textwidth}
\begingroup%
  \makeatletter%
  \providecommand\color[2][]{%
    \errmessage{(Inkscape) Color is used for the text in Inkscape, but the package 'color.sty' is not loaded}%
    \renewcommand\color[2][]{}%
  }%
  \providecommand\transparent[1]{%
    \errmessage{(Inkscape) Transparency is used (non-zero) for the text in Inkscape, but the package 'transparent.sty' is not loaded}%
    \renewcommand\transparent[1]{}%
  }%
  \providecommand\rotatebox[2]{#2}%
  \newcommand*\fsize{\dimexpr\f@size pt\relax}%
  \newcommand*\lineheight[1]{\fontsize{\fsize}{#1\fsize}\selectfont}%
  \ifx\svgwidth\undefined%
    \setlength{\unitlength}{1211.22450941bp}%
    \ifx\svgscale\undefined%
      \relax%
    \else%
      \setlength{\unitlength}{\unitlength * \real{\svgscale}}%
    \fi%
  \else%
    \setlength{\unitlength}{\svgwidth}%
  \fi%
  \global\let\svgwidth\undefined%
  \global\let\svgscale\undefined%
  \makeatother%
  \begin{picture}(1,0.22576902)%
    \lineheight{1}%
    \setlength\tabcolsep{0pt}%
    \put(0,0){\includegraphics[width=\unitlength,page=1]{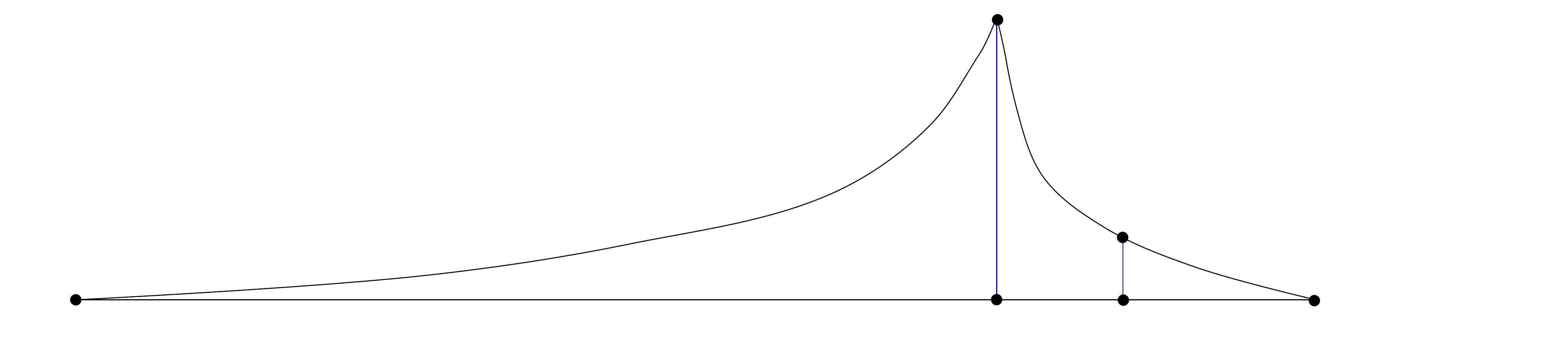}}%
    \put(0.64801436,0.20670902){\color[rgb]{0,0,0}\makebox(0,0)[lt]{\lineheight{1.25}\smash{\begin{tabular}[t]{l}$b_m$\end{tabular}}}}%
    \put(0.8525511,0.02612857){\color[rgb]{0,0,0}\makebox(0,0)[lt]{\lineheight{1.25}\smash{\begin{tabular}[t]{l}$a_{n_m}$\end{tabular}}}}%
    \put(-0.00187456,0.02594113){\color[rgb]{0,0,0}\makebox(0,0)[lt]{\lineheight{1.25}\smash{\begin{tabular}[t]{l}$a$\end{tabular}}}}%
    \put(0.62924952,0.00467429){\color[rgb]{0,0,0}\makebox(0,0)[lt]{\lineheight{1.25}\smash{\begin{tabular}[t]{l}$z$\end{tabular}}}}%
    \put(0.7249502,0.08536308){\color[rgb]{0,0,0}\makebox(0,0)[lt]{\lineheight{1.25}\smash{\begin{tabular}[t]{l}$x$\end{tabular}}}}%
    \put(0.73558361,0.00498269){\color[rgb]{0,0,0}\makebox(0,0)[lt]{\lineheight{1.25}\smash{\begin{tabular}[t]{l}$y$\end{tabular}}}}%
  \end{picture}%
\endgroup%

\caption{The thin triangle showing that $(b_m\cdot a_{n_m})_a$ is large.}
\label{thintriangle}
\end{figure}
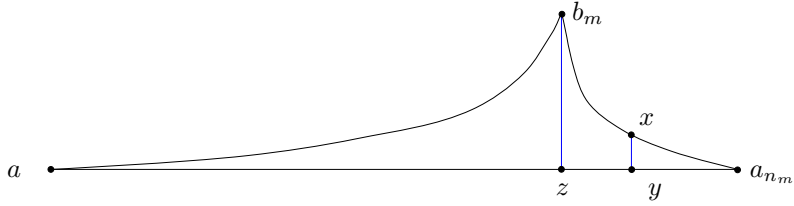

Let $x\in [b_m,a_{n_m}]$. We wish to show that $x$ is uniformly far from $a$, so that $(b_m\cdot a_{n_m})_a$ is large. By $\delta$-thinness of the triangle $[b_m,a_{n_m}]\cup [a_{n_m},z]\cup [z,b_m]$, there exists $y\in [z,a_{n_m}]$ with $d(x,y)\leq E(1)+\delta$. Hence, \[d(x,a)\geq d(y,a)-E(1)-\delta\geq d(z,a)-E(1)-\delta\geq d(b_m,a)-2E(1)-\delta.\] See Figure \ref{thintriangle}. However, $d(b_m,a)\to \infty$ as $m\to \infty$. In particular, choosing $m$ large enough, this contradicts that $(a_{n_m} \cdot b_m)_a\leq k$.
\end{proof}

This proves that $F$ is surjective.

We also have the following corollary of our methods:

\begin{cor}
\label{boundarychconv}
Suppose that $\{a_n\}_{n=1}^\infty \subset \NS(S)^0$. Then $a_n$ converges to a point $F(L)$ in $\partial \NS(S)$ if and only if $a_n\toCH L$.
\end{cor}

\begin{proof}
Suppose that $[\{a_n\}]\in \partial \NS(S)$ and $[\{a_n\}]=F(L)$. We wish to show that $a_n\toCH L$. Consider a Hausdorff convergent subsequence $\{a_{n_i}\}$ of $\{a_n\}$ with $a_{n_i}\toH M\in \GL(S)$. By the work in Sections \ref{casei}-\ref{caseiii}, $M$ contains a sublamination $L_0\in \EL(S)$. By Lemma \ref{surjectivity}, we have $F(L_0)=[\{a_{n_i}\}_i]=[\{a_n\}_n]$. Thus $F(L_0)=F(L)$. By Lemma \ref{injectivity}, $L_0=L$. Since the Hausdorff convergent subsequence $\{a_{n_i}\}$ was arbitrary, this proves that $a_n\toCH L$, as claimed.

On the other hand, suppose that $a_n\toCH L$. We wish to show that $[\{a_n\}]=F(L)$. Choose a leaf $\vec{l}$ of $L$ endowed with an orientation, $a\in \NS(S)^0$, and a point $p\in a\cap \vec{l}$. Denote the infinite bicorn path $\B(a,\vec{l},p)$ by $\{b_n\}$.

We claim that for any $m$, $b_m$ is a bicorn between $a_n$ and $a$ for any $n$ sufficiently large. To see this, note that, given $\epsilon>0$, for $n$ sufficiently large $a_n$ contains a subsegment which is $\epsilon$-Hausdorff close to a subsegment of $l$ of radius $1/\epsilon$ centered at $p$. For if this were not the case, we could choose $\epsilon>0$ and a subsequence $\{a_{n_i}\}$ such that $a_{n_i}$ contains no $\epsilon$-Hausdorff close subsegment to that subsegment of $l$ of radius $1/\epsilon$ for any $i$. Choosing a further subsequence to assume that $\{a_{n_i}\}$ Hausdorff converges, it is clear that $a_{n_i}$ cannot Hausdorff converge to any lamination containing $L$. This contradicts that $a_n\toCH L$.

So given $m$, choose $N_m$ large enough that $b_m$ is a bicorn between $a_n$ and $a$ for every $n\geq N_m$. By a thin triangle argument analogous to that of Lemma \ref{surjectivity}, we have that \[(a_n \cdot b_m)_a\geq d(a,b_m)-2E(1)-3\delta\] for all $n\geq N_m$. Consequently if $n,n'\geq N_m$ we have \[(a_n \cdot a_{n'})_a \geq \min\{ (a_n \cdot b_m)_a, (a_{n'}\cdot b_m)_a\}-5\delta \geq d(a,b_m)-2E(1)-8\delta.\] This proves that $\liminf_{n,n'\to \infty} (a_n\cdot a_{n'})_a=\infty$ so that $[\{a_n\}]\in \partial \NS(S)$. Since $a_n\toCH L$, we have by Lemma \ref{surjectivity} that $[\{a_n\}]=F(L)$.
\end{proof}

So far we have proved that $F$ is a continuous bijection. To complete the proof that $F$ is a homeomorphism we must show the following:

\begin{lem}
The map $F:\ELW(S)\to \partial \NS(S)$ is open.
\end{lem}

\begin{proof}
Let $\{L_k\}_{k=1}^\infty\subset \ELW(S)$ and $L\in \ELW(S)$ such that $F(L_k)\to F(L)$. We want to show that $L_k\toCH L$. Consider a Hausdorff convergent subsequence $L_{k_i}\toH M$. We want to show that $L\subset M$.

Write $F(L_{k_i})=[\B(a,\vec{l_i},p_i)]=[\{a_n^i\}_{n=0}^\infty]$ and $F(L)=[\B(a,\vec{l},p)]=[\{a_n\}_{n=0}^\infty]$. For each $i$, by passing to a subsequence of $\{a_n^i\}$, we may suppose that $a_n^i$ Hausdorff converges to a lamination $L^0_i$. By Corollary \ref{bicornchconv} we have that $L_{k_i}\subset L^0_i$.

Pass to a further subsequence of $\{L_{k_i} \}$ to assume without loss of generality that $L^0_i$ Hausdorff converges to a lamination $L^0$. Since $L_{k_i}\subset L^0_i$ for all $i$, we have $M\subset L^0$.

Since $[\{a_n^i\}]\to [\{a_n\}]$ and $a_n^i\toH L^0_i$, for each $i$ we may choose $n_i$ such that $d_{\operatorname{Haus}}(a_{n_i}^i,L^0_i)<1/i$ and \begin{equation} \label{gromprod} \lim_{i\to \infty} (a_{n_i} \cdot a_{n_i}^i)_a=\infty.\end{equation} Write $b_i=a^i_{n_i}$. We have $[\{b_i\}_i]=[\{a_n\}_n]=F(L)$ by Equation \ref{gromprod}. We have by Corollary \ref{boundarychconv} that $b_i\toCH L$. Note that \[d_{\operatorname{Haus}}(b_i,L^0)\leq d_{\operatorname{Haus}}(b_i,L^0_i)+d_{\operatorname{Haus}}(L^0_i,L^0)\] and both quantities on the right hand side converge to 0 as $i\to \infty$ -- the left term by definition of $b_i$ and the right term by the fact that $L_i^0$ Hausdorff converges to $L^0$. Therefore $b_i\toH L^0$ and since $b_i\toCH L$, we have $L\subset L^0$. But we also have $M\subset L^0$ and therefore $M\pitchfork L=\emptyset$. Since $L\in \ELW(S)$, $M$ either contains $L$ or $M$ is contained in $V^C$ where $V$ is the witness filled by $L$. The latter possibility does not occur since if $M\subset V^C$ then $L_{k_i}\subset V^C$ for all large enough $i$, which contradicts that $L_{k_i}\in \ELW(S)$. Therefore we have $L\subset M$, as desired.
\end{proof}

Thus $F$ is a homeomorphism, as claimed. This completes the proof of Theorem \ref{boundarythm}.

\bibliographystyle{plain}
\bibliography{boundary}

\noindent \textbf{Alexander J. Rasmussen } Department of Mathematics, Yale University, New Haven, CT 06520. \\
E-mail: \emph{alexander.rasmussen@yale.edu}

\end{document}